\documentclass[a4paper]{amsart}

\usepackage[utf8]{inputenc}
\usepackage[T1]{fontenc}
\usepackage{lmodern}
\usepackage{microtype}
\usepackage{mathrsfs}			
\usepackage{csquotes}			
\usepackage{mathtools}			
\usepackage{enumitem}			
\usepackage[american]{babel}	
\usepackage{tikz-cd}			
\usepackage{amssymb}
\usepackage[hidelinks, colorlinks, citecolor=blue]{hyperref} 
\usepackage[noabbrev]{cleveref} 

\allowdisplaybreaks

\newlist{enumtheo}{enumerate}{1}			
\newlist{enumstep}{enumerate}{1}

\setlist[enumtheo,enumstep]{wide,label=(\roman*),font=\upshape}
\setlist[enumerate,1]{label=\alph*),font=\upshape}
\theoremstyle{plain}
\newtheorem{satz}{Satz}[section]
\newtheorem{theo}[satz]{Theorem}
\newtheorem{lemm}[satz]{Lemma}
\newtheorem{prop}[satz]{Proposition}
\newtheorem{coro}[satz]{Corollary}
\newtheorem*{thm*}{Theorem}
\newtheorem*{prop*}{Proposition}

\theoremstyle{definition}
\newtheorem{defi}[satz]{Definition}
\newtheorem{ques}[satz]{Question}

\newtheorem{exam}[satz]{Example}
\theoremstyle{remark}
\newtheorem{rema}[satz]{Remark}
\numberwithin{equation}{subsection}
\creflabelformat{equation}{#2#1#3}
\crefname{rema}{remark}{remarks}
\crefname{defi}{definition}{definitions}
\crefname{theo}{theorem}{theorems}
\crefname{prop}{proposition}{propositions}
\crefname{coro}{corollary}{corollaries}
\crefname{enumtheoi}{item}{items}
\crefname{enumstepi}{step}{steps}
\crefformat{subsection}{#2(#1)#3}
\crefname{diag}{diagram}{diagrams}
\crefname{sequ}{sequence}{sequences}

\newcommand{\G}{\mathbb{G}}				
\newcommand{\N}{\mathbb{N}}				
\newcommand{\Z}{\mathbb{Z}}				
\newcommand{\R}{\mathbb{R}}				
\newcommand{\C}{\mathbb{C}}				
\renewcommand{\P}{\mathbb{P}}			

\newcommand{\cA}{\mathcal{A}}			
\newcommand{\cC}{\mathcal{C}}			
\newcommand{\cF}{\mathcal{F}}			
\newcommand{\cG}{\mathcal{G}}
\newcommand{\cI}{\mathcal{I}}			
\newcommand{\cL}{\mathcal{L}}			
\newcommand{\cM}{\mathcal{M}}			
\newcommand{\cN}{\mathcal{N}}			
\newcommand{\cT}{\mathcal{T}}			
\newcommand{\cU}{\mathcal{U}}			
\newcommand{\cV}{\mathcal{V}}			
\newcommand{\cW}{\mathcal{W}}
\newcommand{\cX}{\mathcal{X}}			

\newcommand{\sC}{\mathscr{C}}			
\newcommand{\sF}{\mathscr{F}}			
\newcommand{\sG}{\mathscr{G}}
\newcommand{\sL}{\mathscr{L}}			
\newcommand{\sM}{\mathscr{M}}			
\newcommand{\sN}{\mathscr{N}}			
\newcommand{\sT}{\mathscr{T}}			
\newcommand{\sS}{\mathscr{S}}			
\renewcommand{\O}{\mathscr{O}}			

\newcommand{\per}{\mathsf{p}}					
\newcommand{\Per}{\mathsf{P}}					
\newcommand{\pdom}[1]{\mathsf{D}_{#1}}			
\newcommand{\Weyl}[1]{\mathsf{W}_{#1}}			
\newcommand{\Tw}{\mathsf{T}} 					
\newcommand{\Twr}{\mathsf{T}_{\R}}				

\newcommand{\id}[1]{\mathsf{id}_{#1}}			
\newcommand{\pr}[1]{\mathsf{pr}_{#1}}			

\newcommand{\pt}{\mathsf{pt}}					
\DeclareMathOperator{\df}{Def} 					
\newcommand{\CP}[1]{\mathbb{P}^{#1}}			
\newcommand{\Gr}[2]{\mathsf{Gr}(#1,#2)}			
\newcommand{\Grp}[2]{\mathsf{Gr}^+(#1,#2)}		
\newcommand{\Sec}{\mathsf{Sec}}					
\newcommand{\dou}{\operatorname{\mathsf{Dou}}}	

\newcommand{\Nb}[2]{\mathscr{N}_{#1/#2}}		
\newcommand{\csh}[2]{\underline{#1}_{#2}}		

\newcommand{\Hsh}[3]{\mathsf{H}^{#1}(#2;#3)}		
\newcommand{\hdim}[3]{\mathsf{h}^{#1}(#2;#3)}		
\newcommand{\Hdg}[3]{\mathsf{H}^{#1,#2}(#3)}		
\newcommand{\Rs}[3]{\mathsf{R}^{#1}#2_*{#3}}		

\DeclareMathOperator{\codim}{codim}					
\newcommand{\rk}{\operatorname{\mathsf{rk}}}		
\renewcommand{\deg}{\operatorname{\mathsf{deg}}}	
\newcommand{\End}{\operatorname{\mathsf{End}}}      
\newcommand{\Aut}{\operatorname{\mathsf{Aut}}}		
\newcommand{\Holono}{\operatorname{\mathsf{Hol}}}	
\renewcommand{\Re}{\operatorname{\mathsf{Re}}}		
\renewcommand{\Im}{\operatorname{\mathsf{Im}}}		
\renewcommand{\subset}{\subseteq}					
\renewcommand{\epsilon}{\varepsilon}				
\newcommand*{\defeq}{\coloneqq}						
\newcommand*{\eqdef}{\eqqcolon}						
\renewcommand{\tilde}{\widetilde}
\newcommand{\from}{\colon}
\newcommand{\isom}{\cong}							
\newcommand{\inv}{^{-1}}							
\newcommand{\lto}{\longrightarrow}					
\newcommand{\x}{\times}								
\newcommand{\inj}{\hookrightarrow}					
\newcommand{\bij}{\overset\sim\lto}					
\newcommand{\rest}[2]{\left.{#1}\right\vert_{#2}}	

\newcommand{\Sp}{\operatorname{\mathsf{Sp}}}		
\newcommand{\SL}{\operatorname{\mathsf{SL}}}		
\newcommand{\Orth}{\operatorname{\mathsf{O}}}		
\newcommand{\Orthp}{\operatorname{\mathsf{O}^+}}	
\newcommand{\SO}{\operatorname{\mathsf{SO}}}		
\newcommand{\SU}{\operatorname{\mathsf{SU}}}		
\newcommand{\PSL}{\operatorname{\mathsf{PSL}}}		

\makeatletter
\@namedef{subjclassname@2020}{\textup{2020} Mathematics Subject Classification}
\makeatother

\author{Daniel Greb}
\address{Essener Seminar für Algebraische Geometrie und Arithmetik, Fakultät für Mathematik, Universität Duisburg-Essen, 45117 Essen}
\email{daniel.greb@uni-due.de}
\urladdr{\url{https://www.esaga.uni-due.de/daniel.greb/}}
\author{Martin Schwald}
\address{Universität zu Köln, Department Mathematik/Informatik, Abteilung Mathematik, Weyertal~86-90, 50931 Köln, Germany}
\email{mschwald@uni-koeln.de}
\urladdr{\url{https://www.mi.uni-koeln.de/~mschwald}}


\begin{document}

\title[K3 families and deformed twistor cycles]{
Moduli of K3 families over $\P^1$ and complex-hyperk\"ahler metrics induced by deformed twistor cycles}
\date\today

\begin{abstract}
We answer a question posed independently by  Fels--Huckleberry--Wolf and Looijenga concerning the geometric meaning of small deformations of twistor cycles in the K3 period domain.
These are shown to induce complex-hyperk\"ahler metrics on members of the families via Penrose's Non-linear Graviton construction.
On the way to proving this result, we construct a Hausdorff fine moduli space for families of marked K3 surfaces over smooth rational curves in the K3 period domain.
Over an open subset containing all twistor cycles we construct a family of such families, which is a universal small deformation for every twistor family.
Whenever possible, we extend the results to higher-dimensional irreducible holomorphic-symplectic manifolds. 
\end{abstract}

\maketitle
\tableofcontents

\section{Introduction}

The period domain of complex K3 surfaces is a well-studied open subspace $\Omega$ of a $20$-dimensional complex projective quadric.
One of the key geometric differences between $\Omega$ and period domains of Hermitian-symmetric type is the existence of non-trivial compact submanifolds.
Most prominently, $\Omega$ contains a real $57$-dimensional family of conics defined over the reals that arise via hyperk\"ahler geometry as follows:
If $X_0$ is a complex K3-surface with underlying real four-manifold $M$ and $\alpha\in\Hsh{1,1}{X_0}{\mathbb{R}}$ is a K\"ahler class, the unique K\"ahler-Einstein metric $g_{\mathrm{KE}}$ whose associated $(1,1)$-form lies in $\alpha$ (guaranteed to exist by Yau's solution of the Calabi-conjecture) has holonomy $\mathrm{Sp}(1)= \mathrm{SU}(2)$, and is hence hyperk\"ahler.
In particular, $g_{\mathrm{KE}}$ admits a sphere $S^2= \{aI + bJ + cK \mid a^2+ b^2+ c^2 =1\}$ of compatible complex structures.
These endow $M \times S^2$ with a complex structure, yielding a \emph{twistor family} $(X\stackrel{f}{\to}\mathbb{P}^1)$ of K3-surfaces containing $X_0$ as one of its members, see Section~\ref{twistor} for details.
Any marking $\mu_0$ of $\Hsh2{X_0}{\Z}$ can be uniquely extended to a marking $\mu$ of the family, and if $p \in \Omega$ is the period point of $X_0$ with respect to $\mu_0$, the period map $\per\colon\mathbb{P}^1 \to \Omega$ embeds $\mathbb{P}^1$ as a smooth conic $C_0\subset\Omega$ defined over $\mathbb{R}$ and passing through $p$.
We call conics arising in this way \emph{twistor cycles}.
A quotient of the space of twistor cycles can be identified with a moduli space for K\"ahler-Einstein metrics, see \cite[Section~12.K]{Besse}. Moreover, these cycles were used in crucial ways in classical proofs for the surjectivity of the period map. 
We note that a similar picture also arises when one studies higher-dimensional compact irreducible holomorphic-symplectic (IHS) manifolds. 

Varying the plane spanned by $C_0$ in the complex Grassmannian or using Kodaira-Spencer theory, one observes that $C_0$ is actually a member of a complex $57$-dimensional family $\{C_t\}$, $t \in C_1(\Omega)$, of quadratically embedded smooth rational curves $C_t\subset\Omega$, which we call \emph{complex cycles}.
Arriving at the problem from different angles, Fels--Huckleberry--Wolf and Looijenga independently asked what happens geometrically when one deforms $C_0$ in the imaginary direction, i.e., away from the locus of cycles defined over the reals (which we call \emph{real cycles}). 

\begin{ques}[\protect{\cite[p.~236, Remark]{FHW}}]
\label{geomques}
Let $C$ be a complex cycle obtained as a small deformation of a twistor cycle.
Do the K3 surfaces parametrized by $C$ share a geometric structure? 
\end{ques}
More ambitiously, one can ask as Looijenga \cite[Question~3.1]{Loo21} whether $C_1(\Omega)$ can be interpreted as a moduli space of geometric structures on the real four-manifold $M$ underlying every complex K3 surface. 

Our goal here is to  give a detailed answer to \Cref{geomques} and furthermore make some significant steps in the direction of Looijenga's question. As intermediate steps  arising in our approach, we clarify fundamental properties of the Douady spaces of period domains of IHS manifolds, we investigate the deformation theory of families of IHS manifolds over smooth rational curves inside the period domain and establish a moduli theory for them. We then restrict to the K3 case, where we relate the  moduli space of K3 families constructed rather abstractly in the first step to a concrete sheaf-theoretic construction involving the Burns-Rapoport space. Additionally, we exhibit a family of families of K3 surfaces that is a universal small deformation for every K3 twistor family. We consider these intermediate results to be of independent interest. 

The geometric starting point is the observation that many pathological properties of $\Omega$ disappear when we pass to the parameter space $C_1(\Omega)$ for complex cycles inside $\Omega$.
Hence, it is key to understand in which way the good topological and complex-analytic properties of $C_1(\Omega)$ (and more generally of the entire space of smooth compact complex curves in $\Omega$) are reflected in the moduli theory for (marked) families of K3 surfaces.
We find that this moduli theory shares many features with the moduli theory of (marked) K3 surfaces itself, with $C_1(\Omega)$ playing the role of a period domain for the deformation problem at hand.

\subsection{Existence results for marked families of K3 surfaces}
When studying the \emph{Douady space}, i.e., the space of compact complex subspaces of a given complex manifold, the first question to answer is which non-empty connected components it has.
It is known classically that period domains of IHS manifolds do not contain compact complex subspaces of dimension greater than one, see e.g.~\cite[Proposition 3.4]{Deev} for one argument.
On the other hand, if $C$ is a compact Riemann surface of genus $g$, then for any $d > g$ there exists an embedding of $C$ into the period domain of degree $2d$, see \Cref{Deev} below.
Hence, there are many complex curves inside the period domain.
One of the main features of the theory of K3 surfaces is that the moduli space $\tilde{\Omega}$ of marked K3s is a \emph{non-Hausdorff} sheaf over $\Omega$.
In particular, while the period map from $\tilde{\Omega}$ to $\Omega$ is surjective, there is no universal family over $\Omega$, which one could restrict to curves inside $\Omega$.
Nevertheless, we are able to show the following Lifting Theorem.

\begin{thm*}[\ref{liftfamily}]
Let $C$ be a compact Riemann surface, and let $\per\from C\to\Omega$ be holomorphic. Then, for every $p\in C$ and a marked K3 surface $(X,\mu_X)$ with period point $\per(p)$ there exists a marked family $(\cX,\mu)$ of K3 surfaces over $C$ with associated period map $\per$ and $(\cX_p,\mu_p)\cong(X,\mu_X)$.
\end{thm*}
As the proof uses the Global Torelli Theorem, which is more subtle for higher-dimensional IHS manifolds than for K3 surfaces, we currently do not know to what extent \Cref{liftfamily} generalizes to higher dimensions.

\subsection{Moduli of families of K3 surfaces and IHS manifolds}
\label{introSect_moduli}
While we emphasize that \Cref{liftfamily} also holds for curves of genus $g\ge1$, motivated by \Cref{geomques} we then focus on the case $g=0$.
From \Cref{liftfamily} it is a natural next step to study marked families of K3 surfaces over smooth rational curves $C\subset\Omega$ with period map the inclusion; we call such marked families \emph{embedded}.
As discussed above, we want to establish a moduli theory for such families.
It will turn out that the corresponding moduli space is a sheaf over $C_1(\Omega)$, cf.~Looijenga's question. 

To take some hints, let us look at the classical theory for K3 surfaces themselves. There are two classical constructions of the moduli space $\tilde{\Omega}$ of marked K3 surfaces. 
The approach of Burns--Rapoport \cite{BR75} gives an explicit description of $\tilde{\Omega}$ in terms of K3 lattice theory (see Section~\ref{K3moduli}).
A bit more abstractly, Beauville's approach \cite{Bea81} relies on three crucial properties of K3 surfaces: unobstructedness, the Local Torelli Theorem and faithfulness of the representation of the automorphism group on the second cohomology.

\subsubsection{Second level period maps}
The a priori problem in constructing our moduli theory in a way similar to the second approach is that the total spaces $X$ of embedded families are in general not K\"ahler, see \cite{HitchinNonKaehler} (and \cite{DMN,Cam91} for more general statements in higher dimensions).
Therefore, one cannot use period maps and period domains for these total spaces directly in an attempt to generalize Beauville's construction of $\tilde{\Omega}$.

However, as $X$ is fibered into K3 surfaces, we are nevertheless able to construct holomorphic maps from bases of deformations of marked embedded families to the parameter space of smooth rational curves in $\Omega$.
We call this the \emph{second level period map} (\Cref{periodmap2}). 
In fact, many of the fundamental properties classically known for the period map for K3 surfaces, see for example \cite[Chapter~VIII]{BHPV}, have analogs for the second level period map.

First of all, \Cref{liftfamily} above can be seen as an analog of the \emph{surjectivity of the period map}.
Secondly, Brecan--Kirschner--Schwald proved unobstructedness results for total spaces of IHS families over rational curves, see \cite{BKS}.
We refine these slightly to obtain an analog of the \emph{Local Torelli Theorem} for the second level period map of marked embedded families (\Cref{local torelli}):
A deformation over $B$ of a marked embedded family of K3 surfaces is universal at a point $b\in B$ if and only if its second level period map is a local biholomorphism at $b$.
In particular, near a point corresponding to a marked twistor family $(X\stackrel{f}{\to}\P^1)$, the cycle space is isomorphic to the Kuranishi space of the total space $X$, as well as to the deformation space of the morphism $f$ (\Cref{deformation space}).
These results generalize to higher-dimensional IHS manifolds. 

\subsubsection{Fine moduli spaces for (total spaces of) marked second level families }
The aforementioned results establish the first two necessary pieces to mimic Beauville's construction. As the last piece, we proceed with proving an appropriate \emph{faithfulness result} for automorphism groups of total spaces of embedded families (\Cref{automorphisms}):
If $(X\stackrel{f}{\to}C)$ is an embedded family of K3 surfaces, then $\id X$ is the only biholomorphic automorphism of the total space $X$ that induces the identity on $\Hsh2X{\Z}$.

Then we are in the position to construct a fine\footnote{See for example \cite[Section~23]{Har10} for a discussion of \emph{fine} and \emph{coarse} moduli spaces in the algebraic category, which can be transferred to the complex-analytic category \emph{mutatis mutandis}.} moduli space $\sM$ for marked embedded families of K3 surfaces (\Cref{fine}) by gluing their local universal deformation spaces over the Douady space using second level period maps. 

Additionally, in Section~\ref{section sheaves} we give a second construction of $\sM$, as follows:
if one pulls back the Burns-Rapoport space $\tilde \Omega$ (considered as a sheaf) from $\Omega$ to the total space of the universal family of smooth rational curves in $\Omega$ via the evaluation map and then pushes down to the parameter space of these curves, one obtains a sheaf $\sF$ whose \'espace \'etal\'e we show to be isomorphic to $\sM$.
With the help of the lattice-theoretic description of $\tilde{\Omega}$ this yields an explicit description of $\sM$, which in turn allows us to show that---in contrast to $\tilde{\Omega}$---the moduli space $\sM$ is in fact Hausdorff (\Cref{MT2}).
Let us summarize these fundamental technical results as follows.

\begin{thm*}[cf.~Theorems \ref{fine}, \ref{total spaces}, and \ref{MT2}]
There exists a fine moduli space $\sM$ of marked embedded families of K3 surfaces. This moduli space is Hausdorff and admits a locally biholomorphic map to the space of smooth rational curves in $\Omega$.
It is also a fine moduli space for total spaces $X$ of embedded families of K3 surfaces together with markings of $\Hsh2X{\Z}$.
\end{thm*}

In higher dimensions, automorphism groups of IHS manifolds do generally not act faithfully on the second cohomology. As a consequence, our construction then only yields coarse moduli spaces, see Section~\ref{glueing} for details.
Currently, we do not have much evidence on whether these coarse moduli spaces are also Hausdorff.

\subsubsection{Coarse moduli spaces for (total spaces of) embedded families}
While moduli spaces of marked objects are interesting in their own right, the initial motivation for introducing markings is to be able to use period maps but eventually factor out the additional marking data to obtain a (coarse) moduli space for isomorphism classes of the considered objects themselves.
In the case of (marked) K3 surfaces, one is not able to carry out this second step in a satisfactory manner, as the discrete group $\Gamma = \Orthp(\Lambda)\subset\Orth(\Lambda)$ that realizes changes of markings acts wildly on the period domain; in particular, the quotient topology on $\Omega / \Gamma \simeq \tilde\Omega/\Gamma$ is very non-Hausdorff and a coarse moduli space of K3 surfaces does not exist in the category of complex spaces. 

Again, passing to the parameter space $C_1(\Omega)$ of complex cycles eliminates this problem:
the induced action of $\Gamma$ on $C_1(\Omega)$ is properly discontinuous, see \cite[Chapter 11]{FHW}.
As a quotient by a properly discontinuous action on a complex manifold, $C_1(\Omega)/\Gamma$ is a Hausdorff orbifold.
From this fundamental fact and the Hausdorff property of $\sM$ we obtain the following result. 

\begin{thm*}[\ref{thm:coarsemoduli}]
Let $\sM^\circ$ be the part of $\sM$ lying over the component $C_1(\Omega)$ containing twistor cycles.
Then, the group $\Gamma$ acts properly discontinuously on $\sM^\circ$ and the complex space $\sM^\circ/\Gamma$ is a coarse moduli space both for embedded families of K3 surfaces over cycles in $C_1(\Omega)$ and for their total spaces. 
\end{thm*}

\subsection{The geometric meaning of complex cycles}
\label{answerquest}
Once the results concerning deformation and moduli theory described above are established, we turn to \Cref{geomques}.
After the previous discussion, it is natural to interpret it as a question about marked embedded families obtained as small deformations of twistor families.
While in search for a first answer one could look at small deformations of a single twistor family. However, in the moduli-theoretic spirit of the previous results and motivated by Looijenga's question \cite[Question 3.1]{Loo21}, we work as globally as possible and consider all twistor cycles and twistor families at once.

\subsubsection{A family containing all deformations of twistor spaces}
First we give a new construction of a real-analytic family consisting of all K3 twistor families (\Cref{realfamily}), which is based on our construction and investigation of $\sM$ outlined above.
This family was originally described by Kobayashi~\cite[Theorem 10]{Kob90} using differential-geometric methods that build on work of Morrison, Looijenga, and Todorov; see \emph{loc.~cit.} for the precise references.
From our point of view, it can be seen as a section in the sheaf $\sF$ corresponding to $\sM$ over the (totally real) subset  of $C_1(\Omega)$ that parametrizes twistor cycles. 

Using the sheaf-theoretic interpretation, we extend this family to a family $\cT_U$ over a $\Gamma$-stable open subspace of the cycle space $C_1(\Omega)$ (\Cref{extsection}).
The results discussed in Section~\ref{introSect_moduli} above allow us to show that $\cT_U$ is a \emph{universal small deformation of marked K3 twistor spaces} (\Cref{defT}).
In other words, $\cT_U$ contains all families of K3 surfaces that \Cref{geomques} applies to. 

\subsubsection{The Nonlinear Graviton Construction}
The final idea for our answer of \Cref{geomques} is to apply a relative version of \emph{Penrose's Nonlinear Graviton Construction} \cite{MR0439004}. 
The absolute version of this construction relevant for us\footnote{Discussed for example in \cite[Theorem~9.3.3 and 10.5.5]{Dun10}.} says that out of 1) a smooth family $(X\stackrel{f}{\to}\mathbb{P}^1)$ of (compact) complex surfaces over $\mathbb{P}^1$, 2) a twisted $f$-relative holomorphic $2$-form $X$, and 3) a section $\sigma\colon \mathbb{P}^1 \to X$ whose image has normal bundle $\O_{\mathbb{P}^1} \oplus \O_{\mathbb{P}^1}$ one produces a holomorphic-Riemannian metric\footnote{E.g.~in the sense of \cite{Leb83}.} on an open subset $S$ of the four-dimensional Douady space of sections of $f$, which contains the point corresponding to the section $\sigma$.
Moreover, this metric is Ricci-flat (this was Penrose's original motivation) and more precisely has complex holonomy contained in $\SL(2, \C)$ (which should be considered as a complex-symplectic group).
We call such holomorphic metrics \emph{complex-hyperk\"ahler} and refer the reader to Section \ref{subsubsect_HK} for further points of view on this construction and for the relevant references. 

Given a totally real submanifold $M$ of $S$, we may restrict the holomorphic-Riemannian metric on $X$ to what we call a \emph{complex-hyperk\"ahler metric} on $M$, see \Cref{CHK}. In our setup, the resulting geometric architecture on $M$ in particular consists of a sphere $\{I_t\}_{t \in S^2}$ of complex structures on $M$; each $(M, I_t)$ is hence a complex surface.\footnote{Note the relation to the work of Bielawski--Schwachh\"ofer \cite{MR3085661} on pluricomplex geometry.} 
In case that the totally real embedding $M \hookrightarrow S$ comes from considering the real sections in a family $f$ in which all the data 1), 2), 3) carry a compatible real structure, this leads to the construction of a real (indefinite) hyperk\"ahler metric on $M$; this is the well-known twistor construction in hyperk\"ahler geometry, see \cite{HKLR}.

\subsubsection{The answer to the question of Fels--Huckleberry--Wolf}
Coming back to our specific situation, we show that the universal small deformation $\cT_U$ of marked K3 twistor spaces constructed above carries a relative holomorphic-symplectic form (\Cref{sympl}).
Moreover, we show that the sections of the twistor fibrations it contains (so-called \emph{twistor lines}) deform freely inside the total space of $\cT_U$ (\Cref{deform twistor lines}); more precisely, the spaces of sections of the marked embedded families contained in $\cT_U$ form a smooth family $p\from S\to U$ over $U$ as well. In summary, the data 1), 2), and 3) deform together with $\cT_U$ over $U$.
A small technical issue arises when one now wants to perform a \emph{relative version} of the Nonlinear Graviton Construction, which can however be solved at the small cost of discussing metrics up to pointwise scalings; i.e.~one looks at \emph{conformal equivalence classes} of metrics. As being hyperk\"ahler is not conformally invariant, this leads to the classical notion of anti-self dual conformal structures, using which we can finally formulate our precise answer to \Cref{geomques} as follows.

\begin{thm*}[\ref{conformalStructures}]
Let $p\from S\to U$ be as introduced above and denote by $M$ the real 4-manifold underlying every K3 surface.
Then, up to shrinking $U\subset C_1(\Omega)$ around the space of twistor cycles, there is a canonical $\Gamma$-equivariant family of anti-selfdual holomorphic conformal structures on $p\colon S \to U$ such that for every $t = [C] \in U$ the induced conformal equivalence class of metrics on the fiber $S_t$ contains a complex-hyperk\"ahler representative that induces exactly those complex structures on $M$ that correspond to the points $p \in C \subset \Omega$ via the Global Torelli Theorem. Over every twistor cycle this complex-hyperk\"ahler metric is the complexification of the K\"ahler-Einstein metric coming from Yau's solution of the Calabi conjecture.
\end{thm*}

As a consequence, given a complex cycle $C$ obtained by a small deformation of a twistor cycle, the K3 surfaces parametrized by $C$ share a complex-hyperk\"ahler metric, see \Cref{cor:commonHKmetric}. We emphasize that for cycles that are not defined over $\mathbb{R}$ the complex-hyperk\"ahler structure found above is not the complexification of a K\"ahler-Einstein metric but a ``truly complex'' object.

\subsection{Acknowledgments}
We thank Ana-Maria Brecan and Tim Kirschner for participating in the first part of the project that led up to \cite{BKS}, which serves as a foundation and motivation for our paper.
We thank Tim Kirschner for additionally contributing ideas for the proof of the Lifting Theorem in the case where the base is $\P^1$ and for the Hausdorff property of the components of the moduli space $\sM$ over $C_1(\Omega)$, and Ana-Maria Brecan for mathematical input related to \Cref{ExampleLooijenga}. 

Daniel Greb wants to thank Daniel Huybrechts and Keiji Oguiso for patiently answering his questions regarding automorphism groups of K3 surfaces, Henri Guenancia for repeated guidance concerning analytic aspects, relevant especially for the proof of \Cref{sympl}, Eyal Markman for several informative discussions during the ``Hyperkahler quotients, singularities, and quivers'' workshop at the Simons Center, Roger Bielawski for help with finding modern references for Penrose's Non-linear Graviton Construction as well as for interesting discussions about pluricomplex structures, and Andrea Iannuzzi for discussions on crown domains of Riemannian symmetric spaces that helped him correct some of his misconceptions about the cycle space of the K3 period domain.
Finally, he wants to thank John Rawnsley for introducing him to twistor theory 20 years ago, Alan Huckleberry for pointing out that the connection of cycles spaces of flag domains to moduli theory has been largely unexplored, and Andrew Dancer for mentioning Penrose's work during a conversation that took place at FU Berlin in September 2015.

The authors want to thank the referees for their constructive criticism and their suggestions towards improving the exposition.

Both authors gratefully acknowledge support by the DFG-Research Training Group 2553 ``Symmetries and classifying spaces -- Analytic, arithmetic, and derived''.
DG was furthermore partially supported by the ANR-DFG Project ``QuaSiDy -- Quantization, Singularities, and Holomorphic Dynamics''.

\part{Preliminaries}

We collect preliminary material to fix notation and conventions and for later reference. The content of Section~\ref{sect1} is almost entirely known, Section~\ref{cycle section} contains some new examples and observations on cycle spaces of IHS period domains.

\section{Complex-analytic and complex-Riemannian geometry, IHS manifolds}\label{sect1}

\subsection{Complex spaces and complexifications}
\label{section complex spaces}

We use the basic definitions regarding Complex Geometry as they are introduced in \cite{Fischer}.
Manifolds and complex spaces are assumed to be Hausdorff and second countable unless mentioned otherwise.
Partly, we will also have to work with certain non-Hausdorff versions of complex spaces; a reference for the basic theory for such spaces is \cite{GrothendieckAnalyticSpaces}.
For the reader's convenience and to fix a consistent terminology, we recall the basic notions of families, deformations, Douady spaces and topological sheaves.

\subsubsection{Families of compact complex manifolds}
\label{families}
If $X$ and $S$ are complex spaces and $f \from X\to S$ is a proper holomorphic submersion in the sense of \cite[2.18]{Fischer}, we call $\cF\defeq (X\stackrel{f}{\to}S)$ a \emph{family of compact complex manifolds}.
We denote the fibers of $f$ as $X_s\defeq f\inv(s)$ for $s\in S$. 
Let $\cF = (X\stackrel{f}{\to}S)$ and $\cF' = (X'\stackrel{f'}{\to}S')$ be two families of compact complex manifolds, $g\from X\to X'$ and $h\from S\to S'$ holomorphic maps with $f'\circ g=h\circ f$.
We call $\phi=(g,h)\from\cF\to\cF'$ a \emph{morphism of families} if the corresponding diagram is a Cartesian square, that is, a pullback diagram.
In particular we only consider morphisms that induce fiberwise isomorphisms.
The Cartesian property is automatic if $g$ and $h$ are isomorphisms; in that case $\phi$ is an isomorphism of families.

\subsubsection{Deformation theory of complex manifolds}\label{deformation sec}
A \emph{deformation} of a complex manifold $X$ is a flat, proper holomorphic map $f\from Z\to S$ of complex spaces together with an isomorphism of the \emph{central fiber} to $X$. A \emph{small deformation} is a deformation over a complex space germ $(S,0)$.
In our terminology, regarding $X$ as a family over a reduced point $\pt$, we view a deformation as a morphism $(X\to\pt)\to(Z\stackrel{f}{\to}S)$ of families, such that $\pt$ gets mapped to a distinguished point $0\in S$ inducing an isomorphism $X\isom Z_0$;
in the case of small deformations we consider the family $(Z\stackrel{f}{\to}S)$ only up to shrinking the base $S$ around $0$.
\emph{Complete} and \emph{(semi-)universal deformations} are defined as usual, \cite[Section~2.9]{BKS}.
Every compact complex manifold admits a semi-universal small deformation, the \emph{Kuranishi-family}, whose base space is denoted by $\df(X)$.
If $\df(X)$ is smooth at $0$, we say that $X$ \emph{has unobstructed deformations}. 

\subsubsection{Relative Douady spaces}\label{douady section}
Douady constructed for every complex space $X$ a complex space $\dou(X)$ parametrizing the compact complex subspaces of $X$, \cite[\S9, Th\'{e}or\`{e}me~1]{Douady}.
Pourcin generalized Douady's construction to a relative setting \cite[Th\'{e}or\`{e}me~2]{Pou69}, which can be spelled out as follows:

Let $X\to B$ be a holomorphic map of complex spaces, then the \emph{relative Douady space} is a complex space $\dou(X/B)$ over $B$ for which there exists a closed complex subspace $Z\subset\dou(X/B)\x_B X$ that is flat and proper over $\dou(X/B)$.
Then $(Z\stackrel{\pr1}{\to}\dou(X/B))$ is called the \emph{universal family of compact complex subspaces of $X$ over $B$}, and is characterized by being the terminal object in the category of flat and proper morphisms $(Y\stackrel{\pr1}{\to}S)$, where $S$ is a complex space over $B$ and $Y\subset S\x_B X$ a closed subspace with projection $\pr1$ onto $S$.
That is, for each such $(Y\stackrel{\pr1}{\to}S)$ there is a unique holomorphic map $\alpha\from S\to\dou(X/B)$ over $B$ that together with $\rest{(\alpha\x\id X)}{Y}\from Y\to Z$ and the two projections $\pr1$ from $Y$ and $Z$ forms a Cartesian diagram.

If $B$ is just a point, we write $\dou(X/B)=\dou(X)$.
In general, the universal property induces a canonical embedding $\dou(X/B)\inj\dou(X)$.
On the other hand, for each $b\in B$ and every compact complex subspace $C\subset X_b$ its family over a point defines a unique point in $\dou(X_b)\subset\dou(X/B)$.
Denoting this point by $[C]$, the fiber of $Z\to\dou(X/B)$ over $[C]$ is $\{[C]\}\x C$.

In our applications, we will always restrict $Z\to\dou(X/B)$ over open subspaces $U\subset\dou(X/B)$ such that all fibers are smooth.
Such a restriction is then a family of compact complex manifolds over $U$ in the sense of Section~\ref{families}.

\subsubsection{Topological sheaves}
\label{topsheaves}

Using the terminology of \cite[p.~223--227]{GR84}, if $\pi\from E\to X$ is a local homeomorphism of topological spaces, we call $\sF\defeq (E\stackrel{\pi}{\to}X)$ a \emph{topological sheaf} on $X$ (also known as \emph{\'etal\'e space} over $X$); see also \cite[Chapter~3]{Wed16}.
For every subspace $Z\subset X$ we let $\sF(Z)$ be the set of continuous sections of $\rest{\pi}{\pi\inv(Z)}$.
As $\pi$ is a local homeomorphism, the assignment $U\mapsto\sF(U)$ for open subsets $U\subset X$ defines a sheaf of sets on $X$.
Vice versa, for a sheaf of sets $\sF$ on $X$, the set $\dot\bigcup_{x\in X}\sF_x$ of stalks carries a natural structure as a topological sheaf on $X$.
These constructions define inverse functorial isomorphisms between the categories of topological sheaves and sheaves of sets on $X$.
If $f\from Y\to X$ is a continuous map of topological spaces, then the inverse image sheaf of $\sF$ corresponds to the pulled back topological sheaf $f\inv\sF=(Y\x_X E\stackrel{\pr1}{\to} Y)$.
For a topological sheaf $\sG$ on $Y$, the topological structure of the direct image sheaf $f_*\sG$ is more complicated to describe, so we refer the reader to the above-cited literature.

If $E$ is a Hausdorff topological space, $\sF$ is called a \emph{Hausdorff sheaf}.
For two sections $\sigma_1,\sigma_2\in\sF(Z)$ the set $\{x\in X\mid\sigma_1(x)=\sigma_2(x)\}\subset X$ is always open, as $\pi$ is a local homeomorphism; if $\sF$ is a Hausdorff sheaf, it is also closed.
If $G$ is a group acting equivariantly on $E\stackrel{\pi}{\lto}X$, we can consider $\sF$ as a \emph{$G$-equivariant sheaf} via the maps $\sF(U)\mapsto\sF(g(U))$, $\sigma\mapsto\left(g(\sigma)\from x\mapsto g(\sigma(g\inv(x))\right)$ for each subspace $U\subset X$ and $g\in G$.

In our applications, $X$ and $Y$ are always complex spaces, which are in particular paracompact. As a consequence, the sections of $\sF$ over a closed subspace $Z\subset X$ can be calculated as the direct limit $\sF(Z)=\varinjlim_U\sF(Z)$ over all open neighborhoods $U$ of $Z$ in $X$, \cite[Proposition~II.2.5.1(iii)]{KS94}; every section $\sigma\in\sF(Z)$ can thus be extended to an open neighborhood $U$ of $Z$ in $X$.
When $f$ is proper, this implies $(f_*\sG)(Z)=\sG(f\inv(Z))$; every section $\sigma\in (f_*\sG)(Z)$ can be extended to an $f$-saturated open neighborhood of the fiber $f\inv(Z)$.
Moreover, note that $\pi\from E\to X$ induces a unique complex structure on the total space $E$ of $\sF$ such that $\pi$ is \'{e}tale and all sections of $\sF$ over complex subspaces of $X$ are holomorphic.

\subsubsection{Complexification of real-analytic manifolds}
\label{section complexification}

Let $X$ be a complex manifold with (integrable almost) complex structure $I$.
We call a real-analytic submanifold $M\subset X$ a \emph{totally real submanifold} of $X$ if $I(T_mM)\cap T_mM=\{0\}$ in $T_mX$ for every $m\in M$.
We call a real-analytic embedding $\iota\from M\inj X$ of a real-analytic manifold a \emph{totally real embedding} if $\iota(M)$ is a totally real submanifold of $X$.
If $X$ carries an antiholomorphic involution $\tau$ and $M$ is real-analytically isomorphic to the non-empty fixed locus $X^{\tau}$, we call $X$ a \emph{complexification} of $M$, cf.~\cite{Kul78}.

In this context, for a complex manifold $X$ with complex structure $I$, we note that antiholomorphic involutions of $X$ can be considered as biholomorphisms $\tau\from X \to \overline{X}$, where $\overline{X}$ is $X$ considered with the \emph{conjugate complex structure} $-I$.
In other words, if $\O_X \subset \sC^\infty_X$ is the subsheaf of holomorphic functions inside the sheaf of smooth $\C$-valued functions, then $\O_{\overline{X}} = \overline{\O_X} \subset \sC^\infty_X$. This reformulation can be used to introduce conjugate complex structures $\overline{X}$ and discuss antiholomorphic involutions on complex spaces, see \cite[Chapter~I, \S4]{GMT}.

\begin{rema}[Local coordinates of complexifications]
\label{localcoordinates}
For now, let $M$ be an $n$-dimensional real-analytic submanifold of an $n$-dimensional complex manifold $X$.
It follows from the real-analytic implicit function theorem that $M$ is a totally real submanifold if and only if for every $m\in M$ there exists a biholomorphism $\phi$ between an open neighborhood $U$ of $m$ in $X$ and an open subset $V\subset\C^n$ such that $\phi(U\cap M)=V\cap\R^n$. In this case, for varying $m\in M$, the complex conjugations on the $V\cap\C^n$ glue together to an antiholomorphic involution $\tau$ of an open neighborhood $X'$ of $M$ in $X$, making $X'$ a complexification of $M$. In other words, $M$ is a totally real $n$-dimensional submanifold if and only if an open neighborhood of $M$ in $X$ is a complexification of $M$.
On the other hand, let $X$ be an $n$-dimensional complex manifold with an anti\-holomorphic involution $\tau$ having a submanifold $M\defeq X^{\tau}\neq\emptyset$ as fixed locus, then $M$ is totally real.
If $M$ is $n$-dimensional, the complex charts $\phi$ discussed above can be chosen such that in local complex coordinates $z_j=x_j+iy_j$ we have $M=\{y_1=\ldots=y_n=0\}$ and $\tau$ looks like complex conjugation in $\C^n$. 
\end{rema}

Bruhat and Whitney proved in \cite[Proposition~1]{BW} that if $M$ is a real-analytic manifold, then there exists a complexification $\iota\from M\inj X$ of $M$ with the following universal property: If $\iota'\from M\inj X'$ is another complexification of $M$, then there exists a biholomorphism between open neighborhoods of $\iota(M)$ and $\iota'(M)$ extending the map $\iota'\circ\iota\inv\from\iota(M)\to\iota'(M)$. More details for the construction were provided by Shutrick in \cite{MR0098846}.


\begin{lemm}
\label{douinvolution}
Let $X$ be a complex manifold with an antiholomorphic involution $\tau$.
Then $\tau$ induces an equivariant antiholomorphic involution on the universal family $Z\to\dou(X)$ of compact complex subspaces of $X$.
\end{lemm}
\begin{proof}
We consider the antiholomorphic involution as a biholomorphism $\tau\from X \to \overline{X}$, which we extend to a biholomorphism
$$\hat\tau\defeq \id{\dou(X)}\x\tau\from\dou(X)\x X\to\dou(X)\x\overline{X}$$ over $\dou(X)$.
Then $\hat{\tau}$ maps the universal family $Z\subset \dou(X)\x X$ to a closed complex subspace of $\dou(X)\x \overline{X}$; in particular $\hat{\tau}(Z)\to\dou(X)$ is flat and proper.
Similarly, the complex conjugate $\overline{Z}\subset\overline{\dou(X)}\x \overline{X}$ is flat and proper over $\overline{\dou(X)}$.
Denoting the universal family of compact complex subspaces of $\overline{X}$ by $Y \subset \dou(\overline{X})\x\overline{X}$, the universal property of the Douady space $\dou(\overline{X})$ (see Section~\ref{douady section}) induces morphisms
\begin{align*}
(\alpha\x\id{\overline{X}},\alpha)\from \left(\hat{\tau}(Z)\to\dou(X)\right) &\to \left(Y\to\dou(\overline{X})\right)\\
(\beta\x\id{\overline X},\beta)\from\left(\overline{Z}\to\overline{\dou(X)}\right)&\to\left(Y\to\dou(\overline{X})\right)
\end{align*}
of families of compact complex subspaces of $\overline{X}$. 
Clearly, these are isomorphisms, as we can argue with $X$ and $\overline{X}$ interchanged.
The equivariant antiholomorphic involution we are looking for is the composition $((\beta\inv\circ\alpha)\x\id{\overline{X}},\beta\inv\circ\alpha)\circ(\hat{\tau},\id{\dou(X)})$.
\end{proof}

\subsection{IHS manifolds, K3 surfaces and their period domains}
\label{section ihs}

\subsubsection{IHS manifolds}

\label{ihs}

An \emph{irreducible holomorphic-symplectic manifold}, abbreviated as \emph{IHS manifold}, is a simply connected compact K\"ahler manifold $X$ with $\Hsh0X{\Omega^2_X} = \C\omega$, where $\omega$ is nowhere degenerate, cf.~\cite[p.~763--764]{Bea83}.
We call $\omega$ a holomorphic-symplectic form on $X$. 
The Beauville-Bogomolov form $\langle\cdot,\cdot\rangle$ makes $\Hsh2X{\Z}$ an even unimodular lattice of rank $(3,n_-)$. We use the term \emph{IHS lattice} for lattices appearing this way.

\subsubsection{Markings, period domains and period maps}
A ($\Lambda$-)\emph{marking} of an IHS manifold $X$ is a lattice isomorphism $\mu \from \Hsh2X\Z \to \Lambda$ to an IHS lattice $\Lambda$.
A ($\Lambda$-)\-\emph{mar\-king of a family $(X\stackrel{f}{\to}S)$ of IHS manifolds} is an isomorphism $\mu \colon \Rs2f{\csh\Z X} \to \csh\Lambda S$ of sheaves such that over each point $s \in S$ the induced map $\mu_s \colon \Hsh2{X_s}\Z \to \Lambda$ is a marking of the fiber $X_s$.
Such a marking exists if and only if the locally constant sheaf $\Rs2f{\csh\Z X}$ is constant, for example if $S$ is simply connected.
A \emph{morphism of $\Lambda$-marked families of IHS manifolds} $(X\stackrel{f}{\to}S,\mu)\to(X'\stackrel{f'}{\to}S',\mu')$ is a morphism $(g,h)$ of families such that we get a commutative diagram
\begin{center}
\begin{tikzcd}
h^{-1}(\Rs2{f'}{\csh\Z {X'}}) \rar{} \dar{h^{-1}(\mu')} & \Rs2{f}{\csh\Z{X}} \dar{\mu} \\
h^{-1}(\csh\Lambda {S'}) \rar{} & \csh\Lambda{S}
\end{tikzcd}
\end{center}
via the topological base change maps induced by $g$ and $h$; see also \cite[\S10]{LP80}. 
Consequently, this definition of marked families and morphisms thereof is compatible with the construction of pullbacks of (marked) families, \cite[Sections~2.2, 2.5, 2.6]{BKS}.

As usual, a $\Lambda$-marking $\mu$ of a family $(X\stackrel{f}{\to}S)$ of IHS manifolds induces a holomorphic map $\per\from S\to\pdom{\Lambda}$, $s\mapsto(\mu_s)_{\C}(\Hdg20{X_s})$, called the \emph{period map}.
Here $\pdom{\Lambda}$ denotes the corresponding \emph{period domain}, which is an open subspace $\pdom{\Lambda}\defeq\{x\in Q_{\Lambda}\mid \langle x,\bar{x}\rangle_{\C}>0\}$ of the smooth projective quadric $Q_{\Lambda}\defeq\{[x]\in\P(\Lambda_{\C})\mid \langle x,x\rangle_{\C}=0\}$.
The \emph{period point} $\mu_{\C}(\Hdg20X)\in \pdom{\Lambda}$ of a marked IHS manifold $(X,\mu)$ determines the Hodge structure on $\Hsh2X{\Z}$.
From this one obtains an isomorphism $\pdom{\Lambda}\cong\Orth(3,n_-)/(\SO(2)\x \Orth(1,n_-))$ and concludes that $\pdom{\Lambda}$ is a connected and simply connected homogeneous space, \cite[Corollary~25.9]{GHJ}.
The period map is functorial with respect to pullbacks by morphisms of marked families of IHS manifolds, \cite[Section~2.8]{BKS}.

\subsubsection{Abstract K3 lattice theory}
\label{abstract setting}

We quickly review the basic notions from \cite{BR75}, \cite{LP80}, \cite[p.~265--269]{Kob90} and \cite[Chapter~VIII]{BHPV} to have a consistent notation.
We fix a \emph{K3 lattice}, an even unimodular lattice $(\Lambda,\langle\cdot,\cdot\rangle)$ of signature $(3,19)$, which is unique up to isomorphism.
Then $$\Omega\defeq \pdom{\Lambda}=\{[x]\in\P(\Lambda_{\C})\mid\langle x,x\rangle_\C=0,\langle x,\bar{x}\rangle_\C>0\}$$ is the \emph{K3 period domain}.

Let $\Delta\defeq \{\delta\in \Lambda\mid\langle\delta,\delta\rangle=-2\}$, then each $\delta\in\Delta$ induces an \emph{(abstract) Picard--Lefschetz reflection} $s_{\delta}\from x\mapsto x+\langle x,\delta\rangle \delta$, which is a lattice automorphism of $\Lambda$.
Its extension to $\Lambda_{\R}$ is the reflection at the real hyperplane $\delta^{\perp}\defeq \{x\in\Lambda_{\R}\mid \langle x,\delta\rangle_{\R}=0\}$ orthogonal to $\delta$.

We continuously choose for every point $p=[x]\in\Omega$ one of the two connected components $$V_p\subset\{\kappa\in\Lambda_{\R}\mid\langle\kappa,x\rangle_{\C}=0, \langle\kappa,\kappa\rangle_{\R}>0\},$$ which we then call the \emph{(abstract) positive cone at $p$}.
This is possible as $\Omega$ is simply connected.
Denoting $\Delta_p\defeq \{\delta\in\Delta\mid\langle\delta,x\rangle_{\C}=0\}$, the connected components of $$V^{\circ}_p\defeq \{\kappa\in V_p\mid\forall\delta\in\Delta_p\colon\langle\kappa,\delta\rangle_{\R}\neq0\}$$ are called the \emph{(abstract) Weyl chambers} of $V_p$.
Each chamber $K\in\pi_0(V_p^{\circ})$ induces a partition of $\Delta_p$ into
$$\Delta_p^+\defeq\{\delta\in\Delta_p\mid\forall\kappa\in K\colon\langle\kappa,\delta\rangle_{\R}>0\}\text{ and }\Delta_p^-\defeq\{\delta\in\Delta_p\mid\forall\kappa\in K\colon\langle\kappa,\delta\rangle_{\R}<0\}$$
with the property
\begin{equation}
\label{partition}
\forall\delta_1,\ldots\delta_k\in\Delta^+_p,\forall n_1,\ldots, n_k\in\N\colon\text{ if $\delta=\sum_{i=1}^kn_i\delta_i\in\Delta$, then $\delta\in\Delta_p^+$.}
\end{equation}
Vice versa, every partition of $\Delta_p=\Delta^+_p\dot{\cup}\Delta_p^-$ with property \eqref{partition} corresponds to a unique Weyl chamber $K\in\pi_0(V_p^{\circ})$.

\subsubsection{Geometric interpretation of \ref{abstract setting}}
\label{geometric setting}
Let $(X,\mu)$ be a marked K3 surface with period point $p=\mu_{\C}(\Hdg20X)\in\Omega$.
The abstract notions discussed in the preceding subsection correspond to the following geometric notions on $X$ under $\mu$:
The \emph{positive cone} is the connected component $C(X)\subset\{x\in\Hdg11X\cap\Hsh2X{\R}\mid\langle x,x\rangle>0\}$ containing K\"ahler classes.
Up to possibly changing the sign of $\mu$, we have $C(X)\defeq\mu_{\R}\inv(V_p)$.
The set of \emph{$(-2)$-classes} $\Delta(X)\defeq \{d\in\Hdg11X\cap\Hsh2X{\R}\mid\langle d,d\rangle=-2\}=\mu\inv(\Delta_p)$ is partitioned into effective divisorial classes $\Delta^+(X)$ and anti-effective ones $\Delta^-(X)=-\Delta^+(X)$, \cite[Proposition~VIII.3.7(i)]{BHPV}. The effective ones are linear combinations of the Poincar\'e duals of the fundamental classes of smooth rational curves in $X$.
Each $\delta\in\Delta(X)$ induces a \emph{Picard-Lefschetz reflection} $\mu\inv\circ s_{\delta}\circ\mu\from\Hsh2X{\Z}\to\Hsh2X{\Z}$,
whose complex extension preserves the Hodge structure on $\Hsh2X{\Z}$.
The connected components of $\{x\in C(X)\mid \forall d\in\Delta^+(X)\colon\langle x,d\rangle\neq0\}$ are the \emph{Weyl chambers of $X$}.
The cone of K\"ahler classes $\kappa\in\Hdg11X\cap\Hsh2X{\R}$ equals the Weyl chamber $C^+(X)\defeq \{x\in C(X)\mid\forall d\in\Delta^+(X)\colon\langle x,d\rangle>0\}$, \cite[Corollary~VIII.3.9]{BHPV}.
Note that $\mu_{\R}(C^+(X))$ is the abstract Weyl chamber inducing the partition $\mu(\Delta^+(X))\dot{\cup}\mu(\Delta^-(X))$ of $\Delta_p$, where property~\eqref{partition} follows from the K\"ahler property.

\subsubsection{Moduli spaces of K3 surfaces}
\label{K3moduli}

We define the following real-analytic spaces.
\begin{align*}
K\Omega&\defeq \{(\kappa,p)\in\Lambda_{\R}\x\Omega\mid\kappa\in V_p\}\\
(K\Omega)^{\circ}&\defeq \{(\kappa,p)\in\Lambda_{\R}\x\Omega\mid\kappa\in V_p^{\circ}\}\\
\tilde{\Omega}&\defeq (K\Omega)^{\circ}/\sim,
\end{align*}
where the equivalence relation on $(K\Omega)^{\circ}$ is defined by $(\kappa,p)\sim(\kappa',p')$ if $p=p'$ and if $\kappa,\kappa'$ lie in the same Weyl chamber of $V_p$.
We call $\tilde\Omega$ the \emph{Burns-Rapoport space}.
The projection $\pi\from\tilde{\Omega}\to\Omega$, $\pi([\kappa],p)=p$ makes $\tilde{\sF}=(\tilde\Omega\stackrel{\pi}{\to}\Omega)$ a topological sheaf on $\Omega$, thus uniquely inducing an analytic structure in the sense of \cite{GrothendieckAnalyticSpaces} on $\tilde\Omega$ for which $\pi$ becomes locally biholomorphic, \cite[Theorem~2.5]{BR75}. With this structure, $\tilde{\Omega}$ is thus $20$-dimensional and smooth.

When $(X\stackrel{f}{\to}S,\mu)$ is a marked family of K3 surfaces with period map $\per$, we obtain the \emph{refined period map} $\tilde{\per}\from S\to\tilde{\Omega}$, $s\mapsto(\mu_{\R}(C^+(X_s)),\per(s))$.
When $\per$ is an embedding, this induces a section $\tilde{\per}\circ\per\inv\in\tilde{\sF}(\per(S))$.
By gluing marked local universal families of K3 surfaces one can construct a universal marked family over a non-Hausdorff fine moduli space $M$ of marked K3 surfaces such that the associated period map $M\to\Omega$ is locally biholomorphic, \cite[Proposition~10.2, Section~10.3]{LP80}.
The \emph{Global Torelli Theorem} states that the refined period map induces an isomorphism $M\bij\tilde{\Omega}$ over $\Omega$, see \cite[Theorem~10.5]{LP80}, \cite[Theorem~14.1]{BHPV}, and \cite[Section~1]{Bea81}.
Hence $\tilde\Omega$ carries a universal marked family $(\cM\stackrel{\nu}{\to}\tilde{\Omega})$ of K3 surfaces (and is thus likewise non-Hausdorff).

\subsubsection{The Weyl group for K3 surfaces and the action on the period domain}
\label{weyl group}
For each point $p\in\Omega$ the Picard-Lefschetz reflections $s_{\delta}$ for $\delta\in\Delta_p$ generate a group $\Weyl p$, the \emph{Weyl group at $p$}.
All $s_{\delta}$ for $\delta\in\Delta$ generate a discrete group $\Gamma$.
It follows from the work of Wall, Ebeling and Kneser that \[\Gamma:=\Orthp(\Lambda)\subset\Orth(\Lambda)\] is the index two subgroup of lattice automorphisms preserving any given orientation in the three positive directions, \cite[Theorem~14.2.2]{K3book}.

We extend the action of $\Gamma$ linearly to $\Lambda_{\R}$ and $\Lambda_{\C}$, thus $\Gamma$ acts also on $\Omega$.
For each $p\in\Omega$ the stabilizer $\Gamma_p$ contains the Weyl group $\Weyl p$.
Moreover, $\Weyl p$ acts properly discontinuously on the positive cone $V_p$ at $p$, inducing a simply transitive action on the set of its Weyl chambers, see \cite[Corollary~VIII.3.9]{BHPV}.
By continuity we thus have $\gamma(V_p)=V_{\gamma(p)}$ for each $\gamma\in\Gamma$ and $p\in\Omega$.
As $\pi_0(V_p^{\circ})=\pi\inv(p)\subset\tilde{\Omega}$, this induces an equivariant $\Gamma$-action on $\tilde{\Omega}\stackrel{\pi}{\to}\Omega$, such that $\Weyl p$ acts transitively on the fiber $\pi\inv(p)\subset\tilde{\Omega}$ for each $p\in\Omega$, \cite[Lemma~10.4]{LP80}. 

While set-theoretically $\Omega/\Gamma$ is the moduli space of biholomorphism classes of K3 surfaces, the action of $\Gamma$ on $\Omega$ is far from being properly discontinuous:
The union of translates of the $(-2)$-hypersurfaces $H_\delta\defeq\P(\delta^\perp) \cap \Omega$ is dense in the period domain, and the quotient $\Omega/\Gamma$ is very non-Hausdorff, cf.~\cite[Proposition~7.1.3]{K3book}.
Moreover, to calculate the stabilizer $\Gamma_p$ of $\Gamma$ at a point $p\in\Omega$, we first realize $p$ as the period point of a marked K3 surface $(X,\mu)$.
The holomorphic automorphism group $\Aut(X)$ is a subgroup of $\Orthp(\Lambda)$ via its action on $\Hsh2{X}{\Z}\stackrel{\mu}{\isom}\Lambda$, see \cite[Proposition~15.2.1]{K3book}, and that the stabilizer of $p$ in $\Gamma$, i.e. the group of Hodge isometries of $\Hsh2{X}{\Z}$, is equal to
\[\Gamma_p \cong \Weyl p \rtimes \Aut(X).\]
In particular, period points of K3 surfaces with automorphisms of infinite order have infinite stabilizer in $\Gamma$, again underlining the fact that the $\Gamma$-action on $\Omega$ is not proper.
On the other hand, if $\Delta_p \neq \emptyset$, then not all elements in $\Gamma_p$ are geometrically induced by holomorphic automorphisms of $X$.

\subsubsection{The twistor construction for IHS manifolds}
\label{twistor}
Let $X=(M,I)$ be an $n$-dimensional IHS manifold with complex structure $I$.
Then, by Yau's solution of the Calabi conjecture, in every K\"ahler class $a\in\Hsh2M{\R}$ of $M$ there exists a unique \emph{K\"ahler-Einstein metric} -- that is, a Ricci-flat Riemannian metric $g$ on $M$ with $g(I(\cdot),\cdot)\in a$ that is \emph{hyperk\"ahler}, i.e. K\"ahler with respect to three complex structures $I,J,K$ on $M$ satisfying the quaternionic relations.
Then $g$ is K\"ahler with respect to a sphere of complex structures $\cI\defeq \{I^{t_1,t_2,t_3}_M\defeq t_1I+t_2J+t_3K\mid (t_1,t_2,t_3)\in S^2\subset\R^3\}$ on $M$.
We fix an identification of $\cI\cong S^2$ with $\P^1$, carrying the natural complex structure $I_{\P^1}$.
For every $I_M^t\in\cI$ the corresponding K\"ahler form is a linear combination $\kappa_t\defeq g(I_M^t(\cdot),\cdot)=t_1\kappa_I+t_2\kappa_J+t_3\kappa_K$ of the K\"ahler forms $\kappa_I,\kappa_J,\kappa_K$ of $I,J,K$.
The so-called \emph{twistor space} of $M$ is $T\defeq M\x S^2$ equipped with the following four structures \cite[Theorem~3.3]{HKLR}, \cite[Section~10.1.3]{Joy07}:
\begin{enumerate}
\item An integrable almost complex structure defined by $I_{T,(m,t)}\defeq ((I^t_M)_m,I_{\P^1,t})$ for $m\in M$ and $t\in S^2\cong\P^1$.
The projection $f\from T\to\P^1$ is holomorphic and for every $t\in\P^1$ the fiber can be identified with $T_t=f^{-1}(t)\cong (M,I^t_M)$.
\item\label{twistor line}For each $m\in M$ the real submanifold $l_m\defeq \{m\}\x\P^1$ is actually a smooth rational curves in $T$. We call these curves \emph{twistor lines}; they have normal bundle $\Nb{l_m}{T}\cong\O_{\P^1}(1)^{\oplus n}$.
\item\label{antipodal} An \emph{antiholomorphic} involution $\tau\from T\to T$, $(m,I^t_M)\mapsto(m,I^{-t}_M)=(m,-I^t_M)$ with $\tau(l_m)=l_m$ for all $m\in M$ induced by and covering the antipodal map on $S^2\cong\P^1$.
\item\label{omega} A $\tau$-invariant relative symplectic form $\omega\in \Hsh0X{\Omega_{T/\P^1}^2\otimes f^*\O_{\P^1}(2)}$, which for every $t\in\P^1$ induces via base change a holomorphic-symplectic form $\omega_t$ on $(M,I^t_M)$ that at the point $t=(1,0,0)$ yields $\omega_{1,0,0}=\kappa_J+i\kappa_K$.
\end{enumerate}

Given a twistor space constructed using a hyperk\"ahler metric $g$ on an IHS manifold $X$, the metric $g$ can be recovered from the data a)--d) listed above via the so-called \emph{inverse Penrose construction}, see for example \cite[Theorem~3.3]{HKLR}. 
Properties a) and d) make $\cF=(T\stackrel{f}{\to}\P^1)$ a family of IHS manifolds, which we will call \emph{twistor family}.
As the projective line is simply connected, every marking of the fiber $X = T_{(1,0,0)}$ extends uniquely to a marking $\mu$ of the family.

We will see in the subsequent section that dropping c) from the list above leads to a correspondence in the realm of complex-Riemannian geometry.

\subsection{Complex-Riemannian geometry}

Complex-Riemannian geometry is the study of complex manifolds endowed with \emph{complex-symmetric} non-degenerate tensors. For later reference, in this section we summarize the basic definitions and fundamental theorems we need for our discussion. 

\subsubsection{Complex-Riemannian metrics}
We mainly follow \cite[Sections~I--II]{Leb83}. A \emph{holomorphic Rie\-mann\-ian metric} on a complex manifold $X$ is a $\C$-linear symmetric non-degenerate $2$-tensor ${g}$ on its holomorphic tangent bundle. Such a metric can in local holomorphic coordinates $z_1,\ldots,z_n$ on an open subset $U\subset X$ be written as $\sum_{i,j=1}^n g_{ij}dz_i\otimes dz_j$ for some $g_{ij}\in\O_X(U)$ with nowhere vanishing $\det\big((g_{ij})\big)\in\O_X^{\x}(U)$. Note that this does not define a (pseudo-)Riemannian metric on the underlying real manifold, however, it may come from complexifying one: When $X$ is a complexification of a real-analytic manifold $M$, then every real-analytic Riemannian metric ${g}$ on $M$ in the classical sense can be extended to a holomorphic Riemannian metric ${g}_{\C}$ on an open neighborhood of $M$ in $X$.
For this, at each $m\in M$, using the local complex coordinates $z_j=x_j+iy_j$ from \Cref{localcoordinates}, we can write ${g}=\sum g_{ij}dx_i\otimes dx_j$ for real-analytic functions $g_{ij}$, and define ${g}_{\C}=\sum g_{ij}dz_i\otimes dz_j$ by extending the $g_{ij}$ to holomorphic functions.
As we have $M=\{y_1=\ldots=y_n=0\}$ locally in these coordinates, the $dy_j$ vanish on $M$, hence $\rest{{g}_{\C}}M={g}$.

More generally, a \emph{complex-Rie\-mann\-ian metric} on a real-analytic manifold $M$ is a real-analytic $\C$-linear symmetric $2$-tensor ${g}$ on the complexification $T^\C M$ of its real tangent bundle.

Every holomorphic Riemannian metric ${g}$ on a connected complexification $X$ of a real-analytic manifold $M$ is uniquely determined by the real-analytic, $\C$-linear tensor $\rest{g}M$, which is a complex-Riemannian metric.
Conversely, every complex-Riemannian metric on $M$ can be extended uniquely to a holomorphic Riemannian metric on a small neighborhood of $M$ in $X$. 

A \emph{holomorphic conformal structure} on a complex manifold $X$ is a holomorphic sub line bundle $\mathscr{L} \hookrightarrow \mathrm{Sym}^2(\sT_X)$ such that the associated zero set $\mathcal{Q}_\mathscr{L} \subset \mathbb{P}(\sT_X)$ is a non-singular quadric. Analogously, we define a \emph{complex conformal structure} on a real-analytic manifold $M$ to be a real-analytic complex sub line bundle $L \hookrightarrow \mathrm{Sym}^2(T^\C M)$ with the corresponding non-degeneracy condition.

Many of the standard concepts of Riemannian geometry, such as Levi-Civitá connections, curvature, parallel transport and holonomy groups, \emph{mutatis mutandis} exist for holomorphic Riemannian and complex-Riemannian metrics such that the expected properties hold. The reader is referred for example to \cite[Section~2]{MR1984459} and \cite[Section~I]{Leb83} for a discussion of some of these notions.\footnote{Note however that the authors refer to the geometry of holomorphic Riemannian metrics as \emph{complex}-Riemannian geometry.} 

By a \emph{family of metrics} $(g_t)_{t\in U}$ of a given type \emph{on a submersion} of real-analytic or complex manifolds $(X\to U)$, we mean a symmetric $2$-tensor on the respective relative tangent bundle of $X$ over $U$ that restricts to a metric of the specified type on each fiber. If $(X\to U)$ is more restrictively assumed to be a fiber bundle in the chosen category, say with fiber $M$, then we call $(g_t)_{t\in U}$ a \emph{family of metrics on $M$}. Analogously, we will speak of \emph{families of conformal structures} on submersions and fiber bundles. 
In particular, if in the real-analytic situation the map $(X \to U)$ is proper, then the real-analytic Ehresmann Theorem tells us that a family of metrics on $(X \to U)$ is a family of metrics on a given fiber $X_t$, $t \in U$. 

\subsubsection{Complex-hyperk\"ahler metrics}
\label{subsubsect_HK}

\begin{defi}[Complex-hyperk\"ahler metrics: holomorphic case]\label{HolHK}
 A holomorphic Riemannian metric $g$ on a connected complex manifold $X$ of dimension $n=4k$ will be called \emph{complex-hyperk\"ahler} if its holonomy $\Holono(g)$ is contained in the complex symplectic group\footnote{We adopt the convention that $\mathrm{Sp}(k, \C)$ is defined by preserving the standard complex-symplectic form on $\C^{2k}$.} $\Sp(k, \C)$ acting on $\C^{4k}$ via its embedding 
 \begin{equation}\label{groupembedding}
 \Sp(k, \C) \hookrightarrow \SO(4k, \C),
 \end{equation} obtained by complexifying the sequence of embeddings $\Sp(k) \hookrightarrow \SU(2k) \hookrightarrow \SO(4k)$ of compact Lie groups.
\end{defi}

 We will deal almost exclusively with the case $k=1$, where $\Sp(1, \C) = \SL(2, \C)$, and in which the embedding \eqref{groupembedding} becomes the embedding
\begin{align*}
 \SL(2,\C) &\hookrightarrow \SL(2,\C) \x_{\Z_2} \SL(2, \C) \cong \SO(4, \C) \\
  g &\mapsto [(g, \id{\C^2})].
\end{align*}

As in the real case, the holonomy condition can be equivalently expressed in terms of the existence of (a wealth of) compatible geometric structures on the complex manifold $X$. While \Cref{HolHK} is easy to state, these more explicit conditions are easier to check in practice. 
One equivalent set of conditions is given in \cite[Definition~4.2]{MR4477203} together with a reference to \cite{MR3279259}, where the same structure is considered as a special case of a classical object called a \emph{holomorphic $3$-web}, namely an \emph{$\SL(2)$-web generated by a trisymplectic structure}. Another equivalent definition occurs in \cite[p.~4f]{MR4248680}. 
In dimension four, these are the \emph{anti-selfdual (ASD) Ricci-flat} complex-Riemannian metrics, see for example \cite[Section~9.2--9.3]{Dun10}.

\begin{defi}[Complex-hyperk\"ahler metrics: real-analytic case]\label{CHK}
 A complex-Riemannian metric $g$ on a connected real-analytic manifold $M$ of dimension $n=4k$ will be called \emph{complex-hyperk\"ahler} if its holonomy $\Holono(g)$ is contained in the complex symplectic group $\Sp(k, \C) \subset \SO(4k, \C)$.\footnote{Roger Bielawski informed us that these are particular cases of \emph{pluricomplex structures} in the sense of \cite{MR3085661}.}
\end{defi}

It was the original insight of Penrose \cite{MR0439004} that what we will call \emph{complex twistor spaces} can be used to produce complex-hyperk\"ahler metrics that in general will not arise as the complexification of a real metric. The by now classical \emph{Non-linear Graviton Construction} of his can be formulated as follows.
\begin{theo}\label{InversePenrose}
 Let $\cF=(f \from T\to\P^1)$ be a family of $n$-dimensional (compact) complex manifolds admitting a \emph{regular section} $\Sigma \subset T$, i.e., a section with normal bundle $\sN_{\Sigma/T} = \O_{\P^1}(1)^{\oplus n}$, and a relative symplectic form $\omega \in \Hsh0T{\Omega_{T/\P^1}^2\otimes f^*\O_{\P^1}(2)}$. Then, this data (in a canonical way) endows the connected component $\Sec_0(T) \subset \dou(T)$ of $[\Sigma]$ inside the open subset of $\dou(T)$ consisting of regular sections with a complex-hyperk\"ahler metric. 
 
If more specifically $\cF=(f \from T\to\P^1)$ is a twistor family constructed as in Section~\ref{twistor} from a hyperk\"ahler metric $g$ on an IHS manifold $X$, then the restriction of the holomorphic metric postulated above to the connected component of $[\Sigma]$ inside the fixed point locus $\Sec_0(T)^\tau \ni [\Sigma]$\footnote{This component is real-analytically isomorphic to $X$.} of the induced antiholomorphic involution  on $\Sec_0(T)$ given by \Cref{douinvolution} recovers the hyperk\"ahler metric $g$ on $X$. 
\end{theo}

In \cite[Section~4.4, Example~3]{MR1329460} the four-dimensional complex version of \Cref{InversePenrose} was derived from a more general result about holonomy of affine connections on section spaces of fibrations; see also \cite[p.~1505, Example~2]{MR1327027}. A discussion of the construction of the metric on $\Sec_0(T)$ from the point of view of trisymplectic structures is given in \cite[Claim~5.4]{MR3279259}. A proof of the four-dimensional complex case in the language of ASD geometry can be found in \cite[Theorem~10.5.5]{Dun10}. An argument for the reproducing property of the construction in the case of twistor spaces for hyperk\"ahler metrics on IHS manifolds is given in \cite[Theorem~3.3]{HKLR}. See also \cite[Theorem~10.1.4]{Joy07}. 

Actually, there is a one-to-one correspondence between twistor spaces and hyperk\"ahler metrics, which actually holds true also in the complex case, see e.g.~\cite{MR0439004} and \cite[Theorem~10.5.5]{Dun10} for the four-dimensional case; however, one has to carefully adapt the construction of twistor spaces to this setup, cf.~\cite[p.~14]{MR4477203}, \cite[Section 3.3]{MR4248680}, and \cite[p.~266]{Dun10}. We will apply this reconstruction result in the proof of \Cref{conformalStructures} below.

\section{Cycles in the period domain and twistor families}
\label{cycle section}

While variations or special cases of some of the results discussed in this section are partially contained in the literature or can be deduced from known facts, most of the time, precise references are missing; therefore we include a detailed discussion. We also use the opportunity to correct several wrong claims made in the literature.

Let $\Lambda$ be an IHS lattice of rank $n$.
We denote by $\mathsf{S}_d(\pdom{\Lambda})\subset\dou(\pdom{\Lambda})$ the subset of smooth rational curves with embedding degree $d$ in the IHS period domain $\pdom{\Lambda}$.
This is either empty or a smooth open subspace of dimension $m(n,d)\defeq (n-2)(d+1)-3$, see \cite[Corollary~4.2]{BKS} and also \Cref{Deev} below.
In the latter case, we denote by $(\cC_d(\pdom{\Lambda})\stackrel{\pr1}{\lto}\mathsf{S}_d(\pdom{\Lambda}))$ the universal family of smooth rational curves over it (see Section~\ref{douady section}).
Let us first focus on the case $\mathsf{S}_2(\pdom{\Lambda})$, which is non-empty of dimension $m(n,2)=3n-9=3n_{-}$.

\subsection{Grassmannians and cycle spaces}\label{subsect:cyclespaces}

Let $\G\defeq\Gr 3 {\Lambda_{\C}}$ be the Grassmannian of three-spaces $V\subset\Lambda_{\C}$, equipped with the antiholomorphic involution $\tau\from V\mapsto \overline{V}$ induced by the complex conjugation $\overline{\, \cdot\, }$ on $\Lambda_\C$.
We define the following subspaces
\begin{align*}
\G^+&\defeq\{V\in \G\mid \forall v\in V\setminus\{0\}\colon\langle v,\bar{v}\rangle_{\C}>0\},\\
\G_{\R}&\defeq \G^{\tau}=\{V\in \G\mid \overline{V}=V\},\\
\G_{\R}^+&\defeq \G_{\R}\cap \G^+=(\G^+)^{\tau}.
\end{align*}

Then $(\G^+,\rest{\tau}{\G^+})$ is a complexification of $\G_{\R}^+$ and $\G^+_{\R}$ is a totally real submanifold of $\G^+$.
Note that $\G_{\R}$ is real-analytically isomorphic to $\Gr 3 {\Lambda_{\R}}$ via $V\mapsto V\cap\Lambda_{\R}$, and that $\dim_{\C}\G^+=\dim_{\R}\G_{\R}^+=3n_-$.

We start by proving two lemmata on the relation between three-spaces and cycles; in the case of the K3 period domain similar results appear in~\cite[p.~234]{FHW} and \cite{HeierDiplom}.

\begin{lemm}[Three-spaces from degree $2$ rational curves]
\label{signature}
Every smooth rational curve $[C]\in\mathsf{S_2(\pdom{\Lambda})}$ of embedding degree $2$ in $\pdom{\Lambda}$ spans a unique projective plane $\langle C\rangle=\P(V)\subset\P(\Lambda_{\C})$ with $\P(V)\cap Q_{\Lambda}=C$.
The signature of $\langle\cdot,\overline{\cdot}\rangle_{\C}$ on the three-space $V\in\G$ is either $(3,0,0)$, $(2,1,0)$ or $(2,0,1)$.
If $C=\overline{C}$, then $V\in\G^+_{\R}$ is positive.
\end{lemm}
\begin{proof}
Let $[C]\in\mathsf{S}_2(\pdom{\Lambda})$.
As the space of homogeneous quadratic polynomials in two variables has dimension three, $C$ indeed spans a projective plane $\langle C\rangle\cong\P^2$ in $\P(\Lambda_{\C})$, which must be the projectivization $\P(V)$ of a complex three-space $V\subset\Lambda_{\C}$.
As $C\subset \P(V)\cap Q_{\Lambda}$ is an inclusion of smooth plane projective quadrics, equality must hold.
If the claim on the index of $\langle\cdot,\overline{\cdot}\rangle_{\C}$ on $V$ was false, there would be a complex plane $W\subset V$ on which $\langle\cdot,\overline{\cdot}\rangle_{\C}$ is negative semi-definite.
Then $\P(W)\cong\P^1$ would intersect the curve $C\cong\P^1$ within $\P(V)\cong\P^2$.
However, this is impossible because $C\subset\pdom{\Lambda}$ and thus $\langle v,\bar{v}\rangle_{\C}>0$ for every point $[v]\in C \cap \P(W)$, $v\in\Lambda_{\C}\setminus\{0\}$.

For the last claim, note that $C=\overline{C}$ implies $V=\overline{V}$. Let $p=[v]\in C\subset\pdom{\Lambda}$ be a point, then $\langle v,v\rangle_{\C}=0$ and $\langle v,\overline{v}\rangle_{\C}>0$. Hence $v\in\Lambda_{\C}\setminus\Lambda_{\R}$, and by rescaling we may assume $\langle v,\overline{v}\rangle_{\C}=1$.
We extend $\{v,\overline{v}\}$ to a basis of $V$ via a $\kappa$ orthogonal to $v$ and $\overline{v}$.
Then, necessarily $\kappa\in\Lambda_{\R}$.
If we had $\langle\kappa, \overline{\kappa}\rangle_{\C} = \langle\kappa,\kappa\rangle_{\C}\le0$, there would be a $t\in\R$ with $t^2=-\langle\kappa,\kappa\rangle_{\C}$.
Hence $w\defeq\frac{t}{2}(v+\overline{v})+\kappa\in V$ would be real with $\langle w,\overline{w}\rangle_{\C}=\langle w,w\rangle_{\C}=t^2-\langle \kappa,\kappa\rangle_{\C}=0$ and thus $[w]\in \P(V)\cap Q_{\Lambda}\setminus\pdom{\Lambda}$, a contradiction.
\end{proof}

\begin{prop}[Degree $2$ rational curves from positive three-spaces]\label{cyclemap}
The map $\psi\from \G^+\to\dou\pdom{\Lambda}$, $V\mapsto [\P(V)\cap\pdom{\Lambda}]$ is a holomorphic embedding of $\G^+$ as a connected open subset of the space $\mathsf{S}_2(\pdom{\Lambda})$ of smooth rational curves that have embedding degree $2$. 
\end{prop}

\begin{proof}
Let $V\subset\Lambda_{\C}$ be a complex three-space.
When the sesquilinear form $\langle \cdot,\bar{\cdot}\rangle_{\C}$ is positive definite on $V$, then the intersection of $\P(V)$ with $Q_{\Lambda}$ lies completely in $\pdom{\Lambda}$ and is a smooth plane projective quadric in $\P(V)\isom\P^2$.
Hence $\P(V)\cap Q_{\Lambda}$ is a smooth rational curve in $\pdom{\Lambda}$ of embedding degree $2$.
Using the defining holomorphic vector bundle of rank three on the Grassmannian $\G$, we obtain a holomorphic $\P^1$-bundle over $\G^+$ and then $\psi$ is the holomorphic map induced by the universal property of the Douady space of $\pdom{\Lambda}$. On the other hand, by \Cref{signature}, every smooth rational curve $C\subset\pdom{\Lambda}$ of embedding degree $2$ spans a unique projective plane $\langle C \rangle = \P(V)\subset\P(\Lambda_{\C})$ for a $V\in\G$, which might not lie in $\G^+$, but at least it follows that $\psi$ is injective.
Moreover, the deformations of $C$ in $\pdom{\Lambda}$ are unobstructed and build a family of smooth rational curves in $\pdom{\Lambda}$ over a $3n_-$-dimensional smooth open subspace $S\subset\dou \pdom{\Lambda}$.
$\P^1$ is rigid and the embedding degree of smooth rational curves is locally constant in families \cite[Corollary~3.15]{BKS}.
Furthermore, $\pdom{\Lambda}$ does not contain any straight line \cite[Lemma~3.7]{BKS}, hence $C$ cannot degenerate into a singular curve in $\pdom{\Lambda}$, and therefore $S$ is a connected component of $\mathsf{S}_2(\pdom{\Lambda})\subset\dou\pdom{\Lambda}$. 
As $\G^+$ and $S$ have the same dimension, $\psi$ is an open holomorphic embedding of $\G^+$ into $S$ and hence into $\mathsf{S}_2(\pdom{\Lambda})$.
\end{proof}

\begin{defi}[Cycle space]
\label{cyclespace}
Let $\psi \from \G^+\to\dou\pdom{\Lambda}$ be as in \Cref{cyclemap}.
Let $C_1(\pdom{\Lambda})$ be the connected component of $\mathsf{S}_2(\pdom{\Lambda})$ containing $\psi(\G^+) =: C_1^+(\pdom{\Lambda})$.
We call $C_1(\pdom{\Lambda})$ the \emph{complex cycle space} and $C_1^+(\pdom{\Lambda})$ the \emph{positive cycle space}.
We refer to $C_1(\pdom{\Lambda})_{\R}\defeq\{[C]\in C_1(\pdom{\Lambda})\mid C=\overline{C}\}=\psi(\G^+_{\R})\subset C_1^+(\pdom{\Lambda})$, cf.~\Cref{signature}, as the \emph{real cycle space}. We denote by \[\cC^{(+)}(\pdom{\Lambda})\defeq\rest{\cC_2(\pdom{\Lambda})}{C_1^{(+)}(\pdom{\Lambda})}\subset \dou\pdom{\Lambda}\x C_1(\pdom{\Lambda})\] the universal family of smooth rational curves together with its projection down to $C_1^{(+)}(\pdom{\Lambda})$, and by $\cC_{\R}(\pdom{\Lambda})$ its restriction to $C_1(\pdom{\Lambda})_{\R}$.
\end{defi}

In the trivial case $n=3$ we have $n_-=0$ and $\G=\G^+=\{\Lambda_{\C}\}$, thus  $C_1(\pdom{\Lambda})=C_1^+(\pdom{\Lambda})=C_1(\pdom{\Lambda})_{\R}$ consists only of the single cycle $\psi(\Lambda_{\C})=[\pdom{\Lambda}]=[Q_{\Lambda}]$.

\begin{rema}[Not every complex cycle is positive]\label{rem:LooijengaWrong}
Let $n>3$. Then, in contrast to what is claimed in \cite[Question~3.1]{Loo21}, the inclusion of $C_1^+(\pdom{\Lambda})$ into $C_1(\pdom{\Lambda})$ is strict.
We demonstrate via the following explicit example found by Ana-Maria Brecan that in fact all three possible signatures mentioned in \Cref{signature} appear for complex three-spaces spanned by cycles $[C]\in C_1(\pdom{\Lambda})$.
\end{rema}

\begin{exam}[Non-positive three-spaces yielding complex cycles in $\pdom{\Lambda}$]\label{ExampleLooijenga}
Let $n>3$.
We fix a real orthogonal basis $\{e_1,\ldots,e_n\}$ of $\Lambda_{\R}$ such that $\langle e_i,e_i\rangle_{\R}$ equals $1$ for $i\in\{1,2,3\}$ and $-1$ for $i\in\{4,\ldots,n\}$.
Hence $\langle\cdot,\bar{\cdot}\rangle_{\C}$ is positive definite on $V_0\defeq\C e_1\oplus\C e_2\oplus\C e_3$.
We take $[C_0]\defeq\psi(V_0)$ as a \emph{base cycle} inside $\pdom{\Lambda}$ and deform $V_0$ in $\G$ via $V_t \defeq \C (e_1+ite_4)\oplus\C e_2\oplus\C e_3$ for $t\in\R$.

First of all, to show that all curves $C_t\defeq\P(V_t)\cap Q_{\Lambda}$ are smooth, we check that $\rest{\langle\cdot,\cdot\rangle_{\C}}{V_t}$ is non-degenerate.
 In fact, let $v=\alpha_1(e_1+ite_4)+\alpha_2 e_2+\alpha_3e_3\in V_t$, then $\langle v,w\rangle_{\C}=0$ for all $w\in V_t$ implies $0=\langle v,e_1+ite_4\rangle_{\C}=\alpha_1(1+t^2)$, $0=\langle v,e_2\rangle_{\C}=\alpha_2$, $0=\langle v,e_3\rangle_{\C}=\alpha_3$ and thus $v=0$. On the other hand, note that $\langle e_1+ite_4,\overline{e_1+ite_4}\rangle_{\C}=1-t^2$; hence the signature of $\rest{\langle\cdot,\bar{\cdot}\rangle_{\C}}{V_t}$ is $(3,0,0)$ for $t<1$, $(2,0,1)$ for $t=1$ and $(2,1,0)$ for $t>1$.
Let $v=\alpha_1(e_1+ie_4)+\alpha_2e_2+\alpha_3e_3\in V_t$ with $0=\langle v,v\rangle_{\C}=(1+t^2)\alpha_1^2+\alpha_2^2+\alpha_3^2$.
Then $|\alpha_2^2+\alpha_3^2|=(1+t^2)|\alpha_1|^2$ and thus $\langle v,\bar{v}\rangle_{\C}=(1-t^2)|\alpha_1|^2+|\alpha_2|^2+|\alpha_3|^2\ge(1-t^2)|\alpha_1|^2+|\alpha_2^2+\alpha_3^2|=2|\alpha_1|^2\ge0$.
Hence $\langle v,\bar{v}\rangle_{\C}\le0$ implied $\alpha_1=0$ and thus also $\alpha_2=\alpha_3=0$ and $v=0$.
We conclude $\langle v,\bar{v}\rangle_{\C}>0$ for every $[v]\in C_t$ and thus $C_t \subset \pdom{\Lambda}$.

Hence the $C_t$ are smooth rational curves in $\pdom{\Lambda}$ and all $[C_t]$ must lie in the same component of $S_2(\pdom{\Lambda})$ as the base cycle $[C_0]$.
Therefore the $[C_t]$ lie in the complement $C_1(\pdom{\Lambda}) \setminus C_1^+(\pdom{\Lambda})$ for $t\ge1$. 
\end{exam}

\begin{lemm}[Intersecting cycles with hyperplanes]
\label{transverse}
Let $[C]\in C_1^+(\pdom{\Lambda})$ and $\delta\in\Lambda$ with orthogonal complement $\delta^{\perp}\subset\Lambda_{\C}$. Then either $\P(\delta^{\perp})\supset C$ or $\P(\delta^{\perp})$ intersects $C$ transversely in two distinct points.
\end{lemm}
\begin{proof}
By \Cref{cyclemap} there is a positive three-space $V\in\G^+$ with $\psi(V)=[C]$.
By dimension reasons we have $\dim_{\C}(V\cap\delta^{\perp})\ge2$.
If this equals $3$, then $\P(\delta^{\perp})\supset\P(V)\supset C$.
Otherwise $V\cap\delta^{\perp}$ is a plane in $V$, on which the sesquilinear form $(v,w)\mapsto\langle v,\bar w\rangle_{\C}$ is positive definite.
Hence $C\cap\P(\delta^{\perp})=\P(V)\cap Q_{\Lambda}\cap\P(\delta^{\perp})=Q_{\Lambda}\cap \P(V\cap\delta^{\perp})$ is completely contained in $\pdom{\Lambda}$ and defined by a non-degenerate projective quadric in $\P(V\cap\delta^{\perp})\cong\P^1$ -- that is, two distinct points.
Moreover, by degree reasons, in the latter case both intersections are transverse.
\end{proof}

\begin{rema}
\label{realcycles}
Note that $\tau$ induces a fixed-point free antiholomorphic involution of $\pdom{\Lambda}$ via $[x]\mapsto[\tau(x)]$, and due to \Cref{douinvolution} also an antiholomorphic involution of $C_1(\pdom{\Lambda})$ that we denote by $\tau$ as well.
We have $$C_1(\pdom{\Lambda})_{\R}=C_1(\pdom{\Lambda})^{\tau}=\{[C]\in C_1(\pdom{\Lambda})\mid C=\overline{C}\},$$ hence $(C_1(\pdom{\Lambda}),\tau)$ is a complexification of $C_1(\pdom{\Lambda})_{\R}$ and we get $$\dim_{\C}C_1(\pdom{\Lambda})=\dim_{\R}C_1(\pdom{\Lambda})_{\R}=3n_-.$$
\end{rema}

\subsection{The group-theoretic cycle space}\label{M_D}

The period domains $\pdom{\Lambda}$ are special cases of so-called \emph{flag domains}, whose cycle spaces have been studied intensively in Representation Theory. The following summary of known results is based mostly on \cite[Sections~5.1 and 16.4]{FHW}.

The complexification $\SO_0(3, n_{-})^\C = \SO(n, \C)$ acts holomorphically on the Douady space $\dou Q_{\Lambda}$ of the quadric $Q_\Lambda$.
Choosing the same real cycle $[C_0] \in C_1(\pdom{\Lambda})_{\R} \subset \dou Q_{\Lambda}$ as in \Cref{ExampleLooijenga} as a base point, it is clear that $g (C_0) \subset \pdom\Lambda$ for all $g \in \SO_0(3, n_{-})$.
Moreover, for $g \in \SO(n, \C)$ close to the totally real $\SO_0(3, n_{-})$, we still have $g (C_0) \subset\pdom\Lambda$. The orbit $\SO(n, \C) \cdot [C_0] \subset \dou Q_{\Lambda}$ is a complex submanifold of $\dou Q_{\Lambda}$ containing $[C_0]$. It is biholomorphic to $\SO(n, \C) /(\SO(3, \C) \x \SO(n_{-}, \C))$ and its intersection with $C_1(\pdom{\Lambda})$ consists of finitely many connected components, one of which contains the base cycle; we call this component $C_1^{\mathrm{Lie}}(\pdom{\Lambda})$. Clearly, it is a locally closed submanifold of $C_1(\pdom{\Lambda})$. However, with the notation introduced above, we also have by definition an open inclusion
\begin{align*}C_1^{\mathrm{Lie}}(\pdom{\Lambda}) &= \{g\cdot [C_0] \mid g \in \SO(n, \C) \text{ and } g(C_0) \subset \pdom \Lambda \}^\circ \\ &\subset \SO(n, \C) /(\SO(3, \C) \x \SO(n_{-}, \C)).
\end{align*}
We note that the $\SO_0(3, n_{-})$-orbit through $[C_0]$, which contains the $\SO_0(3, n_{-})$-translates of the base cycle inside the period domain, is isomorphic to the symmetric space $\SO_0(3, n_{-})/(\SO(3) \x \SO(n_-))$. These are exactly the cycles in $\psi(\G_{\R}^+)$. The inclusion $\SO_0(3, n_{-})/ (\SO(3) \x \SO(n_-)) \hookrightarrow C_1^{\mathrm{Lie}}(\pdom{\Lambda})$ is a complexification.

On the one hand, the group-theoretic cycle space $C_1^{\mathrm{Lie}}(\pdom{\Lambda})$ is a \emph{closed} submanifold of $C_1(\pdom{\Lambda})$ by \cite{MR2271422}. On the other hand, \cite[Theorem~18.4.13]{FHW} says that $C_1^{\mathrm{Lie}}(\pdom{\Lambda}) \hookrightarrow C_1(\pdom{\Lambda})$ is also an \emph{open} inclusion. We conclude that 
\begin{equation}
 C_1^{\mathrm{Lie}}(\pdom{\Lambda}) = C_1(\pdom{\Lambda}); 
\end{equation}
that is, every cycle $[C]$ in the complex cycle space $C_1(\pdom{\Lambda})$ is actually a translate of the base cycle by an element $g \in \SO(n, \C)$. In other words, if $[C_0] = [\mathbb{P}(V) \cap \pdom{\Lambda}]$ for $V\in\G$, then $[C] = [\mathbb{P}(g(V)) \cap \pdom{\Lambda}]$. While intersections of the quadric $Q_\Lambda$ with projectivizations of positive-definite three-spaces $V$ are in $C_1^{\mathrm{Lie}}(\pdom{\Lambda})$, yielding an embedding of the Hermitian-symmetric space $\G^+$ into $C_1^{\mathrm{Lie}}(\pdom{\Lambda})$, the complement is actually non-empty if $n>3$, as we saw in \Cref{rem:LooijengaWrong} and \Cref{ExampleLooijenga} above. 



\subsection{Induced actions on the cycle space}
\label{subsection_cycle_action}

The group $G_0$ acts (transitively) on the period domain $\pdom{\Lambda}$ and hence induces a real-analytic action by holomorphic transformations on the universal family
$\cC(\pdom{\Lambda})$ and on $C_1(\pdom{\Lambda})$, making the natural projection equivariant. The open subset $C^{+}(\pdom{\Lambda})$ is preserved by this action. The action restricts to a real-analytic equivariant action on $\cC(\pdom{\Lambda})_{\R}\to C_1(\pdom{\Lambda})_{\R}$. Moreover, the map $\psi$ is equivariant with respect to this action on the target and with respect to the $G_0$-action on $\G^+$.

\begin{theo}
The $G_0$-action on $C_1(\pdom{\Lambda})$ is proper. 
\end{theo}
\begin{proof}
 Fels--Huckleberry--Wolf proved in \cite[Theorem~11.3.7]{FHW}  that $C_1(\pdom{\Lambda})$ coincides with the \emph{Akhiezer-Gindikin domain}, which is constructed with the defining property that the action of $G_0$ be proper, see \cite[Section~6.1B]{FHW}. 
\end{proof}

The following direct consequence is of crucial importance for us. 
\begin{prop}[Action on the cycle space is properly discontinuous]\label{ProperlyDisc}The action of $\Gamma$ on $C_1(\pdom{\Lambda})$ is properly discontinuous. Consequently, there exists a quotient $\pi\from C_1(\pdom{\Lambda}) \to C_1(\pdom{\Lambda})/\Gamma$ in the category of complex spaces. In particular, $C_1(\pdom{\Lambda})/\Gamma$ is a Hausdorff complex space, and $\pi$ is holomorphic.
\end{prop}
\begin{proof} The group $\Gamma$ is a closed subgroup of $G_0$. Properness of the $G_0$-action hence implies that the action of $\Gamma$ on $C_1(\pdom{\Lambda})$ is proper, and hence properly discontinuous, as $\Gamma$ is a fortiori a discrete subgroup of $G_0$. The statements about the existence of quotients by actions of properly discontinuous actions are classical, see \cite{CartanDiscreteQuotients}.
\end{proof}

For the remainder of this subsection, we focus on the K3 period domain $\Omega$ and prove two basic properties regarding the interaction of cycles with reflection in $(-2)$-classes. 

First, while the union of the $H_\delta$ for $\delta \in \Delta$ is dense in $\Omega$, properness of the $\Gamma$-action implies that the situation is better in $C_1^+(\Omega)$. 
\begin{coro}\label{cor:FixedPointSetsAnalytic}
 The set of positive cycles fixed by some Picard-Lefschetz reflection, 
 \[ C_1^+(\Omega)^{\Delta} \defeq \bigcup_{\delta\in \Delta} C_1^+(\Omega)^{s_\delta}, \]
 is a closed, proper analytic subset of $C_1^+(\Omega)$.
\end{coro}
\begin{proof}
 This just follows from the fact that $C_1^+(\Omega)^{\Delta}$ is a union of components of the ramification locus of the holomorphic map $\pi$. 
\end{proof}

Second, while $[C] \in C_1^+(\Omega)^{s_\delta}$ a priori only means that $C \subset \Omega$ is preserved as a set, more is true.

\begin{lemm}[Cycles fixed by reflections]\label{FixedEqualsPointwiseFixed}
 Let $\delta \in \Delta$. If $[C] \in C_1^+(\Omega)^{s_\delta}$, then $C \subset H_\delta \subset \P(\delta^\perp)$, that is, $s_{\delta}$ fixes $C$ pointwise.
\end{lemm}
\begin{proof} 
Let $V\subset\Lambda_{\C}$ be the positive complex three-space such that $C = \P(V) \cap \Omega$.
Suppose that $C$ is not contained in $H_\delta$.
Then, by \Cref{transverse} it must intersect $\P(\delta^\perp)$ in two distinct points, say $p = [v_1]$ and $q=[v_2]$, where $v_1, v_2 \in V^{s_\delta}$.
Note that $\Lambda_{\C}$ splits into $s_\delta$-eigenspaces as
 \[\Lambda_{\C} = (\Lambda_{\C})^{s_\delta} \oplus \C\cdot \delta.\]
Since $C$ is preserved by $s_\delta$, $V$ is also $s_\delta$-stable.
As by assumption $V$ is not contained in $(\Lambda_{\C})^{s_\delta}$, it must contain $\delta$, cf.~\cite[Section~3]{Shayman}.
However, $\delta$ is real, thus we have $\langle\delta, \bar{\delta}\rangle_{\C} = \delta^2 < 0$, so that $V$ cannot be positive with respect to $\langle\cdot,\bar{\cdot}\rangle_{\C}$, contradicting the assumption.
\end{proof}

\subsection{Twistor families and related cycles in the K3 period domain}\label{cycleintro}

In the K3 surface case, it seems to be well-known which positive three-spaces $V\in\G^+$ correspond to twistor cycles, see for example \cite[Section~VIII.13]{BHPV}, \cite[Section~1.2]{Kob90} or \cite[Section~16.4B]{FHW}. As however for example in \emph{loc.~cit.}, the effect of $(-2)$-classes is not taken into account correctly, we give a detailed description. 

We start by making the following observation.

\begin{lemm}
\label{twistor cycle decomposition}
Let $(T\stackrel{f}{\to}\P^1,\mu)$ be a marked twistor family of IHS manifolds.
Then its induced period map embeds $\P^1$ in $\pdom{\Lambda}$ as a real cycle $[C_{X,\mu,a}]=\psi(V)\in C_1(\pdom{\Lambda})_{\R}$, such that for all $t\in\P^1$ we have $V=\mu_{\C}\left(\Hdg20{T_t}\oplus\C[\kappa_t]\oplus\Hdg02{T_t}\right)\subset\Lambda_{\C}$.
\end{lemm}

\begin{proof}
The induced period map $\per$ embeds $\P^1$ in $\pdom{\Lambda}$ as a real cycle $[C]=\psi(V)\in C_1(\pdom{\Lambda})_{\R}$ and it turns out that $V=\mu_{\C}\left(\C[\kappa_I]\oplus\C[\kappa_J]\oplus\C[\kappa_K]\right)\subset\Lambda_{\C}$ is determined by the three K\"ahler classes with respect to $I,J,K$, \cite[Proposition~25.7(ii)]{GHJ}.
In particular, for each $t\in\P^1$ we have $\mu_{\C}\left([\kappa_t]\right)\in V$.
Note that also the period point $\per(t)=\mu_{\C}(\Hdg20{T_t})$ lies in $C$, as well as its complex conjugate $\mu_{\C}(\Hdg02{T_t})$ because $[C]$ is real.
Due to orthogonality in $\Hsh2M{\C}$, we must have $V=\mu_{\C}\left(\Hdg20{T_t}\oplus\C[\kappa_t]\oplus\Hdg02{T_t}\right)$.
\end{proof}

For cycles $[C_{X,\mu,a}]$ coming from marked twistor families as in \Cref{twistor cycle decomposition} we use the term \emph{twistor cycles}. Next, we define
\begin{align*}
(\G^+)^{\circ}&\defeq\{V\in\G^+\mid \not\exists\delta\in\Delta\colon V\subset\delta^{\perp}\} ,&\Tw\defeq\psi((\G^+)^{\circ}) ,\\
(\G^+_{\R})^{\circ}&\defeq\G^+_{\R}\cap(\G^+)^{\circ}=\{V\in(\G^+)^{\circ}\mid V=\overline{V}\},&\Twr\defeq\psi((\G_{\R}^+)^{\circ}).
\end{align*}

Using \Cref{FixedEqualsPointwiseFixed}, we see that
\[
 \Tw = C_1^+(\Omega) \setminus C_1^+(\Omega)^\Delta \;\; \;\text{ and } \;\;\;\Twr = C_1(\Omega)_{\R} \cap \Tw,
\]
that is, we removed for each $\delta\in\Delta$ the cycles $[C] \in C_1^+(\Omega)$ with $s_\delta \in \Gamma_{[C]}$.
It hence follows from \Cref{cor:FixedPointSetsAnalytic} that $\Tw$ is the complement of a closed, proper analytic subset of $C_1^+(\Omega)$.
Furthermore, we note that
\[\codim_{\R}\bigl(\G^+_{\R} \setminus (\G^+_{\R})^{\circ} \bigr)=\codim_{\C}\bigl(\G^+ \setminus (\G^+)^{\circ}\bigr)=3.\]
As a consequence, both $\Tw$ and $\Twr$ are connected and simply connected, see e.g.~\cite[Lemma~8.11]{BL18}.
The condition $V\subset\delta^{\perp}$ for a (real) positive three-space $V$ and $\delta\in\Delta$ is invariant under $\Orth(\Lambda)$, hence $\Tw$ and $\Twr$ are invariant subspaces under the $\Orth(\Lambda)$-action on the positive cycle space. With these notations, we obtain the following characterization. 
\begin{prop}
\label{twr}
$\Twr$ is the set of K3 twistor cycles.
\end{prop}

\begin{proof}
Let $[C]=\psi(V)\in\Twr$.
We pick any point $p\in C$ and write $p=[x]$ for $x\in\Lambda_{\C}\setminus\{0\}$.
As $V=\overline{V}$, we can write $V=\C x\oplus\C\alpha\oplus\C\overline{x}$ for an $\alpha\in\Lambda_{\R}$ orthogonal to $x$ and $\overline{x}$.
Let $(X,\mu)$ be a marked K3 surface with period point $p$, then $a\defeq \mu_{\R}\inv(\alpha)\in\Hsh{1,1}X{\R}$.
There cannot be a $\delta\in\Delta(X)$ orthogonal to $a$, otherwise we would have $V\subset\mu(\delta)^{\perp}$.
By changing the sign of $\mu$ and composing with an element of the Weyl group $\Weyl p$ if necessary, we can assume $a$ to be a K\"ahler class without changing the period point of $(X,\mu)$.
By construction we have $V=\mu_{\C}\left(\Hdg20X\oplus\C a\oplus\Hdg02X\right)$ and thus $C=C_{X,\mu,a}$ by \Cref{twistor cycle decomposition}.

For the other direction, let $\psi(V)=[C_{X,\mu,a}]\in C_1(\Omega)_{\R}$ be a twistor cycle.
If there was a $\delta\in\Delta$ with $V\subset\delta^{\perp}$, then $\mu_{\R}(a)\perp\delta$ and thus $a\perp\mu\inv(\delta)\in\Delta(X)$, which up to a sign is the sum of Poincar\'e duals of smooth rational curves in $X$.
Hence $a$ could not be a K\"ahler class on $X$.
\end{proof}

\part{IHS families and complex-hyperk\"ahler metrics}

\section{Existence of K3 families with prescribed period map}
\label{section lift}
As already mentioned in the introduction, the K3 period domain $\Omega$ contains many compact complex curves. In fact, this is also true for most IHS period domains. More precisely, there is the following result.

\begin{lemm}\label{Deev}
 Let $\Lambda$ be an IHS lattice of rank at least five. Then, we have $\mathsf{S}_{2d}(\pdom{\Lambda})\neq\emptyset$ for every positive integer $d$. More generally, for \emph{every} compact Riemann surface $C$ of genus $g$ and every integer $d > g$, there exists a holomorphic embedding $C \hookrightarrow \pdom{\Lambda}$ with embedding degree $2d$.
\end{lemm}

\begin{proof}
 The following refines an argument by D\'eev by keeping track of degrees of embeddings, cf.~\cite[p.~376]{Deev}. Also note that the assumption on the rank of the lattice is missing in \emph{loc.~cit.}. 
 
Let $p\in C$ be a general point on $C$ and $d>g$ an integer.
Then there exists a meromorphic function $f\in\Hsh0C{\O_C(dp)}\setminus\Hsh0C{\O_C((d-1)p)}$ on $C$, \cite[p.~273f]{GH94}.
This induces a holomorphic map $f\from C \to \P^1$ of degree $d$.

Let $\P^1 \overset{\imath}{\hookrightarrow}\pdom{\Lambda}$ be the embedding of a twistor cycle into an IHS period domain.
As $\imath(\P^1)$ is the intersection of the projectivization of a (positive) three-space with the quadric $Q_{\Lambda} \subset \P(\Lambda_\C)$, inside $Q_{\Lambda}$, it is a complete intersection of $n_-$ elements of the linear system $\left|\rest{\O_{\P(\Lambda_\C)}(1)}{Q_{\Lambda}}\right|$. Consequently, we have 
 \begin{equation*}
\imath^*(\sT_{\pdom{\Lambda}})  = \imath^*(\sT_{Q_{\Lambda}}) \cong \sT_{\P^1} \oplus \O_{\P^1}(2)^{\oplus n_-}  = \O_{\P^1}(2)^{\oplus(n-2)}. 
 \end{equation*}

Now, we consider the composition $\tilde f\defeq  \imath \circ f\from C \to\pdom{\Lambda} \subset Q_{\Lambda}$.
Then, $\tilde{f}^*(\sT_{\pdom{\Lambda}}) \cong \tilde{f}^*(\O_{\P^1}(2))^{\oplus(n-2)}$, so that we may write $\tilde{f}^*(\sT_{Q_{\Lambda}}) \cong \O_C(D)^{\oplus(n-2)}$ for some effective divisor $D$ of degree $2d$ on $C$.
Then for every two points $p, q \in C$ the assumption on $d$ implies $\deg (D -p -q) > 2g-2$, from which we conclude that 
 \[ \Hsh1C{\tilde{f}^*(\sT_{Q_{\Lambda}})\otimes \O_C(-p-q)} =  \Hsh1C{\O_C (D - p - q)}^{\oplus(n-2)} = \{0\}. \]
 Note that $\dim Q_{\Lambda}=\rk\Lambda-2\ge3\ge\hdim1C{\tilde{f}^*(\sT_{Q_{\Lambda}})\otimes \O_C(-p-q)}+3$, therefore
 \cite[Theorem~II.1.8.2]{KollarRationalCurves} applies to show that the morphism $\tilde{f} = \tilde{f}_0$ can, via a small deformation with fixed source and target, be deformed to an embedding $\tilde{f}_\epsilon\from C \hookrightarrow Q_{\Lambda}$. 
 As the image of $\tilde{f}$ equals $\imath(\P^1) \subset\pdom{\Lambda}$, the image of $\tilde{f}_\epsilon$ is likewise contained in $\pdom{\Lambda}$. 
\end{proof}

It is therefore natural to ask whether these curves carry marked families of K3 surfaces, that is, whether each of these curves gets parametrized by the period map of a marked family. The following results answer this question strongly positively. They will be applied in \Cref{section sheaves}. We use the notation from Section~\ref{K3moduli} throughout. 

\begin{prop}
\label{lift} Let $C$ be a compact Riemann surface, and let $\per\from C\to\Omega$ be holomorphic. Then, for every $p\in C$ and $K\in\pi_0(V^{\circ}_{\per(p)})$ there exists a holomorphic map $\tilde{\per}\from C \to\tilde{\Omega}$ with $\pi\circ\tilde{\per}=\per$ and $\tilde{\per}(p)=(K,p)$.
\end{prop}

\begin{rema}[Uniqueness of the lift and known cases]
We will see in the proof of \Cref{lift} that the lift $\tilde{\per}$ is only unique in cases where $\Delta^{\perp}\defeq\{\delta\in\Delta\mid\forall q\in C\colon\langle\delta,\per(q)\rangle_{\C}=0\}$ is empty.
This situation occurs in particular for all K3 twistor cycles, for which \Cref{lift} has been known classically.
We do not know if \Cref{lift} can be generalized to IHS manifolds besides the case of generic twistor cycles of IHS manifolds, for which the result is proven in \cite[Proposition~5.4]{Huy12}.
\end{rema}

\begin{theo}[Lifting Theorem]
\label{liftfamily}
 Let $C$ be a compact Riemann surface, and let $\per\from C\to\Omega$ be holomorphic. Then, for every $p\in C $ and a marked K3 surface $(X,\mu_X)$ with period point $\per(p)$ there exists a marked family $(\cX,\mu)$ of K3 surfaces over $C$ with associated period map $\per$ and $(\cX_p,\mu_p)\cong(X,\mu_X)$.
\end{theo}
\begin{proof}
The marking $\mu_X$ maps the K\"ahler cone of $X$ onto a Weyl chamber $K_p\in\pi_0(V^{\circ}_{\per(p)})$. By \Cref{lift} we can choose a holomorphic map $\tilde{\per}\from C\to\tilde{\Omega}$ with $\pi\circ\tilde{\per}=\per$ and $\per(p)=K_p$.
Pulling back the universal family of marked K3 surfaces from $\tilde{\Omega}$ gives a marked family $(\cX,\mu)\defeq\tilde{\per}^*\cM$ over $C$ with period map $\pi\circ\tilde{\per}=\per$.
By the Global Torelli Theorem \cite[Theorem~9.1]{LP80}, the two marked K3 surfaces $(\cX_p,\mu_p)$ and $(X,\mu)$ are isomorphic.
\end{proof}

\begin{proof}[Proof of \Cref{lift}]
The set $\Delta^{\perp}=\{\delta\in\Delta\mid\forall q\in C\colon\langle\delta,\per(q)\rangle_{\C}=0\}$ is contained in $\Delta_{\per(q)}$ for all $q\in C $.
As for each $\delta\in\Delta\setminus\Delta^{\perp}$ the hypersurface $\P(\delta^\perp)$ intersects $\per(C)$ only in finitely many points, the set $N\defeq \{q\in C \mid\Delta^{\perp}\subsetneq\Delta_{\per(q)}\}$ is countable.
For each $q'\in C \setminus N$ we have $\Delta_{\per(q')}=\Delta^{\perp}$.
The given Weyl chamber $K_p\in\pi_0(V^{\circ}_{\per(p)})$ induces a partition $\Delta_{\per(p)}=\Delta^+\dot{\cup}\Delta^-$ with property \eqref{partition}, where
$$\Delta^+\defeq\{\delta\in\Delta_{\per(p)}\mid\forall\kappa\in K_p\colon\langle\kappa,\delta\rangle_{\R}>0\}.$$
This restricts to a partition $P$ of $\Delta^{\perp}=(\Delta^{\perp}\cap\Delta^+)\dot{\cup}(\Delta^{\perp}\cap\Delta^-)$ automatically satisfying property \eqref{partition} as well.
For every $q'\in C \setminus N$ this partition of $\Delta^{\perp}$ is induced by a unique Weyl chamber $K_{q'}\in\pi_0(V^{\circ}_{\per(q')})$.

As $\pi\from\tilde{\Omega}\to\Omega$ is locally biholomorphic, for every $q\in C$ there is an open neighborhood $V_q$ of $\per(q)$ in $\Omega$ admitting a continuous section $\sigma_q$ of $\pi$ over $V_q$.
Then $\sigma_q\circ\per$ is a continuous lift of $\per$ over $U_q\defeq\per\inv(V_q)$.
By continuity of $\sigma_q$, all Weyl chambers $\sigma_q(q')$ for $q'\in U_q$ induce the same partition $P'$ of $\Delta^{\perp}$ as $\sigma_q(q)$.
As $ C \setminus N$ is dense in $ C $ we can pick a $q_0\in U_q\setminus N$.
The Weyl group $\Weyl{\per(q_0)}$ acts transitively on the set of Weyl chambers $\pi_0(V^{\circ}_{\per(q_0)})$, hence there is an $s\in \Weyl{\per(q_0)}$ that maps $\sigma_q(q_0)$ to the Weyl chamber $K_{q_0}$ corresponding to the partition $P$.
As $\Delta_{\per(q_0)}=\Delta^{\perp}\subset\Delta_{\per(q')}$ for all $q'\in U_q$, we have $s\in \Weyl{h(q')}$ for all $q'\in U_q$.
Therefore $\tilde{\per}_q\defeq s\circ\sigma_q\circ \per$ is a continuous lift of $\per$ over $U_q$ with $\tilde{\per}_q(q')=K_{q'}$ for all $q'\in U_q\setminus N$.

Let $\cU$ be the open covering of $ C $ by the $U_q$ for $q\in C$.
By the following \Cref{toplem}, $\cU$ has a refinement by a covering $\cA$ of $ C $ by closed subsets, such that each two different $A,B\in\cA$ do intersect at most outside of $N$.
As the maps $\tilde{\per}_q$ are continuous and do at most differ at the points in $N$, by the pasting lemma \cite[Theorem~18.3]{Munkres} they glue to a continuous lift $\tilde{\per}\from C \to\tilde{\Omega}$ of $\per$.
As $\pi$ is locally biholomorphic, every continuous local section $s\circ\sigma_q$ is holomorphic, hence $\tilde{\per}$ fulfills all the desired properties.
\end{proof}

\begin{lemm}[Coverings of surfaces by closed subsets]\label{toplem}
Let $C$ be a closed orientable surface of genus $g$.
Let $\cU$ be an open covering of $C$ and $N\subset C$ an at most countably infinite subset.
Then, there is a covering $\cA$ of $C$ by closed subsets with the following two properties:
\begin{enumerate}
\item For every $A\in\cA$ we have $A\subset U$ for some $U\in\cU$.
\item For each two different $A,B\in\cA$ we have $A\cap B\subset C\setminus N$.
\end{enumerate}
\end{lemm}

To give an easy first example of \Cref{toplem}, in the case $g=0$ we can find a sufficiently fine subdivision of the sphere $C\cong S^2$ via circles of longitude and latitude not meeting $N$.

\begin{proof}[Proof of \Cref{toplem}]
Via a fixed homeomorphism, we identify $C$ with an ``upright'' embedded surface $S\subset\R^3$ of genus $g$, such that the standard height function $\varphi$ given by projecting onto the third coordinate is Morse.
Then, $\varphi$ has $2g+2$ critical points $p_1,\ldots,p_{2g+2}$ with values $c_j\defeq \varphi(p_j)$, where we denote by $p_1$ the minimum, by $p_2,\ldots,p_{2g+1}$ the saddle points and by $p_{2g+2}$ the maximum, cf.~\cite[p.~36, Figure~2.10]{Audin}. We may assume that none of the critical points is contained in $N$.
Equipping $S$ with the metric induced by the Euclidean metric on $\R^3$, for a given open covering $\cU$ of $S$, there exists a Lebesgue number. 
That is, there is an $\epsilon>0$ such that for every $p\in S$ the metric ball $B_{\epsilon}(p)$ is completely contained in one of the $U\in\cU$, see e.g.~\cite[Lemma~27.5]{Munkres}.

Then, as $N$ is countable, only countably many level sets $\varphi^{-1}(c)$ for $c\in (c_1, c_{2g+2})$ or (closures of) flowlines for $-\mathrm{grad}(\varphi)$ contain points of $N$.
Hence, via finitely many level sets and closures of flowlines not containing any points of $N$, we can subdivide $S$ into a set $\cA$ of closed subsets $A\subset S^2$ that all have diameter less than $\epsilon$.
In this subdivision, the resulting neighborhood of $p_j$ for $j=1$ is of the form $\varphi^{-1}([c_1, c_1 + \delta])$, for $j=2,\ldots,2g+1$ it is cut out by some flowlines together with $\varphi^{-1}(c_j -\delta)$ and $\varphi^{-1}(c_j +\delta)$, and for $j=2g+2$ it is of the form $\varphi^{-1}([c_{2g+2} - \delta, c_{2g+2}])$, cf.~\cite[p.~25, Figure~2.2]{Audin}.
By construction, every $A\in\cA$ is contained in a $U\in\cU$, and each intersection of two different $A,B\in\cA$ is a subset of one of the chosen level sets or flowlines, which by selection do not meet $N$.
\end{proof}

\section{Second level families and second level period maps}
\label{slf section}

To make precise what we mean by deforming families of IHS manifolds, we introduce a category that we call \emph{second level families}. Using this notion we spell out some consequences of the unobstructedness result obtained in \cite{BKS}.
The construction of a holomorphic \emph{second level period map} then allows us to prove analogs in our current setup of the Extension Theorem obtained in \cite{BKS} and of the Local Torelli Theorem. Finally we present the construction of a moduli space of marked \emph{embedded} IHS families; this is one of the key foundational results of our paper. The section is concluded by some applications of the construction of second level period maps. 

\subsection{Deformations of families of IHS manifolds}\label{DefFamiliesIHS}

Let $\cF_0=(X_0\stackrel{f_0}{\to}C_0)$ be a family of compact complex manifolds.
According to Ran's definition \cite[Definition~1.1]{Ran89}, a \emph{(small) deformation of the morphism $f_0$} over a complex space $B$ with a distinguished point $0\in B$ consists of two (small) deformations $X\stackrel{g}{\to}B$, $C\stackrel{h}{\to}B$ of $X_0$ and $C_0$ over $B$, such that $g=h\circ f$ for a holomorphic map $X\stackrel{f}{\to}C$ inducing $f_0$ on the central fibers.
Note that, as $f_0$ is a submersion, $f$ is a submersion near the central fiber \cite[Lemma~4.4]{BKS}.
Hence, up to shrinking the base, a deformation of $f_0$ induces a deformation of the family $\cF_0$ in the following sense.

\begin{defi}[Deformations of families]
\label{defslf}
Let $\cF_0=(X_0\stackrel{f_0}{\to}C_0)$ be a family of compact complex manifolds.
A \emph{deformation of the family $\cF_0$} over $B$ is a deformation $X\stackrel{f}{\to}C\stackrel{h}{\to}B$ of $f_0$ such that $(X\stackrel{f}{\to}C)$ is a family of compact complex manifolds.
\end{defi}

\begin{rema}[Representation of the functor of small deformations of families]
When $\cF_0=(X_0\stackrel{f_0}{\to}C_0)$ is a family of compact complex manifolds, then the functors of small deformations of $f_0$ and of $\cF_0$ are isomorphic.
\end{rema}

\begin{defi}[Second level families]
Let $X\stackrel{f}{\to}C\stackrel{h}{\to}B$ be a deformation of a family $\cF_0$ of compact complex manifolds.
We call $\cF_B\defeq(X\stackrel{f}{\to}C\stackrel{h}{\to}B)$ a \emph{second level family over $B$} if $(X\stackrel{f}{\to}C)$ is a family of IHS manifolds and $(C\stackrel{h}{\to}B)$ is a family of smooth rational curves.
A \emph{marking of $\cF_B$} is a marking of $(X\stackrel{f}{\to}C)$.

A morphism of two (marked) second level families $(X\stackrel{f}{\to}C\stackrel{h}{\to}B)$, $(X'\stackrel{f'}{\to}C'\stackrel{h'}{\to}B')$ consists of holomorphic maps $\alpha\from X\to X',\beta\from C\to C',\gamma\from B\to B'$ such that $(\alpha,\beta)$ and $(\beta,\gamma)$ are morphisms of (marked) families of IHS manifolds or of smooth rational curves, respectively.
For each $b\in B$ we denote $C_b\defeq h\inv(b)$ and $X_b\defeq f\inv(C_b)$.
Note that the restriction $\cF_b\defeq (\rest{f}{X_b}\from X_b\to C_b)$ of $\cF_B$ over $b\in B$ is a family of IHS manifolds over $C_b\cong\P^1$.
A marking $\mu$ of $\cF_B$ restricts over each $b\in B$ to a marking $\mu_b$ of $\cF_b$.
\end{defi}

\begin{defi}[Degree of IHS families]
\label{degree}
When $\cF_0=(X_0\stackrel{f_0}{\to}C_0)$ is a family of IHS manifolds over a smooth rational curve, we call $\deg\cF_0\defeq-\deg(f_*\Omega^2_{X_0/C_0})$ its \emph{degree}.
We call $\cF_0$ an \emph{embedded family (of degree $d$)} if moreover for a marking $\mu_0$ the period map of $(\cF_0,\mu_0)$ is an embedding.
When $\cF_B$ is a second level family for which all fibers are (embedded) families of IHS manifolds of degree $d$, we call $\cF_B$ an \emph{(embedded) second level family of degree $d$}.
Instead of degree $2$ we also say \emph{quadratic}.
\end{defi}

\begin{rema}
\label{degree of period map}
The degree of a marked family of IHS manifolds over a smooth rational curve $C$ with period map $\per$ equals the embedding degree of $\per(C)$ in $\pdom{\Lambda}$, \cite[Proposition~3.3]{BKS}.
The degree of the fibers $\cF_b$ of a second level family $\cF_B$ is locally constant in $b\in B$, \cite[Corollary~4.2(ii)]{BKS}.
Every twistor family is quadratic, and every quadratic family of IHS manifolds is an embedded family of IHS manifolds, \cite[Example~3.11 and Proposition~3.12]{BKS}, hence also every quadratic second level family is embedded.
\end{rema}

\begin{rema}[Second level families as small deformations of IHS families over $\P^1$]
\label{slf}
Let $X\stackrel{f}{\to}C\stackrel{h}{\to}B$ be a deformation of a family $\cF_0=(X_0\stackrel{f_0}{\to}C_0)$ of IHS manifolds, where $C_0\cong\P^1$.
Then $\cF_B=(X\stackrel{f}{\to}C\stackrel{h}{\to}B)$ is up to shrinking $B$ a second level family.
On the other hand, $\cF_B$ can be considered as a deformation of each of its fibers $\cF_b$ for $b\in B$.
Considering $\cF_0$ as a second level family over a reduced point, we can thus understand a 
small deformation of $\cF_0$ as a morphism $(\cF_0\to\pt)\to\cF_B$ of second level families.
Similarly, when $\mu_0$ is a marking of $\cF_0$, by a small deformation of $(\cF_0,\mu_0)$ we understand a morphism $(\cF_0\to\pt,\mu_0)\to(\cF_B,\mu)$ of marked second level families of IHS manifolds.
Moreover, by \Cref{degree of period map}, when $\cF_0$ is embedded, $\cF_B$ can be assumed to be embedded as well (of the same degree).
\end{rema}

\begin{prop}[Universality of marked (second level) families]
\label{deformation of marking}
\mbox{}
\begin{enumerate}
\item\label{part a} A marked family $(X\stackrel{f}{\to}B,\nu)$ of IHS manifolds is a universal deformation of one of its fibers $(X_b,\nu_b)$ if and only if $X$ is a universal deformation of~$X_b$.
\item\label{part b} An embedded second level family $(\cF_B,\mu)$ is a universal deformation of one of its fibers $(\cF_b,\mu_b)$ if and only if $\cF_B$ is a universal deformation of $\cF$.
\end{enumerate}
\end{prop}
In suggestive notation, we can equivalently state this as $\df(X)=\df(X,\nu)$ and $\df(\cF_B)=\df(\cF_B,\mu)$.
Of course, part~\ref{part a} is well-known and we mainly want to point out its analogy to part~\ref{part b}.

\begin{proof}[Proof of \Cref{deformation of marking}]
To prove \ref{part a}, let $(\cF',\nu')=(X'\stackrel{f'}{\to}B',\nu')$ be another deformation of $(X_b,\nu_b)$. If $\cF=(X\stackrel{f}{\to}B)$ is a universal deformation of $X_b$, then up to shrinking $B'$ the morphism $(X_b\to\{b\})\inj\cF$ factors through $(X_b\to\{b\})\inj\cF'$ via a morphism $\phi\from\cF'\to\cF$ of families.
Up to shrinking $B$ and $B'$ to simply connected neighborhoods of $b$, any given marking $\nu_b$ of $X_b$ can be uniquely extended to a marking $\nu$ of $\cF$, and a marking of $\cF'$ that must coincide with $\nu'$.
The uniqueness implies $\phi^*\nu=\nu'$, making $\phi$ a morphism of marked families.

For the other direction, assume that $(\cF,\nu)$ is a universal deformation of $(X_b,\nu_b)$.
Then we can consider every universal morphism $\phi\colon(\cF',\nu')\to(\cF,\nu)$ as a morphism $\cF'\to\cF$ of the unmarked families.
The uniqueness does still hold, again because $\nu$ and $\nu'$ are locally near $b$ the unique extended markings of $\nu_b$.

The argument to prove \ref{part b} is analogous to the proof of \ref{part a}.
For this, let $\cF_B=(X\stackrel{f}{\to}C\stackrel{h}{\to}B)$ and note that by shrinking $B$ to a simply connected neighborhood, we also make $C$ a simply connected neighborhood of $C_b$.
Then every marking of $\cF_b$ extends uniquely to a marking of $\cF_B$.\qedhere
\end{proof}

\begin{rema}[Relative symplectic forms]
\label{rel sympl}
Let $\cF=(X\stackrel{f}{\to}C)$ be an embedded family of IHS manifolds.
Then $\deg\cF=d$ is by the projection formula equivalent to $f_*(\Omega^2_{X/C}\otimes f^*\O_C(d))\cong f_*\Omega^2_{X/C}\otimes\O_C(d)\cong\O_C$.
Hence there is a nowhere vanishing section $\omega\in\Hsh0X{\Omega^2_{X/C}\otimes f^*\O_C(d)}$, which we call a \emph{relative symplectic form} on $X$.
By construction $\omega$ restricts to a symplectic form on each fiber of $f$. Similarly, for an embedded second level family $\cF_U=(\cX\stackrel{F}{\to}\cC\stackrel{\pi}{\to} U)$ a relative symplectic form on $\cX$ is a section $\omega\in\Hsh0{\cX}{\Omega^2_{\cX/\cC}\otimes F^*\O_{\cC}(d)}$ that restricts to a relative symplectic form on each fiber of $\cF_U$, where $\O_{\cC}(d)$ denotes the relative $\O(d)$-bundle on the $\P^1$-bundle $\cC\to U$. When $U$ is connected, $\deg\cF_p$ is by \Cref{degree of period map} constant in $p\in U$.
Hence $\sL\defeq F_*(\Omega^2_{\cX/\cC}\otimes F^*\O_{\cC}(d))$ is a holomorphic line bundle that is trivial on every fiber of $\pi$, and thus $\sL\cong\pi^*\pi_*\sL$.
Hence the relative symplectic forms on $\cX$ correspond to the pulled back nowhere vanishing sections of $\pi_*\sL$.

\end{rema}

\subsection{The second level period map}

Like in Section~\ref{cycle section}, let $\cC_d(\pdom{\Lambda})\stackrel{\pr1}{\lto}\mathsf{S}_d(\pdom{\Lambda})\subset\dou(\pdom{\Lambda})$ denote the universal family of smooth rational curves of embedding degree $d$ in an IHS period domain $\pdom{\Lambda}$.

\begin{prop}[Second level period map]
\label{periodmap2}
Let $\cF_B\defeq(X\stackrel{f}{\to}C\stackrel{h}{\to}B)$ be an embedded second level family of degree $d$ with a marking $\mu$ and period map $\per$.
Then $$\Per\from B\to \mathsf{S}_d(\pdom{\Lambda})\subset\dou(\pdom{\Lambda}),\quad b\mapsto [\per(C_b)]$$ defines a holomorphic map, which we call the \emph{second level period map} of $(\cF_B,\mu)$.
More precisely, $((\Per\circ h,\per),\Per)\from(C\stackrel{h}{\to}B)\to(\cC_d(\pdom{\Lambda})\stackrel{\pr1}{\to}\mathsf{S}_d(\pdom{\Lambda}))$ defines a morphism of families of smooth rational curves.
\end{prop}

\begin{rema}
We expect \Cref{periodmap2} to hold also for non-embedded families.
The restriction is for the sake simplicity, as we are mostly interested in the case $d=2$, where the assumption always holds by \Cref{degree of period map}.
\end{rema}

\begin{proof}[Proof of \Cref{periodmap2}.]
To show that $\Per$ is holomorphic at a given point $0\in B$, we may shrink $B$ around $0$.
We obtain a holomorphic map $(h,\per)\from C\to B\x\pdom{\Lambda}$.
It is injective as we assumed the $\rest{\per}{C_b}$ for $b\in B$ to be embeddings.
As the $C_b$ are compact, we may up to shrinking $B$ assume $(h,\per)$ to be proper and thus, by Remmert's Proper Mapping Theorem, its image to be a complex subspace $Y\subset B\x\pdom{\Lambda}$.
As by assumption $h$ is a submersion and $d\per$ has at least rank $1$, the differential $d(h,\per)$ is everywhere of maximal rank and thus $(h,\per)$ is an embedding.
Hence the projection $\pr1\from Y\to B$ is like $h$ a proper submersion.
Thus the universal property of the Douady space induces a holomorphic map $\alpha\from B\to\dou\pdom{\Lambda}$.
As $\cF_b$ is an embedded family of degree $d$ for all $b\in B$, $\alpha$ actually maps to the (smooth) subspace $\mathsf{S}_d(\pdom{\Lambda})\subset\dou(\pdom{\Lambda})$.
Then the universal property states that $(\alpha\x\id{\pdom{\Lambda}},\alpha)\from(Y\stackrel{\pr1}{\to}B)\to(\cC_d(\pdom{\Lambda})\stackrel{\pr1}{\to}\mathsf{S}_d(\pdom{\Lambda}))$ is a morphism of families of smooth rational curves.
Hence $((\alpha\circ h,\per),\alpha)\from(C\stackrel{h}{\to}B)\to(\cC_d(\pdom{\Lambda})\stackrel{\pr1}{\to}\mathsf{S}_d(\pdom{\Lambda}))$ is also a morphism of families.
By construction, the point $\per(c)$ has to lie on the cycle $\alpha(h(c))$ for each $c\in C$, from which we conclude $\alpha=\Per$.
\end{proof}

\begin{rema}[Functoriality of the second level period map under pullbacks]
\label{functoriality2}
For an embedded second level family $\cF_B=(X\stackrel{f}{\to}C\stackrel{h}{\to}B)$ and a holomorphic map $\gamma\from B'\to B$ of complex spaces we define the pulled back second level family $\gamma^*\cF_B=(X'\stackrel{f'}{\to}C'\stackrel{h'}{\to}B')$ via $C'\defeq C\x_B B'$ with projections $\beta,h'$ onto $C,B'$, and $X'\defeq X\x_C C'$ with projections $\alpha,f'$ onto $X,C'$.
Each marking $\mu$ of $(X\stackrel{f}{\to}C)$ induces a marking $\beta^*\mu$ of $(X'\stackrel{f'}{\to}C')$.
Then $(\alpha,\beta,\gamma)$ is a morphism of (marked) embedded second level families.
Let $\per$ be the period map of $(X\stackrel{f}{\to}C,\mu)$ and $\Per$ the second level period map of $(\cF_B,\mu)$.
By the functoriality of the period map, the period map of $(X'\stackrel{f'}{\to}C')$ is $\per\circ\beta$.
Thus the second level period map $\Per'$ of $(\gamma^*\cF_B,\beta^*\mu)$ is $B'\to C_1(\pdom{\Lambda})$, $b'\mapsto[(\per\circ\beta)(C'_{b'})]=[\per(C_{\gamma(b')})]=(\Per\circ\gamma)(b')$, so $\Per'=\Per\circ\gamma$.
\end{rema}

\subsection{Gluing universal morphisms of families}
\label{Glueing universal morphisms of families}
The following technical result seems to be well-known, cf.~\cite[Chapter~I, Remark after (10.6)]{BHPV}; we include a proof owing to the lack of a reference. 

\begin{prop}
\label{glueuniv}
Let $S\subset D$ be a complex subspace of a complex manifold and $\{U_i\}_{i\in I}$ an open covering of an open neighborhood of $S$ in $D$.
Assume that we are given on each $U_i$ a family $\cF^i$ of compact complex manifolds that is universal at every $s\in S\cap U_i$ and with $\hdim0{X_s}{\sT_{X_s}}=0$.
If the restrictions $\rest{\cF^i}{U_i\cap S}$ glue together to a family of complex manifolds over $S$, then (up to going over to a locally finite refinement) restrictions of the families $\cF^i$ glue together over an open neighborhood of $S$ in $D$.
\end{prop}

\begin{proof}
We can go over to a locally finite refinement $(V_i)_{i\in I}$ of $(U_i)$ with $\overline{V_i}\subset U_i$ for all $i\in I$, \cite[Lemma~41.6]{Munkres}.
For all $i,j,k\in I$ we denote $V_{ij}\defeq V_i\cap V_j$ and $V_{ijk}\defeq V_i\cap V_j\cap V_k$.
Then for all $i,j\in I$ the families $\cF^i,\cF^j$ glue together over an open neighborhood $W_{ij}$ of $S\cap V_{ij}$ via a map $\phi_{ji}$ that is unique up to shrinking $W_{ij}$ around $S\cap W_{ij}$, \cite[Theorem~6.3(i)]{BKS}.
By this uniqueness and as the covering is locally finite, each $s\in S$ has an open neighborhood where $\phi_{kj}\circ\phi_{ji}=\phi_{ki}$ for every $i,j,k\in I$ with $s\in V_{ijk}$.
Hence, up to shrinking the $W_{ij}$ around $S\cap W_{ij}$, we can assume that the $\phi_{ij}$ satisfy the cocycle condition.
Then, for each $i\in I$, we set $\tilde{V}_i\defeq V_i\setminus\bigcup_{j\in I}\overline{V_j\setminus W_{ij}}$, which is open in $D$ because the refinement $(V_i)$ was chosen to be locally finite.
By construction the $\rest{\cF^i}{\tilde{V}_i}$ glue together to a family over $\bigcup_{i\in I}\tilde{V}_i$.
\end{proof}
For later reference we note the following well known result.
\begin{coro}
\label{refined}
If two marked families of K3 surfaces over the same base space have the same refined period map, they are isomorphic as marked families.
In particular, for every marked family of K3 surfaces $(\cF,\mu)$ with refined period map $\tilde{\per}$ we have $(\cF,\mu)\isom\tilde{\per}^*\cM$.\qed
\end{coro}


\subsection{Deformations of families of degree $d$ and their total spaces}
\label{section degree d families}

By gluing Kuranishi families of the fibers of $\cF$ via \Cref{glueuniv} the following result was proven in \cite[Chapters~6 and 7]{BKS}.
\begin{theo}[Extension theorem for embedded families]\label{extension}
Let $(\cF,\mu)$ be a $\Lambda$-marked embedded family of IHS manifolds over a complex space $S$ with period map $\per$.
Then $(\cF,\mu)=\per^*(\tilde{\cF},\tilde{\mu})$ for a marked family $(\tilde\cF,\tilde{\mu})$ with period map $U\inj\pdom{\Lambda}$ the inclusion of an open neighborhood $U$ of $\per(S)$ in $\pdom\Lambda$.
Up to shrinking $U$ around $\per(S)$, the family is unique up to isomorphism of marked families.
\end{theo}

The same statement can also be deduced from Markman's result \cite[Theorem~1.1]{Mar21}, see~\cite[Section~1.3]{BKS} for a discussion. The special case where $\cF$ is an embedded family of K3 surfaces follows directly from the fact that a section of the topological sheaf $\tilde{\sF}=(\tilde{\Omega}\stackrel{\pi}{\to}\Omega)$ over a closed subspace $C\subset\Omega$ can be extended to an open neighborhood (cf.~Section~\ref{topsheaves}).

The Unobstructedness Theorem \cite[Theorem~1.1]{BKS} states that the deformations of the total space $X_0$ of a marked embedded family of IHS manifolds are unobstructed, identifying $\df(X_0)$ with a germ of the Douady space of the corresponding period domain. Actually, there is more information coming out of the proof presented in \cite[Section~4]{BKS}, which we can now formulate concisely using the terminology of second level families introduced in Section~\ref{DefFamiliesIHS}. The central new observation is that the universal morphisms constructed as part of the proof glue together to the second level period map.

\begin{theo}[Unobstructedness Theorem for embedded families]
\label{unobs}
Let $(\cF_0,\mu_0)=(X_0\stackrel{f_0}{\to}C_0,\mu_0)$ be a $\Lambda$-marked embedded family of IHS manifolds of degree $d$ with period map $\per_0$.
Then there exists a universal deformation \[{(\cF_U,\mu)=(X\stackrel{f}{\to}\rest{\cC_d(\pdom\Lambda)}U\stackrel{\pr1}{\to}U,\mu)}\] of $(\cF_0,\mu_0)$ over an open neighborhood $U$ of $[\per_0(C_0)]$ in $\mathsf{S}_d(\pdom\Lambda)$ with period maps $\per=\pr2\from\rest{\cC_d(\pdom\Lambda)}U\to\pdom{\Lambda}$ and $\Per\from U\inj \mathsf{S}_d(\pdom\Lambda)$ the inclusion.
Induced universal morphisms $(\cF'_B,\mu')\to(\cF_U,\mu)$ of deformations of $(\cF_0,\mu_0)$ are induced by the second level period map of $(\cF'_B,\mu')$ and are uniquely determined as morphisms of second level families. Moreover, $X\stackrel{\pr1\circ f}{\lto}U$ is a universal deformation of $X_0$.
\end{theo}

\begin{rema}[Universal deformations of (marked) families and their total spaces]
\label{deformation space}
\Cref{unobs} and \Cref{deformation of marking} together yield canonical isomorphisms \[(U,[C_0])\cong\df(\cF_0,\mu_0)\cong\df(\cF_0)\cong\df(X_0).\]
Let $n\defeq \rk(\Lambda)$.
As $U$ is smooth of dimension $m(n,d)\defeq (n-2)(d+1)-3$ by \cite[Corollary~4.2]{BKS}, these deformations are unobstructed.
Specifically, for twistor families we have $d=2$ and thus $m=3n-9=3n_{-}$.
\end{rema} 

\begin{proof}[Proof of \Cref{unobs}]
As the period map $\per_0$ induces an isomorphism $(\cF_0,\mu_0)\cong(\per_0\inv)^*(\cF_0,\mu_0)$, we can assume without loss of generality $C_0\subset\pdom{\Lambda}$ and that $\per_0$ is the inclusion. 
By \Cref{extension} there is a morphism of marked families $(\cF_0,\mu_0)\inj(\tilde{\cF},\tilde{\mu})$, where $\tilde{\cF}=(\tilde{X}\stackrel{\tilde{f}}{\to}V,\tilde{\mu})$ is a marked family with period map $V\inj\pdom{\Lambda}$ the inclusion of an open neighborhood of $C_0$ in $\pdom{\Lambda}$.

Then $U\defeq\{[C_t]\in \mathsf{S}_d(\pdom{\Lambda})\mid C_t\subset V\}$ is an open neighborhood of $[C_0]$ in $\dou(\pdom{\Lambda})$.
We obtain a marked family of IHS manifolds $(\cF,\mu)\defeq \pr2^*(\tilde{\cF},\tilde{\mu})$ with $\cF=(X\stackrel{f}{\to}C)$ over $C\defeq \pr1\inv(U)=\rest{\cC_d(\pdom\Lambda)}U$.
This gives rise to a marked second level family $(\cF_U,\mu)=(X\stackrel{f}{\to}\rest{\cC_d(\pdom\Lambda)}U\stackrel{\pr1}{\to}U,\mu)$ whose restriction over $[C_0]\in U$ is by construction isomorphic to $(\cF_0,\mu_0)$.
By functoriality, the period map of $(\cF_U,\mu)$ is $\id V\circ\pr2=\pr2$.
Hence $(\cF_U,\mu)$ is an embedded second level family of degree $d$ with second level period map $U\inj\dou(\pdom{\Lambda})$ the inclusion.

To show the universality of $(\cF_U,\mu)$, let $(\cF'_B,\mu')=(X'\stackrel{f'}{\to}C'\stackrel{h'}{\to}B,\mu')$ be another deformation of $(\cF_0,\mu_0)$.
We denote the period maps of $(\cF'_B,\mu')$ by $\per'$ and $\Per'$.
By \Cref{periodmap2} we get a morphism $(Q',\Per')\from (C'\stackrel{h'}{\to}B)\to(\cC_d(\pdom{\Lambda})\stackrel{\pr1}{\to}\mathsf{S}_d(\pdom{\Lambda}))$ of families of smooth rational curves, where $Q'=(\Per'\circ h',\per')$.
By the Local Torelli Theorem for IHS manifolds, $(\tilde{\cF},\tilde{\mu})$ is a universal deformation of each of its IHS fibers because its period map is a local biholomorphism.
This induces locally near each $c\in C_0$ universal morphisms $\phi_c\from(\rest{X'}{U_c}\stackrel{f'}{\to} U_c,\mu')\to(\tilde{\cF},\tilde{\mu})$.
By functoriality of period maps, each $\phi_c$ must induce $\per'$ on $C'$.
As $\hdim0{X_c}{\sT_{X_c}}=0$ for all $c\in C_0$, the $\phi_c$ must also be uniquely determined on the total space and thus by \Cref{glueuniv} glue to a morphism of families on an open neighborhood of $C_0$ in $C'$.
As $h'$ is proper, we can up to shrinking $B$ assume that the $\phi_c$ glue to a morphism $\phi\from(X'\stackrel{f'}{\to}C',\mu')\to(\tilde{\cF},\tilde{\mu})$ over $\per'$ that thus must also be uniquely determined on $X'$.
As $\pr2\circ Q'=\per'$, by applying the Cartesian universal property we see that the morphism $\phi$ must factor uniquely through the morphism $(\cF,\mu)\to(\tilde{\cF},\tilde{\mu})$ obtained above. In summary, we constructed a morphism $(\cF'_B,\mu')\to(\cF_U,\mu)$.
The uniqueness of this morphism on $B$ and $C'$ follows from functoriality of the period maps, the uniqueness on $X'$ follows from the uniqueness of $\phi$.
It requires some more work to conclude that $X\stackrel{\pr1\circ f}{\lto}U$ is a universal deformation of the total space $X_0$, see \cite[Section~4.6]{BKS}.
\end{proof}

From \Cref{unobs} we can draw the following two immediate consequences, which we consider to be natural analogs of the Local Torelli Theorem for IHS manifolds and of the Extension Theorem for embedded families of IHS manifolds (\Cref{extension}).

\begin{theo}[Local Torelli Theorem for embedded families]
\label{local torelli}
A marked embedded second level family $(\cF_B,\mu)$ is universal at a point $b\in B$ if and only if its second level period map is a local biholomorphism at $b$.
\end{theo}
\begin{proof}
By \Cref{unobs} the second level period map $\Per\from B\to\dou(\pdom{\Lambda})$ induces up to shrinking $B$ around $b$ a universal morphism from $(\cF_B,\mu)$ to a universal deformation of $(\cF_b,\mu_b)$.
Hence, $(\cF_B,\mu)$ is universal at $b$ if and only if $\Per$ is a local isomorphism at $b \in B$.
\end{proof}

\begin{theo}[Extension Theorem for embedded second level families]
\label{extension2}
Let $(\cF_B,\mu)$ be a $\Lambda$-marked embedded second level family of degree $d$ whose second level period map $\Per$ is an embedding.
Then $(\cF_B,\mu)=\Per^*(\cF'_U,\mu')$ for a second level family $\cF'_U$ with second level period map the inclusion of an open neighborhood $U$ of $\Per(B)$ in $\mathsf{S}_d(\pdom{\Lambda})$.
Up to shrinking $U$ around $\Per(B)$, the family $(\cF'_U,\mu')$ is unique up to isomorphism of marked second level families.
\end{theo}

\begin{proof}
As $\Per$ induces an isomorphism $(\cF_B,\mu)\cong(\Per\inv)^*(\cF_B,\mu)$, we can without loss of generality assume $B\subset \mathsf{S}_d(\pdom{\Lambda})$ and that $\Per$ is the inclusion.

We first consider the situation locally at a given point $b\in B$.
By applying \Cref{unobs} to $(\cF_b,\mu_b)$ we obtain a second level family $(\cF'_{U_b},\mu'_{U_b})$ over an open neighborhood $U_b$ of $b$ in $\mathsf{S}_d(\pdom\Lambda)$ with second level period map the inclusion $U_b\inj \mathsf{S}_d(\pdom\Lambda)$.
By \Cref{local torelli} it is a universal deformation of each of its fibers, in particular of the central fiber $(\cF'_b,\mu'_b)\cong(\cF_b,\mu_b)$.
Up to shrinking $U_b$ around $b$, we thus have $(\cF'_{U_b},\rest{\mu'_{U_b})}{B\cap U_b}\cong\rest{(\cF_B,\mu)}{B\cap U_b}$, and by the functoriality of the second level period map this isomorphism fixes the base.

Now, by varying $b\in B$, we obtain an open covering of an open neighborhood $U=\bigcup_{b\in B}U_b$ of $B$ in $\mathsf{S}_d(\pdom{\Lambda})$ with marked second level families $(\cF'_{U_b},\mu'_{U_b})$ over $U_b$ for each $b\in B$.
By the universality of these families, for each two $b,b'\in B$ the two marked families are on the intersection $U_{bb'}$ locally isomorphic over the base.
These local isomorphisms are by \Cref{unobs} uniquely determined as morphisms of marked second level families and thus glue together to a unique gluing of $(\cF'_{U_b},\mu'_{U_b})$ and $(\cF'_{U_{b'}},\mu'_{U_{b'}})$ over $U_{bb'}$.
Moreover, by the uniqueness of the local isomorphisms, such gluings must fulfill the cocycle condition and thus the $(\cF'_{U_b},\mu'_{U_b})$ glue uniquely together to a marked family $(\cF'_U,\mu')$ over $U$.
From the properties of the local families $(\cF'_{U_b},\mu'_{U_b})$ we can deduce that $\rest{(\cF'_U,\mu')}B$ is over $B$ isomorphic to $(\cF_B,\mu)$, and the uniqueness of the gluing implies the uniqueness of $(\cF'_{U},\mu')$ up to shrinking $U$ around~$B$.
\end{proof}

 In Section~\ref{section sheaves}, we will establish stronger functoriality results for embedded second level families of K3 surfaces. 

\subsection{Gluing local universal second level families to obtain moduli spaces}\label{glueing}

Let $\Lambda$ be an IHS lattice.
Recalling the construction of the Burns-Rapoport space $\tilde{\Omega}$, a natural follow-up step to \Cref{unobs} is to try to glue the universal families obtained there over $\mathsf{S}_d(\pdom{\Lambda})$, which plays the role of a period domain for our deformation problem.
We will indeed proceed in a way similar to \cite[Proposition~1]{Bea81} to construct a coarse moduli space $\sM_{\Lambda}$ of $\Lambda$-marked embedded families of IHS manifolds. In the K3 surface case, the local families will glue to a global universal family on $\sM=\sM_{\Lambda}$, making $\sM$ a fine moduli space; for this to work the following well-known property is crucial, see e.g.~\cite[Proposition 6]{MR785227} or \cite[Chapter~15, Proposition~2.1]{K3book} for a proof.

\begin{prop}
\label{automorphism property}
Let $X$ be a K3 surface. Then the identity $\id X$ is the only biholomorphic automorphism of $X$ acting trivially on $\Hsh2X{\Z}$.\qed
\end{prop}

This property does not generalize to all IHS manifolds (we discuss this point later in \Cref{ihsm automorphism property}).
The following lemma will serve as an analog of \Cref{automorphism property} for marked embedded families of K3 surfaces.

\begin{lemm}
\label{marked automorphisms}
Let $(\cF,\mu)=(X\stackrel{f}{\to}C,\mu)$ be a marked family of K3 surfaces with injective period map.
Then every marked automorphism of $(\cF,\mu)$ inducing the identity on $\Rs2f{\csh\Z X}$ is trivial.
\end{lemm}
\begin{proof}
Let $\per\from C\inj\Omega$ be the period map of $(\cF,\mu)$, and let $(g,h)\in\Aut_{\O}(\cF,\mu)$ be a marked automorphism inducing the identity on $\Rs2f{\csh\Z X}$.
By the functoriality of the period map we must have $\per\circ h=\per$ and thus $h=\id C$ since $\per$ is injective.
Hence, $g$ induces over every point $p\in C$ a marked automorphism $\rest{g}{X_p}$ of the fiber $(X_p,\mu_p)$.
As $\rest{\mu_p\circ g}{X_p}^*=\mu_p$, the automorphism $\rest{g}{X_p}^*$ of $\Hsh2{X_p}{\Z}$ must be trivial.
Therefore, by \Cref{automorphism property}, $g$ needs to be the identity on each fiber and thus $g=\id X$.
\end{proof}

\begin{lemm}
\label{unique isomorphism}
If two marked families of K3 surfaces with injective period map are isomorphic, the isomorphism is unique.
\end{lemm}
\begin{proof}
For $i=1,2$, let $(g_i,h_i)\from(X\stackrel{f}{\to}C,\mu)\to(X'\stackrel{f'}{\to}C',\mu')$ be two isomorphisms of marked families with injective period maps $\per,\per'$.
By functoriality and injectivity of the period maps, we have $\per=\per'\circ h_i$ and thus the $h_i$ are equal.
Hence $g_2\circ g_1\inv$ acts trivially on $\Rs2f{\csh\Z X}$ and thus needs to be the identity on $X$ by \Cref{marked automorphisms}.
\end{proof}

\begin{theo}[Moduli spaces of marked embedded families]
\label{fine}
For every IHS lattice $\Lambda$ there exists a complex space\footnote{in the sense of \cite{GrothendieckAnalyticSpaces}, a priori not necessarily Hausdorff} $\sM_{\Lambda}$ that is a coarse moduli space of $\Lambda$-marked embedded families of IHS manifolds and has a locally biholomorphic map $P\colon\sM_{\Lambda}\to \mathsf{S}(\pdom{\Lambda})$ to the space of smooth rational curves in $\pdom{\Lambda}$. In the K3 surface case, this moduli space, which we denote by $\sM$, is actually a fine moduli space, and $P\colon\sM\to \mathsf{S}(\Omega)$ is (in every degree) the second level period map of the universal family on $\sM$.
\end{theo}

\begin{proof} We are now in a position to implement the strategy employed in \cite[Proposition~1]{Bea81}:
Let $\sM_{\Lambda}$ be the set of isomorphism classes of $\Lambda$-marked embedded families $(\cF,\mu)$ of IHS manifolds.
By definition this induces a set-theoretic map $P\colon\sM_{\Lambda}\to\mathsf{S}(\pdom{\Lambda})$, mapping each $(\cF,\mu)$ to the smooth rational curve corresponding to the image of its period map.
By \Cref{unobs}, each such family has a universal deformation $(\cF_U,\mu_U)$ over a smooth open subspace $U\subset \mathsf{S}_d(\pdom{\Lambda})$ with second level period map the inclusion, where $d=\deg\cF$.
By the functoriality of the second level period map, the fibers of $(\cF_U,\mu_U)$ are pairwise non-isomorphic as marked families.
By \Cref{local torelli} the $(\cF_U,\mu_U)$ are universal at every fiber, hence for each two of them the space of isomorphic fibers is open in their bases.
Thus, the maps $U\to\sM_{\Lambda}$, $s\mapsto[(\cF_s,\mu_s)]$ are local sections of $P$ and thus give $\sM_{\Lambda}$ the structure of a (potentially non-Hausdorff) complex space in the sense of \cite{GrothendieckAnalyticSpaces} for which $P$ becomes locally biholomorphic.

In particular, in the K3 surface case the $(\cF_U,\mu_U)$ glue together to a global family over $\sM=\sM_{\Lambda}$ by \Cref{unique isomorphism}.
The universal properties (of a coarse or fine moduli space, respectively) follow from those of the local families $(\cF_U,\mu_U)$.
In the K3 surface case, by the construction of the maps $U\to\sM$ the map $P$ must be the second level period map of the universal family on $\sM$, as this must be true over each of the charts $U$.
\end{proof}

In fact, the spaces $\sM_{\Lambda}$ just constructed have further modular properties.

\begin{theo}[Moduli spaces of marked total spaces]
\label{total spaces}
The space $\sM_{\Lambda}$ constructed in \Cref{fine} is also a coarse moduli space for the total spaces $X$ of embedded families of IHS manifolds together with markings on $\Hsh2X{\Z}$.
In the K3 surface case, $\sM$ is a fine moduli space also for this deformation problem.
\end{theo}
Recall that total spaces $X$ of IHS families are not K\"ahler, hence the usual notions of period maps and period domains cannot be used. This makes it difficult to proceed with Beauville's strategy from \cite[p.~141f]{Bea81} directly.
However, we know that the second level families of \Cref{unobs} also define universal deformations of the total spaces of their fibers. Together with \Cref{fine} we therefore only need the following two results \Cref{unique fiber structure} and \Cref{automorphisms} to establish the claim of \Cref{total spaces}.

\begin{lemm}[Automorphisms preserve fibration structure]
\label{unique fiber structure}
Let $\cF=(X\stackrel{f}{\to}C)$ be an embedded family of IHS manifolds.
The structure on $X$ as a family of IHS manifolds is unique up to automorphisms of $C$.
As a consequence, every biholomorphic automorphism of $X$ is induced by a $C$-automorphism of $\cF$.
\end{lemm}
\begin{proof}
As $\cF$ is an embedded family, say of degree $d\defeq\deg\cF$, by choosing a marking $\mu$ of $\cF$ we may assume without loss of generality that $C$ is contained in an IHS period domain $\pdom{\Lambda}$ with embedding degree $d$, and that the period map of $(\cF,\mu)$ is the inclusion.

Let $\cF'=(X\stackrel{h}{\to}C)$ denote an arbitrary structure on $X$ as a family of IHS manifolds over $C$.
We want to show that $f$ and $h$ define the same fibration structure on $X$.
For this, let $(\cF_U,\mu_U)=(\cX\stackrel{F}{\to}\cC\to U)$ be a universal deformation of $(\cF,\mu)$ as obtained in \Cref{unobs}, where $\cC\defeq\rest{\cC_d(\pdom{\Lambda})}U$ and $U$ is an open neighborhood of $[C]$ in $\mathsf{S}_d(\pdom{\Lambda})$.
The map $F$ extends the morphism $f$ on the central fiber $(\cF_{[C]},\mu_{[C]})=(\cF,\mu)$.
The composition $\cX\to U$ can also be viewed as a deformation of the total space $X$.
Due to Kodaira's classical result \cite[Theorem~4]{Kod63} the second fibration structure of $X$ into IHS manifolds must also be stable under small deformations; that is, up to shrinking $U$ around $[C]$, there is a proper holomorphic map $H\from \cX\to \cC$ over $U$ extending $h$ on the central fiber.

As $h$ is proper, every fiber $X_p$ of $f$ gets mapped by $h$ either to a point or surjectively onto $C$.
We want to show that the latter case cannot occur.
To see this, note first that $H$ is also proper, such that more generally $H$ must map each fiber $\cX_{([C'],p')}$ of $F$ either to a point or surjectively onto the curve $\cC_{[C']}$.
Aiming for a contradiction, suppose that $H(\cX_{([C],p)})=h(X_p)=C$ holds for some $p\in C$. In this case, the continuously differentiable map $\rest{H}{\cX_{([C],p)}}\colon \cX_{([C],p)} \to C$ is a submersion, and therefore the restriction of $H$ to neighboring fibers $\cX_{([C'],p')} $ continues to be a submersion. We conclude that in particular we have $H(\cX_{([C'],p')})=\cC_{[C']}$ for all $([C'],p')$ in an open neighborhood $V$ of $([C],p)$ in $\cC$.

As every fiber of $F$ is an IHS manifold, the Stein factorization of each non-trivial map $\rest{H}{\cX_{([C'],p')}}$ must be a Lagrangian fibration with $\dim \cX_{([C'],p')}=2\dim \cC_{[C']}=2$; see \cite[Proposition~24.8]{GHJ}, which generalizes \cite[Theorem~2]{Mat99}. In other words, the restriction of $(\cX\stackrel{F}{\to}\cC,\mu_U)$ to $V$ is a marked family of elliptic K3 surfaces, whose period points fill up the open neighborhood $\pr2(V)$ of $p$ in the K3 period domain $\Omega$.
This is impossible, as elliptic K3 surfaces carry a non-tivial line bundle, while a very general K3 surface has Picard rank zero.

It follows from the contradiction reached above that $f$ and $h$ both contract each others fibers.
As all fibers of both $f$ and $h$ are connected, $C$ is the (intermediate space in the) Stein factorization of both maps.
Hence, the universal property \cite[p.~214]{GR84} of the Stein factorization induces a unique automorphism $\varphi\in\Aut_{\O}(C)$ with $\varphi\circ f=h$.
Hence, $(\id X,\varphi)\from\cF\to\cF'$ is an isomorphism of families (cf.~Section~\ref{families}), which proves the first claim of \Cref{unique fiber structure}.

As a special case, note that every biholomorphic automorphism $g\in\Aut_{\O}(X)$ of the total space $X$ defines via $\cF'=(X\stackrel{f\circ g}{\lto}C)$ a second family structure on $X$ as a family of IHS manifolds over $C$.
Then, the induced isomorphism $(\id X,\varphi)\from\cF\to\cF'$ from above induces an automorphism $(g,\varphi)\from\cF\to\cF$ of $\cF$.
To show that this is a $C$-automorphism, it remains to show that any given automorphism $(g,\varphi)\in\Aut(\cF)$ is trivial on $C$.
As $C\cong\P^1$ is a smooth rational curve, $\varphi$ must have a fixed point, say $p_0\in C$.
Then $(\cF,\varphi^*\mu)$ must also be a marked embedded family with the same total space, but with period map $C\stackrel{\varphi}{\to}C\inj \pdom{\Lambda}$. 
Using the Local Torelli Theorem for IHS manifolds, we can construct a marked family $(\tilde{\cF},\tilde{\mu})$ of IHS manifolds with period map the inclusion $\tilde{W}\subset\pdom{\Lambda}$ of an open neighborhood of $p_0$ such that $\tilde{\cF}$ is a universal deformation of $X_{p_0}$.
Let $W\subset \tilde{W}\cap C$ be an open neighborhood of $p_0$ in $C$ with $\varphi(W)\subset W$.
Then $\tilde{\cF}|_{W}$ is $W$-isomorphic to the pullbacks of $\tilde{\cF}$ under $W\inj \tilde{W}$ and $\rest{\varphi}W$.
By the universality of $\tilde{\cF}$ we must then have $\varphi=\id C$; first on $W$ and then, by the identity theorem, everywhere on $C$.
\end{proof}

\begin{lemm}[Leray spectral sequence of an embedded family]
\label{leray}
Let $\cF=(X\stackrel{f}{\to}C)$ be an embedded family of IHS manifolds.
Then the Leray spectral sequence $E^{pq}_2=\Hsh pC{\Rs qf{\csh\Z X}}\Rightarrow\Hsh{p+q}X{\Z}$ for the morphism $f$ degenerates on $E_2$ up to degree~$2$.
\end{lemm}

\begin{proof}
As $C\cong\P^1$ is simply connected, the locally constant sheaves $\Rs qf{\csh\Z X}$ are isomorphic to the constant sheaves $\underline{\Hsh q{X_0}{\Z}}_{C}$, where $X_0$ denotes a fiber of $f$.
In particular, for all $p,q$ we have $E^{0q}_2\cong\Hsh q{X_0}{\Z}$ and $E^{p0}_2=\Hsh p C{\Z}$.

The degeneracy at $E_2^{00}$ and $E_2^{10}$ is automatic. For the other entries, we conclude from the above that $E_2^{01}=E_2^{11}=\{0\}$, and that $d_2\colon E^{02}_2{\to}E^{21}_2=\{0\}$ and $d_2 \colon E^{01}_2=\{0\}\to E^{20}_2$ are the zero maps. Since in addition $d_3\colon E^{02}_3{\to}E^{30}_3\isom\Hsh3{C}{\Z}=\{0\}$ is also the zero map, the claimed degeneracy holds.
\end{proof}

Analogously to the application of \Cref{marked automorphisms} to prove \Cref{fine}, the following consequence of the previous two lemmata can be used to prove the existence of gluings in \Cref{total spaces}.

\begin{prop}[Automorphisms of total spaces acting trivially on $\Hsh2X{\Z}$]
\label{automorphisms}
Let $\cF=(X\stackrel{f}{\to}C)$ be an embedded family of K3 surfaces.
Then $\id X$ is the only biholomorphic automorphism of the total space $X$ that induces the identity on $\Hsh2X{\Z}$.
\end{prop}
\begin{proof}
We can conclude from \Cref{leray} that $\Hsh2X{\Z}\cong\Hsh2C{\Z}\oplus\Hsh2{X_0}{\Z}$ for a fiber $X_0$ of $f$.
Now, let $g\in\Aut_{\O}(X)$ be a biholomorphic automorphism.
By \Cref{unique fiber structure} we have $f\circ g=f$.
Hence $g^*\in\Aut(\Hsh2X{\Z})$ preserves the decomposition $\Hsh2X{\Z}\cong\Hsh2C{\Z}\oplus\Hsh2{X_0}{\Z}$.
Therefore, when $g^*$ acts trivially on $\Hsh2X{\Z}$, it must act trivially on $\Hsh2{X_0}{\Z}$ for every fiber $X_0$ of $f$.
Thus the result follows from \Cref{automorphism property}.
\end{proof}

\begin{rema}[Generalizations to some classes of IHS manifolds]
\label{ihsm automorphism property}
As the proofs just given show, the only reason to restrict ourselves to K3 surfaces in the formulation of \Cref{marked automorphisms} and \Cref{automorphisms} is to be able to apply \Cref{automorphism property}.
Among the known classes of IHS manifolds there are also some with the property that every element $X$ in the class fulfills this crucial condition; i.e., $\id X$ is the only biholomorphic automorphism acting trivially on $\Hsh2X{\Z}$. For example, every IHS manifold of K3$^{[n]}$-type has this property by \cite[Proposition~10]{MR0728605} and \cite[Theorem~2.1]{MR3112215}; and similarly every IHS manifold of OG$_{10}$-type satisfies the property by \cite[Theorem~3.1]{MR3592467}. Hence, there are fine moduli spaces for marked IHS manifolds of these types.
In particular, by extending our proofs of \Cref{fine} and of \Cref{total spaces} to IHS manifolds of these types, in these cases we can likewise establish the existence of fine moduli spaces both for marked embedded families $(X\stackrel{f}{\to}C,\mu)$ and for their total spaces together with markings on $\Hsh2X{\Z}$.
\end{rema}

\subsection{Applications of second level period maps}
Since the component $C_1(\Omega)\subset\mathsf{S}_2(\Omega)$ containing twistor cycles, although not Hermitian symmetric itself, shares many properties with Hermitian symmetric spaces, we now show that embedded second level families have properties similar to families of projective manifolds that satisfy local Torelli and have a Hermitian-symmetric period domain.

\begin{prop}[Brody hyperbolicity]\label{prop:Brody}
 Let $\cF_{\mathbb{C}}\defeq(X\stackrel{f}{\to}C\stackrel{h}{\to}\mathbb{C})$ be a quadratic second level family of $K3$ surfaces over the complex line $\mathbb{C}$. If there exists a point $z_0 \in \mathbb{C}$ such that the fiber $(X_{z_0} \to C_{z_0})$ is a twistor family, then $\cF_{\mathbb{C}}$ is isomorphic as a second level family to the trivial family over $\mathbb{C}$ with fibers $X_{z_0}$ and $C_{z_0}$. 
\end{prop}
\begin{proof}
 We choose any marking of the family $(X_{z_0} \to C_{z_0})$. Since $\mathbb{C}$ is simply-connected, we may extend it to a marking of $\cF_{\C}$; let the corresponding second level period map be $\Per \colon \mathbb{C} \to\mathsf{S}_2(\Omega)$. Since $\mathbb{C}$ is connected, the image of $\Per$ is contained in $C_1(\Omega) \ni \Per(C_{z_0})$, which is a Kobayashi-hyperbolic manifold by \cite[Theorem~11.3.1]{FHW}. Therefore, $\Per$ is constant with value $\Per(C_{z_0})$. However, the period map factors through the classifying map to the fine moduli space $\sM$. The classifying map consequently has to be constant as well, from which the claim follows. 
\end{proof}

\begin{prop}\label{prop:nocompact}
 Let $\cF_B\defeq(X\stackrel{f}{\to}C\stackrel{h}{\to} B)$ be a marked quadratic second level family of $K3$ surfaces over a connected compact base $B$. If there exists a point $z_0 \in B$ such that the fiber $(X_{z_0} \to C_{z_0})$ is a twistor family, then $\cF_B$ is isomorphic as a marked second level family to the trivial marked family over $B$ with fibers $X_{z_0}$ and $C_{z_0}$ and with the restricted marking.
\end{prop}
\begin{proof}
 As in the previous proof, we consider the induced second level period map $\Per \colon B \to C_1(\Omega)$. This time, we use the fact that $C_1(\Omega)$ is a Stein manifold, see again \cite[Theorem~11.3.1]{FHW}; in particular, $C_1(\Omega)$ does not have positive-dimensional compact complex subspaces. This implies that the proper holomorphic map $\Per$ must be constant, as otherwise $\Per(B)$ would provide such a compact complex subspace by Remmert's Proper Mapping Theorem. We conclude as in the proof of \Cref{prop:Brody} above. 
\end{proof}

\section{The moduli space as a Hausdorff sheaf}
\label{section sheaves}

From this section onward, we focus on families of K3 surfaces. \\
\vspace{-0.2cm}

It is evident from the construction in \Cref{fine} that the moduli space $\sM_{\Lambda}$ of $\Lambda$-marked embedded families of K3 surfaces in each even degree $d$ lies locally biholomorphically over the space $\mathsf{S}_d(\pdom{\Lambda})$. In this section we will show how to construct $\sM:=\sM_{\Lambda}$ out of the Burns-Rapoport space $\tilde{\Omega}$ as a topological sheaf in the sense of Section~\ref{topsheaves}. By studying the sections of this sheaf, we show that $\sM$ is Hausdorff, i.e., it is a complex manifold; this in turn will allow us to deduce that the quotient of $\sM$ by $\Gamma = \Orthp(\Lambda)$ is a coarse moduli space for embedded families of K3 surfaces. We will then use the indicated description to construct a $\Gamma$-invariant marked embedded second level family containing all marked K3 twistor families.

\subsection{The setup}
\label{setup}
We introduce the basic spaces used in our subsequent arguments; the situation will eventually be summarized by Diagram~\eqref{diag1} below. 

For a fixed positive even $d\in\N$ we know from \Cref{Deev} that $\mathsf{S}_d\defeq \mathsf{S}_d(\Omega)\neq\emptyset$ and moreover from \Cref{liftfamily} that there are points of $\sM$ lying over $\mathsf{S}_d$. 
Let $\cC_d\defeq \cC_d(\Omega)\stackrel{\pr1}{\lto}\mathsf{S}_d$ be the family of smooth rational curves in the K3 period domain $\Omega$ described in Section~\ref{cycle section}, with projection $\pr2\colon \cC_d \to \Omega$.
Furthermore, as before, let $\pi\colon \widetilde \Omega \to \Omega$ be the Burns-Rapoport space. 
Then, the fiber product $\cC_d\x_{\Omega}\tilde{\Omega}=\{(([C],p),([\kappa],p))\in\cC_d\x\tilde{\Omega}\}$
is over $\cC_d$ isomorphic to
$$\tilde{\cC}_d\defeq\{([C],([\kappa],p))\in \mathsf{S}_d\x\tilde{\Omega}\mid p\in C\}$$
with projection $q_1\from([C],([\kappa],p))\mapsto([C],p)$ onto $\cC_d$ and projection $q_2$ onto $\tilde{\Omega}$.
As $\pi$ is locally biholomorphic, $q_1$ is locally biholomorphic as well, and the topological sheaf $\sF'\defeq(\tilde{\cC}_d\stackrel{q_1}{\to}\cC_d)$ can be identified with the inverse image sheaf $\pr2\inv(\tilde{\sF})$, cf.~Section~\ref{topsheaves}.
Then $q_1\inv([C],p)\cong\pi_0(V_p^{\circ})$ for each $([C],p)\in\cC_d$.
We pull back the universal family $\nu\from\cM\to\tilde{\Omega}$ of marked K3 surfaces to
$\cN\defeq \tilde{\cC}_d\x_{\tilde\Omega}\cM\cong\cC_d\x_{\Omega}\cM$
with projections $p_1,p_2$ onto $\tilde{\cC}_d$ and $\cM$, respectively, and marking induced by $\cM$.
By definition, $(p_2,q_2)\from(\cN\stackrel{p_1}{\to}\tilde{\cC}_d)\to(\cM\stackrel{\nu}{\to}\tilde{\Omega})$ is then a morphism of marked families of IHS manifolds.
Each fiber of $\cN$ over a point $([C],([\kappa],p))\in\tilde{\cC}_d$ consists of the marked K3 surfaces with period point $p$ and K\"ahler cone $[\kappa]\in\pi_0(V_p^{\circ})$, which is in fact a unique one by the Global Torelli Theorem, see again \cite[Theorem~9.1]{LP80}.
In particular, the period map of $(\cN\stackrel{p_1}{\to}\tilde{\cC}_d)$, indicated by the diagonal arrow in Diagram~\eqref{diag1}, is $\pi\circ q_2\from\tilde{\cC}_d\to\Omega$. As mentioned above, the following diagram summarizes the situation.

\begin{equation}
\begin{aligned}\label{diag1}
\begin{tikzcd}[row sep=normal, column sep=normal]
&\cN\rar{p_2}\dar{p_1}&\cM\dar{\nu}\\
\dou\Omega\x\tilde{\Omega}\dar&\tilde{\cC}_d\drar\rar{q_2}\dar{q_1}\arrow[l,hook']&\tilde{\Omega}\dar{\pi}\\
\dou\Omega\x\Omega\dar&\cC_d\rar{\pr2}\dar{\pr1}\arrow[l,hook']&\Omega\\
\dou{\Omega}&\mathsf{S}_d\arrow[l,hook']&
\end{tikzcd}
\end{aligned}
\end{equation}
\smallskip
\begin{lemm}[Lifts of $\pr2$]
\label{sections of F}
For each complex subspace $S\subset\cC_d$, the natural map 
\begin{align*}\sF'(S)&\longrightarrow\{S\stackrel{\phi}{\to}\tilde{\Omega}\text{ holomorphic with }\pi\circ\phi=\pr2\from S\subset\cC_d\to\Omega\} \\\sigma&\longmapsto q_2\circ\sigma \end{align*}
is a bijection. 
\end{lemm}
\begin{proof}
On the one hand, note that for each $\sigma\in\sF'(S)$ we have $\pi\circ(q_2\circ\sigma)=\pr2\circ q_1\circ\sigma=\pr2$; the map is therefore well-defined.
On the other hand, for each holomorphic lift $S\stackrel{\phi}{\to}\tilde{\Omega}$ of $S\subset\cC_d\stackrel{\pr2}{\to}\Omega$ to $\tilde{\Omega}$, by the universal property of the fiber product $\tilde{\cC}_d$ there exists a unique holomorphic map $S\stackrel{\sigma}\to\tilde{\cC}_d$ such that $q_1\circ\sigma$ is the inclusion $S\subset\cC_d$ and $q_2\circ\sigma=\phi$.
In particular each $\phi$ has a unique preimage $\sigma\in\sF'(S)$.
\end{proof}

\begin{prop}[Sections of $\sF'$]
\label{functor of sections}
For complex subspaces $S\subset\cC_d$, the maps ${\sF'(S)\ni\sigma\mapsto\left[\sigma^*\cN\to S\right]}$ induce an isomorphism between $\sF'$ and the sheaf on $\cC_d$ that assigns to $S$ the set of $S$-isomorphism classes of marked families of K3 surfaces on $S$ with period map $\pr2\from S\to\Omega$.
The inverse map is induced by factoring the refined period map of a marked family through $\tilde{\cC}_d$.
\end{prop}
\begin{proof}
The maps $\sigma\mapsto q_2\circ\sigma$ from \Cref{sections of F} define an isomorphism from $\sF'$ to the presheaf of holomorphic lifts of $\pr2$, which is thus also a sheaf.
By functoriality of the period map, each marked family $\sigma^*\cN=(q_2\circ\sigma)^*\cM$ has period map $\pr2$. Conversely, mapping a marked family over a subspace $S\subset\cC_d$ to its refined period map is also functorial, and two such marked families over $S$ are $S$-isomorphic if and only if their refined period maps coincide.
\end{proof}

\begin{rema}[Continuously varying partitions of $(-2)$-classes]
\label{varying partition}
Let $S\subset\cC_d$ be a subspace.
A section $\sigma\in\sF'(S)$ is by definition a map $([C],p)\mapsto ([C],([\kappa([C],p)],p))$, where $[\kappa([C],p)]$ denotes the Weyl chamber represented by a $\kappa([C],p)\in V_p^{\circ}$.
Recall from Section~\ref{abstract setting} that the connected components of $V_p^{\circ}$ are in one-to-one correspondence with partitions $\Delta_p=\Delta_p
^+\dot{\cup}\Delta_p^-$ satisfying property \eqref{partition}.
In our case, at each point $([C],p)\in S$ the induced partition $P_{([C],p)}$ sorts the $\delta\in\Delta_p$ by the sign of their intersection $\langle\kappa([C],p),\delta\rangle_{\R}$.
Now, since $\pi\from\tilde{\Omega}\to\Omega$ is a local homeomorphism, the representatives $\kappa([C],p)$ can at least locally near each $([C],p)\in S$ be chosen continuously.
Hence, by continuity of the intersection product, this means that the induced sign of each $\delta\in\Delta_p$ is locally constant on the closed subspace $\{([C],p)\in S\mid \delta\in\Delta_p\}$ of $S$.
Via Section~\ref{geometric setting}, this allows an easy geometric interpretation for the induced family $(\cF,\mu)\defeq\sigma^*\cN$: The $\mu\inv(\kappa([C],p))$ are continuously varying K\"ahler classes on the fibers $\cF_{([C],p)}$, such that for each $(-2)$-class $\mu\inv(\delta)$ for $\delta\in\Delta$ the subset of fibers on which it is an effective (resp. anti-effective) $(1,1)$-class is closed in $S$.
\end{rema}

\subsection{The Hausdorff sheaf of marked embedded K3 families}

We consider the direct image sheaf $\sF\defeq(\pr1)_*\sF'$ on $\mathsf{S}_d$. Since $\pr1$ is proper, for a (locally closed) subspace $S\subset \mathsf{S}_d$ a section in $\sF(S)$ is a section $\sigma\in\sF'(\pr1\inv(S))$, cf.~Section~\ref{topsheaves}, and thus induces sections $\sigma_{[C]}\in\tilde{\sF}(C)$ for each $[C]\in S$ via $\sigma_{[C]}(p)\defeq\sigma([C],p)$ for all $p\in C$.

\begin{theo}
\label{T2}
The direct image sheaf $\sF = (\pr1)_*\sF'$ is a Hausdorff sheaf on $\mathsf{S}_d$.
\end{theo}

In other words, if two sections $\sigma_1,\sigma_2\in\sF(U)$ defined over an open subset $U\subset \mathsf{S}_d$ differ at some point $s_0\in U$, they differ in an open neighborhood of $s_0$ in $U$.

\begin{rema}\label{HausdorffRemark}For two sections $\sigma_1,\sigma_2\in\sF(U)$ defined over an open subset $U\subset \mathsf{S}_d$ the set $\{s\in U\mid\sigma_1(s)=\sigma_2(s)\}$ is always open in $U$; 
\Cref{T2} states that it is also closed in $U$.
\end{rema}

Note that $\sF'$ itself is not a Hausdorff sheaf; this is in fact one of the main features of the moduli theory of K3 surfaces.
To see this, recall first how Looijenga and Peters proved that $\tilde{\Omega}$ is not Hausdorff.
They constructed two marked families of K3 surfaces over an open disc $D\subset\Omega$ with period map the inclusion, for which the two induced sections in $\tilde{\sF}(D)$ differ only at a single point \cite[Section~10.3]{LP80}.
Analogously to the proof of \Cref{lift}, we can use this example to produce two sections $\sigma_1,\sigma_2\in\tilde{\sF}(C)$ over a curve $[C]\in \mathsf{S}_d$ that still differ only at a single point; the same continues to hold for the pulled back sections $\pr1\x(\sigma_i\circ\pr2)\in\sF'(\{[C]\}\x C)$.

What is different for $\sF$ is that for two sections $\sigma_i\in\sF(U)$ over some $U\subset \mathsf{S}_d$ to be considered distinct at a curve $[C]\in \mathsf{S}_d$, it already suffices if the induced sections $(\sigma_i)_{[C]}\in \tilde{\sF}(C)$ differ at a single point in $C$.
The proof of \Cref{T2} relies on a further elaboration of the idea explained in \Cref{varying partition}.

\begin{proof}[Proof of \Cref{T2}]
Let $s_0=[C_0]\in \mathsf{S}_d$, then every open neighborhood of $[C_0]$ in $\mathsf{S}_d$ contains an open neighborhood of the form $V=\{[C_t]\in \mathsf{S}_d\mid C_t\subset U\}$ for some open neighborhood $U$ of $C$ in $\Omega$.
As $C_0\cong\CP1$ is rigid, up to shrinking $U$ (and hence $V$) we may assume that over $V$ we are given a $\P^1$-bundle trivialization of $\cC_d$, i.e., 
\begin{align*}
V\x \CP1&\stackrel{\sim}{\to}\pr1\inv(V)\subset\cC_d\\
([C_t],z)&\mapsto([C_t],\phi([C_t],z))\quad\text{with}\quad\phi([C_t],z)\in C_t,
\end{align*}
for some holomorphic map $\phi\from V\x\CP1\to\Omega$. Composing with this trivialization, the space of sections $((\pr1)_*\sF')(V)$ can be identified with the space of continuous maps of the form
\begin{align*}
\sigma\from V\x\CP1&\to \mathsf{S}_d\x\tilde{\Omega}\\
(s,z)&\mapsto (s,([\kappa(s,z)],\phi(s,z)))\quad\text{with}\quad\kappa(s,z)\in V^{\circ}_{\phi(s,z)}.
\end{align*}
Note that since $\pi\from\tilde{\Omega}\to\Omega$ is a local homeomorphism, near every given point $(s_0 , z_0) \in V\x\P^1$ the representatives $\kappa(s,z)\in \Lambda_{\R}$ can locally be chosen in a continuous way.
To show that $(\pr1)_*\sF'$ is a Hausdorff sheaf is now equivalent to showing that each two such maps $\sigma_i=\pr V\x[\kappa_i]\x\phi$, $i=1,2$ that differ at a point $p_0\defeq(s_0,z_0)\in V\x\CP1$, differ also somewhere on $\{s\}\x\CP1$ for every $s$ in an open neighborhood of $s_0$ in $V$. 

The condition $\sigma_1(p_0)\neq\sigma_2(p_0)$ is equivalent to the $\kappa_i(p_0)$ lying in different connected components of $V^{\circ}_{\phi(p_0)}$.
As the Weyl group $\Weyl{\phi(p_0)}$ acts transitively on the set of Weyl chambers $\pi_0(V^{\circ}_{\phi(p_0)})$, this is equivalent to the existence of a $\delta\in\Delta_{\phi(p_0)}$ for which the $\langle\kappa_i(p_0),\delta\rangle$ are non-zero with different sign.
The hyperplane $\P(\delta^{\perp})\subset\P(\Lambda_{\C})$ is the zero set of a linear polynomial, which we denote by $h$.
Then $\phi\inv(\P(\delta^{\perp}))$ is the zero set of the holomorphic function $h\circ\phi$ and defines the set of points $p\in V\x\P^1$ for which $\delta\in\Delta_{\phi(p)}$. As for every $[C]\in \mathsf{S}_d$ we have $C.\P(\delta^{\perp})=\deg\left(\rest{\O_{\P(\Lambda_{\C})}(1)}C\right)=d$, each intersection $C\cap\P(\delta^{\perp})$ is non-empty, and either it is discrete or we have $C\subset\P(\delta^{\perp})$.
We study these two cases separately for the cycle $s_0=[C_0]$.

In the first case, $\phi(s_0,z_0)$ being an isolated point in $C_0\cap\P(\delta^{\perp})$ is equivalent to $(s_0,z_0)$ being an isolated zero of $\rest{h\circ\phi}{\{s_0\}\x\P^1}$.
Owing to the Weierstrass Preparation Theorem we can shrink $V$ and find an open neighborhood $W$ of $z_0$ in $\P^1$ such that for every $s\in V$ the function $h\circ\phi$ has a zero on $\{s\}\x W$.
By shrinking $V$ and $W$ further, we may assume that for $p\in V\x W$ the representatives $\kappa_i(p)$ can both be chosen continuously in $p$, and that the real parts $\Re\left(\langle\kappa_i(p),\delta\rangle_{\C}\right)$ do not vanish.
Then, by construction, at every zero $p\in V\x W$ of $h\circ\phi$ we have $\delta\in\Delta_{\phi(p)}$, the $\langle\kappa_i(p),\delta\rangle$ are real with different signs, and thus $[\kappa_1(p)]\neq[\kappa_2(p)]$ in $\pi_0(V^{\circ}_{\phi(p)})$.

In the second case, $C_0\subset\P(\delta^{\perp})$ means that, for all $z\in\P^1$, we have $\delta\in\Delta_{\phi(s_0,z)}$ and thus the $\langle\kappa_i(s_0,z),\delta\rangle$ are real with different sign.
As $\CP1$ is compact, we can shrink $V$ such that every $(s_0,z)\in\{s_0\}\times\P^1$ has an open neighborhood $V\x W(z)$ in $V\times\P^1$ on which the representatives $\kappa_i$ can be chosen continuously with $\Re\left(\langle\kappa_i(p),\delta\rangle_{\C}\right)\neq0$ for all $p\in V\x W(z)$.
For each $s\in V$ the intersection $C_s\cap\P(\delta^{\perp})$ is non-empty, i.e., the function $\rest{h\circ\phi}{\{s\}\x\P^1}$ must have a zero in one of the $\{s\}\x W(z)$.
Therefore, arguing like in the first case, the two sections need to differ there.
\end{proof}

\begin{prop}[Sections of $\sF$]
\label{functor of sections'}
For complex subspaces $S\subset \mathsf{S}_d$, the maps ${\sF(S)=\sF'(\pr1\inv(S))\ni\sigma\mapsto \left[\sigma^*\cN\to\rest{\cC_d}S\stackrel{\pr1}{\to}{S}\right]}$ induce an isomorphism between $\sF$ and the sheaf on $\mathsf{S}_d$ that assigns to $S$ the set of isomorphism classes of marked embedded second level families of K3 surfaces over $S$ with second level period map the inclusion $S\subset \mathsf{S}_d$.
\end{prop}

\begin{proof}
By \Cref{functor of sections}, $\sF$ is isomorphic to the sheaf that maps a complex subspace $S\subset \mathsf{S}_d$ to the set of $\rest{\cC_d}S$-isomorphism classes of marked families of K3 surfaces over $\rest{\cC_d}S$ with period map $\pr2\from\rest{\cC_d}S\to\Omega$.
These families induce marked second level families over $S$ with second level period map the inclusion $S\subset \mathsf{S}_d$.
We claim that conversely every marked second level family $(X\to\rest{\cC_d}S\to S,\mu)$ with second level period map the inclusion of a complex subspace $S\subset \mathsf{S}_d$ is isomorphic to exactly one of them.
For this, note that the period map induces a holomorphic $S$-automorphism $(\id{S},\per)$ of the $\P^1$-bundle $\rest{\cC_d}S\stackrel{\pr1}{\to}S$.
Hence its inverse $\phi$ induces a second level family $(\phi^*X\to\rest{\cC_d}S\to S,\phi^*\mu)$ with period map $\pr2$.
The uniqueness follows by \Cref{functor of sections}, together with functoriality and injectivity of the second level period map. 
\end{proof}

\begin{rema}[Unobstructedness for embedded families of K3 surfaces]
\label{unobs remark}
Specifically, for every $[C_0]\in \mathsf{S}_d$ sections $\sigma\in\sF([C_0])$ correspond under the isomorphism from \Cref{functor of sections'} to the isomorphism classes of marked embedded families $(\cF,\mu)$ of K3 surfaces with period map mapping to $C_0$.
As $\sF$ is a topological sheaf, each such $\sigma\in\sF([C_0])$ can be extended to a section $\tilde{\sigma}\in\sF(U)$ over an open neighborhood $U$ of $[C_0]$ in $\mathsf{S}_d$; up to shrinking $U$, this extended section is uniquely determined.
This extended section in turn corresponds to the second level family $\tilde{\sigma}^*\cN\to\rest{\cC_d}U\to U$, which by \Cref{local torelli} is universal at every point in $U$.
Up to shrinking $U$ further, this second level family is thus $U$-isomorphic to the universal deformation of $(\cF,\mu)$ constructed in \Cref{unobs}.
\end{rema}

\begin{theo}[The moduli space of marked embedded K3 families is Hausdorff]
\label{MT2}
The total space of $\sF$ is a fine moduli space of marked embedded families of K3 surfaces of degree $d$.
In particular, the moduli space $\sM$ introduced in \Cref{fine} is Hausdorff and therefore a complex space in the usual sense.\footnote{As for example introduced in \cite{Fischer}.}
\end{theo}
\begin{proof}
We write $\sF=(X\stackrel{\pi}{\to}\mathsf{S}_d)$ and let $[C]\in \mathsf{S}_d$.
The fiber $X_{[C]}$ is isomorphic to the space of sections $\sF([C])$. Hence $X$ corresponds via \Cref{unobs remark} to the set of isomorphism classes of marked embedded K3 families of degree $d$.
In fact, the sections $\tilde{\sigma}$ corresponding to the universal deformations of marked embedded K3 families of degree $d$, see \Cref{unobs}, give $X$ the same structure as a complex space as the corresponding connected component of $\sM$.
Therefore, by \Cref{T2}, $\sM$ is also Hausdorff.
\end{proof}

After having constructed fine moduli spaces for \emph{marked} (embedded) families and their total spaces in \Cref{fine} and \Cref{total spaces}, it is desirable to get rid of the marking that was introduced to construct second level period maps. For this, one has to consider the action of the discrete group $\Gamma = \Orthp(\Lambda)$ on $\sM$ and analyze whether it admits a quotient that is a complex space. The Hausdorff property of $\sM$ established in \Cref{MT2} turns this into a sensible problem. The following result says that this goal is at least partly achievable.

To prepare the statement, recall that the universal marked embedded second level family over the fine moduli space $\sM$ of marked embedded families of K3 surfaces yields a second level period map $P\colon \sM \to \mathsf{S}(\Omega)$, which is locally biholomorphic by the Local Torelli Theorem for such families, \Cref{local torelli}, and equivariant with respect to the natural actions of $\Gamma$ on marked families and on the cycle space.

\begin{theo}[Coarse moduli spaces for embedded families over cycles in $C_1(\Omega)$]\label{thm:coarsemoduli}
 Let $\sM^\circ$ be the part of $\sM$ lying over the component $C_1(\Omega)\subset\mathsf{S}_2(\Omega)$ containing twistor cycles; i.e., \[\sM^\circ \defeq P^{-1}(C_1(\Omega)).\] Then, the group $\Gamma$ acts properly discontinuously on $\sM^\circ$ and the complex space $\sM^\circ/\Gamma$ is a Hausdorff coarse moduli space both for embedded families of K3 surfaces over cycles in $C_1(\Omega)$ and for their total spaces. 
\end{theo}
\begin{proof}
 \textcolor{red}{As discussed in 
 }, the action of $\Gamma$ on $C_1(\Omega)$ is properly discontinuous; i.e., for every compact subset $K \subset C_1(\Omega)$ the set $\Gamma_{[K]} \defeq \{\gamma \in \Gamma \mid \gamma\cdot K \cap K \neq \emptyset\}$ is finite. If now $M \subset \sM^\circ$ is compact, then clearly $P(M)$ is compact in $C_1(\Omega)$ and therefore equivariance of $P$ implies that $\Gamma_{[M]} \subset \Gamma_{[P(M)]}$ is finite; i.e., the action of $\Gamma$ on $\sM^\circ$ is properly discontinuous, as claimed. 
 
 As the action is properly discontinuous, the quotient $\sM^\circ/\Gamma$ exists in the category of complex spaces, see e.g.~\cite{CartanDiscreteQuotients}. Since the action of $\Gamma$ on $\sM$ is by change of markings, it is a standard argument to deduce the ``coarse moduli space property'' of the quotient $\sM^\circ/\Gamma$ (for both moduli functors mentioned in the theorem) from the fact that $\sM^\circ$ is a fine moduli space for the respective functor of marked objects established in Theorems \ref{fine} and \ref{total spaces}; see the second paragraph of \cite[Chapter~6, proof of Corollary~4.3]{K3book} for a sketch of the argument in a closely related setup. 
\end{proof}

We move on to discuss further applications of Theorem~\ref{MT2}.

\begin{prop}[Second level families are locally induced by sections]
\label{induced section}
For a marked embedded second level family $(\cF_B,\mu)=(X\to C\to B,\mu)$ of K3 surfaces of degree $d$, we denote by $(Q,\Per)\from(C\to B)\to(\cC_d\to \mathsf{S}_d)$ the morphism of smooth rational curves constructed in \Cref{periodmap2}.
Then, for every $b\in B$ there are connected open neighborhoods $V_b$ of $b$ in $B$ and $U_b$ of $\Per(b)$ in $\mathsf{S}_d$, a section $\sigma_b\in\sF(U_b)$ and a morphism $(\alpha_b,Q,P)\from\rest{(\cF_B,\mu)}{V_b}\to(\sigma_b^*\cN\to\rest{\cC_d}{U_b}\to U_b)$ of marked second level families.
Here, $\sigma_b$ and $\alpha_b$ are unique up to shrinking $V_b$ and $U_b$.
\end{prop}
\begin{proof} We refer the reader to Diagram~\eqref{diag1} for the relation between the morphisms discussed in the subsequent argument. For $b\in B$ we consider the family $(\rest{Q}{C_b}\inv)^*(\cF_b,\mu_b)$ over $Q(C_b)$.
It corresponds to a section in $\sF'(Q(C_b))=\sF(\Per(b))$ that can be extended to a section $\sigma_b$ over a connected open neighborhood $U_b$ of $\Per(b)$ in $\mathsf{S}_d$.
We take for $V_b$ a connected open neighborhood of $b$ in $\Per\inv(U_b)\subset B$.
Let $\tilde{\per}$ be the refined period map of $(X\to C,\mu)$, then by construction we get $\rest{\tilde{\per}}{V_b}=q_2\circ\sigma_b\circ Q$.
Therefore, \Cref{refined} implies that $\rest{(X\to C,\mu)}{V_b}\cong (q_2\circ\sigma_b\circ Q)^*\cM=Q^*(\sigma_b^*\cN)$.
This isomorphism in particular encompasses a morphism $\alpha_b\from X\to\sigma_b^*\cN$ over $Q$. 

It remains to show that the construction does not depend on choices: 
By the universal property of the fiber product, the refined period map $\tilde{\per}\from C \to \tilde \Omega$ factors uniquely through a morphism $\varphi\from C \to \tilde{\cC}_d$, for which we have $\varphi = \sigma_b \circ Q$ on $V_b$.
Hence, as $\rest{Q}{C_b}$ is an embedding, the section $\sigma_b$ is unique up to shrinking $U_b$.
Thus also $\alpha_b$ is unique up to shrinking $V_b$.
\end{proof}

By the functoriality of the period maps and the uniqueness of the local sections $\sigma_b$, we can draw the following immediate consequence.

\begin{theo}[Local isomorphism property of embedded K3 families]
Two (marked) embedded second level families $\cF_B,\cG_B$ of K3 surfaces are locally isomorphic if and only if they are fiberwise isomorphic---that is, every $b\in B$ having an open neighborhood $U$ in $B$ with $\cF_U\cong\cG_U$ is equivalent to $\cF_b\cong\cG_b$ for all $b\in B$.\qed
\end{theo}
\begin{rema}
It is interesting to note that such a \emph{local isomorphism property} does not follow from discreteness of automorphism groups, as it does not hold for families for K3 surfaces themselves---again, because the moduli space of marked K3 surfaces is not Hausdorff; see for example \cite[Theorem~1]{Kir18} and the classical references listed there.
\end{rema}

\Cref{induced section} naturally leads to the problem of determining conditions under which the sections $\sigma_b$ obtained in \emph{loc.~cit.} glue together in $\sF$, or equivalently, the local isomorphisms $\alpha_b$ glue to a global isomorphism of second level families.
Owing to \Cref{extension2}, gluing holds at least in the case when the second level period map is an embedding.
Using the Hausdorff property of $\sF$, we can moreover deduce the following criterion from \Cref{induced section}.

\begin{coro}
\label{global pullback}
Let $(\cF_B,\mu)=(X\to C\to B,\mu)$ be a marked embedded second level family of K3 surfaces of degree $d$ with second level period map $\Per$, let $\sigma_b\in\sF(U_b)$ be the sections introduced in \Cref{induced section}, and let $U$ be a connected open neighborhood of $\Per(B)$ in $\bigcup_{b\in B}U_b$.
Then the following conditions are equivalent:
\begin{enumerate}
\item\label{a} $(\cF_B,\mu)=\Per^*(Y\to\rest{\cC_d}U\to U,\mu)$ for a second level family over $U$ with second level period map the inclusion $U\subset \mathsf{S}_d$.
\item\label{b} There exists a marked second level family $(Y\to\rest{\cC_d}U\to U,\mu)$ over $U$ with second level period map the inclusion $U\subset \mathsf{S}_d$ that is isomorphic to $(\cF_B,\mu)$ after restriction to at least one fiber.
\item\label{c} One of the sections $\sigma_b$ can be extended to $U$.
\item\label{d} The sections $\sigma_b$ glue together to a section in $\sF(U)$.
\end{enumerate}
\end{coro}
\begin{proof}
\Cref{a} implies \Cref{b} because pullbacks have isomorphic fibers.
Assuming \Cref{b}, the family $(Y\to\rest{\cC_d}U\to U,\mu)$ corresponds to a section $\sigma\in\sF(U)$ with $\Per$ inducing an isomorphism of the fibers $(\cF_b,\mu_b)\bij(\sigma^*\cN)_{\Per(b)}$ for a $b\in B$.
Then $\sigma$ and $\sigma_b$ must coincide over $\Per(b)$.
As $\sF$ is Hausdorff by \Cref{T2}, we must have $\rest{\sigma}{U_b}=\sigma_b$, cf.~\Cref{HausdorffRemark}, and thus \Cref{c} holds.
Due to the uniqueness of the $\sigma_b$, \Cref{c} implies \Cref{d}.
Finally, if the $\sigma_b$ glue together to a section $\sigma\in\sF(U)$, then the $\alpha_b$ of \Cref{induced section} glue together to a morphism over $Q$, which is a morphism over $\Per$, implying \Cref{a}. 
\end{proof}

\subsection{Universal small deformation of K3 twistor families}

In the setup of Section~\ref{setup} we set $d=2$ and restrict $\cN\stackrel{p_1}{\to}\tilde{\cC}_2\stackrel{q_1}{\to}\cC_2\stackrel{\pr1}{\to} \mathsf{S}_2$ over the cycle space to $\cN\stackrel{p_1}{\to}\tilde{\cC}\stackrel{q_1}{\to}\cC\stackrel{\pr1}{\to} C_1(\Omega)$.
Here, $\cC\stackrel{\pr1}{\to}C_1(\Omega)\subset \mathsf{S}_2$ is the universal family over the cycle space discussed in Section~\ref{subsect:cyclespaces}, with restriction $\cC_{\R}\defeq\cC_{\R}(\Omega)$ over the real cycle space $C_1(\Omega)_{\R}$.
We denote the restricted sheaves $\sF'\defeq \rest{\sF'}{\cC}$ and $\sF\defeq \rest{\sF}{C_1(\Omega)}$ as before.

We will construct a particular section $\sigma\in\sF(U)$ over an open neighborhood $U$ of $\Twr$ in $\Tw$ invariant under the natural $\Gamma$-action defined below in \Cref{sheaf action}, see \Cref{realfamily} and \Cref{extsection}.
For the induced family $\sigma^*\cN$ over $U$, which contains all K3 twistor families, a universal property will be stated and proven in \Cref{univfamily}.

\begin{rema}[$\Gamma$-actions on the sheaves $\tilde{\sF},\sF',\sF$]
\label{sheaf action}
As $\Gamma$ acts equivariantly on $\tilde{\Omega}\stackrel{\pi}{\to}{\Omega}$ by Section~\ref{weyl group} and obviously also on $\cC\stackrel{\pr2}{\lto}\Omega$, the $\Gamma$-actions on $\cC$ and $\tilde{\Omega}$ induce a $\Gamma$-action on $\tilde{\cC}$ for which the projections $q_1,q_2$ are equivariant.
In particular, we can consider $\tilde{\sF}$ and $\sF'$ as $\Gamma$-equivariant sheaves via conjugation of sections by elements $\gamma\in\Gamma$ (see Section~\ref{topsheaves}).
The same works for $\sF$ because $\Gamma$ also acts equivariantly on $\cC\stackrel{\pr1}{\lto}C_1(\Omega)$ by the operation discussed in Section~\ref{subsection_cycle_action}.
\end{rema}

When discussing the Global Torelli Theorem and surjectivity of the period map for K3 surfaces, classical references, see for example \cite[\S1]{Kob90}, \cite[Chapter~VIII, Proposition~9.1, Corollary~9.2 and Proposition~13.5]{BHPV} and \cite[p.~148]{Bea81}, as an auxiliary object introduce the spaces of classes of volume one\footnote{Albeit with different notation, partially clashing with the one adapted here.} in $K\Omega$
\begin{align*}
(K\Omega)_{=1}&\defeq \{(\kappa,p)\in K\Omega \mid \langle \kappa, \kappa \rangle_\R = 1 \}\\
\intertext{and its open subspace} 
(K\Omega)^{\circ}_{=1}&\defeq (K\Omega)^{\circ}\cap (K\Omega)_{=1},
\end{align*}
together with the real-analytic map $\Pi\from (K\Omega)_{=1}\to\mathbb{G}^+_{\R}$, $(\kappa,p)\mapsto\R\kappa\oplus E(p)$, where $E\from\Omega\to\Grp2{\Lambda_{\R}}$, $[x]\mapsto \R\Re x\oplus\R\Im x\subset\Lambda_{\R}$ for $x \in \Lambda_\C \setminus \{0\}$ is the natural real-analytic two-to-one covering, see \cite[Proposition~6.1.5]{K3book}. Together with the mapping $\psi$ (see \Cref{cyclemap} and \Cref{cyclespace}) and natural projection maps we obtain the following commutative diagram of real-analytic manifolds:
\begin{equation}
\label{diag2}
\begin{tikzcd}[row sep=normal, column sep=huge]
\mathbb{G}_{\R}^+\dar{\psi}&(K\Omega)_{=1}\dar{(\psi\circ\Pi,\pr2)}\lar[swap]{\Pi}\rar{\pr2}&\Omega\dar[equal]\\
C_1(\Omega)_{\R}&\cC_{\R}\lar[swap]{\pr1}\rar{\pr2}&\Omega.
\end{tikzcd}
\end{equation}
\begin{prop}
\label{isom}
The map $(\psi\circ\Pi,\pr2)\from (K\Omega)_{=1}\overset{\sim}\lto\cC_{\R}$ is a $\Gamma$-equivariant isomorphism of real-analytic manifolds.
Over the set of twistor cycles, it restricts to a $\Gamma$-equivariant isomorphism $(K\Omega)^{\circ}_{=1}\stackrel{\sim}{\lto}\rest{\cC}{\Twr}$.
\end{prop}
In other words, the universal families $\cC\stackrel{\pr1}{\lto} C_1(\Omega)$ and $\rest{\cC}{\Tw}\stackrel{\pr1}{\lto}\Tw$ provide natural extensions for $(K\Omega)_{=1}\stackrel{\Pi}{\lto}\mathbb{G}^+_{\R}$ and $(K\Omega)^{\circ}_{=1}\stackrel{\Pi}{\lto}(\mathbb{G}^+_{\R})^{\circ}$, respectively.
\begin{proof}[Proof of \Cref{isom}]
Pointwise, the desired inverse morphism $\cC_{\R}\to (K\Omega)_{=1}$ must map each point in $\cC_{\R}$ via $([C],p)\mapsto(\kappa,p)$, where $\kappa\in E(p)^{\perp}\subset\psi\inv([C])$ has volume one.
As a morphism, we will now construct this map via a normalized section in a real differentiable line bundle on $\cC_{\R}$; as the inverse of a real-analytic map, this then has to be real-analytic as well, e.g.~by the real-analytic inverse function theorem.

Note that pulling back the defining vector bundles on the positive Grassmannians $\Grp3{\Lambda_{\R}}$ and $\Grp2{\Lambda_{\R}}$ via the two maps
\begin{align*}
&\cC_{\R}\stackrel{\pr1}{\lto}C_1(\Omega)_{\R}\stackrel{\psi\inv}{\lto}\G_{\R}^+\bij \Grp3{\Lambda_{\R}}\\
&\cC_{\R}\stackrel{\pr2}{\lto}\Omega\stackrel{E}{\lto}\Grp2{\Lambda_{\R}}
\end{align*}
yields two real-analytic vector bundles $\cV,\cW$ on $\cC_{\R}$ of rank $3$ and $2$, respectively.
By construction, for each $([C],p)\in\cC_{\R}$ the fiber $\cV_{([C],p)}=\psi\inv([C])\cap\Lambda_{\R}$ contains the fiber $\cW_{([C],p)}=E(p)$, hence $\cW$ is a subbundle of $\cV$.
Therefore, the orthogonal complement of $\cW$ in $\cV$ is a real differentiable line bundle $\cL$ on $\cC_{\R}$.
As $\cC_{\R}$ is simply connected by the observations made in Section~\ref{cycleintro}, we have $\Hsh2{\cC_{\R}}{\Orth(1)} = \Hsh2{\cC_{\R}}{\Z/2\Z} = \{0\}$, and hence the line bundle $\cL$ is (differentiably) trivial.
We can thus find a nowhere vanishing differentiable section of $\cL$ -- that is, a map $\kappa\from\cC_{\R}\to\Lambda_{\R}\setminus\{0\}$ such that for each $([C],p)\in\cC_{\R}$ we have an orthogonal decomposition $\psi\inv(C)\cap\Lambda_{\R}=E(p)\oplus\R\kappa([C],p)$.
In particular $\kappa([C],p)$ is always orthogonal to each $x\in\Lambda_{\C}\setminus\{0\}$ with $p=[x]$ and $\langle\kappa([C],p),\kappa([C],p)\rangle_{\R}>0$ since $\psi\inv(C)\in\G^+_{\R}$.
By normalizing $\kappa$, we may assume $\langle\kappa([C],p),\kappa([C],p)\rangle_{\R}=1$ for all $([C],p)\in\cC_{\R}$.
Then, up to changing the sign of $\kappa$, we can assume $\kappa([C],p)\in V_p$ and thus also $(\kappa([C],p),p)\in (K\Omega)_{=1}$ for every $([C],p)\in\cC_{\R}$.
The map $(\kappa,\pr2)\from\cC_{\R}\to(K\Omega)^{\circ}_{=1}$ is thus the inverse real-analytic map we were looking for.

Moreover, the section $\kappa$ is at a point $([C],p)\in\cC_{\R}$ orthogonal to a $\delta\in\Delta_p$ if and only if $\psi\inv([C])\subset\delta^{\perp}$.
Hence $\kappa([C],p)\in V_p^{\circ}$ for all $p\in C$ is equivalent to $[C]$ being a twistor cycle.
Thus, the map $(\kappa,\pr2)\from\cC_{\R}\to (K\Omega)_{=1}$ restricts to an isomorphism $\rest{\cC}{\Twr}\to(K\Omega)^{\circ}_{=1}$.

For the compatibility with the $\Gamma$-action, note first that $E([\gamma x])=\gamma E([x])$ for all $\gamma\in\Gamma$ and ${x\in\Lambda_{\C}\setminus\{0\}}$.
As a consequence, the maps $E$, $\Pi$ and thus also $(\psi\circ\Pi,\pr2)\from (K\Omega)_{=1}\to\cC_{\R}$ are all $\Gamma$-equivariant, as claimed. 
\end{proof}

\begin{theo}[The real-analytic second level family of K3 twistor families]
\label{realfamily}
There exists a $\Gamma$-invariant real-analytic section $\sigma_{\R}\in\sF({\Twr})$ such that the induced real-analytic marked second level family $\sigma_{\R}^*\cN\to\rest{\cC}{\Twr}\to\Twr$ restricts over each $[C]\in\Twr$ to a marked K3 twistor family. 
\end{theo}
\begin{proof}
The real-analytic map $(\kappa,\pr2)\from\rest{\cC}{\Twr}\to(K\Omega)^{\circ}_{=1}$ constructed in the proof of \Cref{isom} is like its inverse $\Gamma$-equivariant.
By construction we obtain a $\Gamma$-equivariant real-analytic section
\begin{align*}
\sigma_{\R}\defeq(\pr1,([\kappa],\pr2))\from\rest{\cC}{\Twr}&\to\tilde{\cC}\\
([C],p)&\mapsto ([C],([\kappa([C],p)],p))
\end{align*}
of $q_1$. Owing to properness of the projection $\pr1$, we have $\sF'(\rest{\cC_{\R}}{\Twr})=\sF(\Twr)$; i.e, $\sigma_{\R}$ can (and will) be interpreted as a $\Gamma$-invariant section in $\sF(\Twr)$.

The induced family $\sigma_{\R}^*\cN\eqdef(X\to\rest{\cC}{\Twr},\mu)$ is a marked real-analytic family of K3 surfaces with refined period map $([C],p)\mapsto([\kappa([C],p)],p)$.
On the other hand, we can take the induced twistor family $\cF=(T\to\P^1)$ associated to a given fiber $X_{([C],p)}$ with K\"ahler class $\mu_{([C],p)}\inv(\kappa([C],p))$.
We extend the marking $\mu_{([C],p)}$ trivially to a marking of $\cF$, then the refined period maps of $\cF$ and $\rest{\sigma_{\R}^*\cN}{\{[C]\}\x C}$ are by \Cref{twistor cycle decomposition} embeddings with the same image in $\tilde{\Omega}$.
By \Cref{refined} both families are isomorphic as marked families.
\end{proof}

\begin{rema}[Comparison of \Cref{realfamily} with \cite{Kob90}]
Our construction in \Cref{realfamily} yields essentially the same family Kobayashi describes in~\cite[Theorem~10]{Kob90}.
The relatively explicit construction of the section $\sigma_{\R}$ in our proof gives a description of the family using which it is easier to discuss extensions of the family to neighborhoods of $\Twr$ in $\Tw$.
Moreover, it was shown above that this family is real-analytic, which is not so obvious from Kobayashi's description that uses gluings of universal deformations of K\"ahler-Einstein K3-surfaces and Torelli Theorems for such partially differential-geometric objects.
\end{rema}

For purely topological reasons (see \cite[Proposition~II.2.5.1(iii)]{KS94}), the section $\sigma_{\R}$ constructed in \Cref{realfamily} extends to a section over an open neighborhood of $\Twr$ in $C_1(\Omega)$.
However, it requires some work to show that the extension of $\sigma_{\R}$ can also be chosen as a $\Gamma$-invariant section.
This follows from the following general result regarding equivariant geometry.

\begin{prop}[$\Gamma$-invariant extension of real sections]\label{extsection_general}
Let $\sG$ be the $\Gamma$-equi\-va\-riant sheaf of sections of a $\Gamma$-equivariant topological sheaf on $\Tw$.
Assume furthermore that we are given a $\Gamma$-invariant section $\sigma_\R \in \sG(\Twr)$ over $\Twr$. Then, there is a $\Gamma$-stable connected open neighborhood $U$ of $\Twr$ in $\Tw$ over which $\sigma_{\R}$ can be extended (uniquely) to a $\Gamma$-invariant section $\sigma\in\sG(U)$. 
\end{prop}

\begin{proof}
Let $\pi\from \Tw \to \Tw/\Gamma$ be the quotient map.
First, we recall that by Section~\ref{ProperlyDisc} the action of $\Gamma$ on $\Tw$ is properly discontinuous,
in particular, the quotient $\Tw/\Gamma$ is Hausdorff.
Moreover, for every $p \in \Tw$ the isotropy subgroup $\Gamma_p$ is finite and $p$ admits a neighborhood basis of $\Gamma_p$-invariant connected neighborhoods $S_p$ (that is $\Gamma_p$-equivariantly biholomorphic to a $\Gamma_p$-stable ball in a $\Gamma_p$-representation) such that the natural map $\Gamma \x_{\Gamma_p} S_p \to X$ yields a $\Gamma$-equivariant biholomorphism between the left hand side and its image $B_p$, which is an open $\Gamma$-stable neighborhood of $\Gamma\cdot p$ in $\Tw$; see e.g.~\cite{CartanDiscreteQuotients}.
We call $S_p$ a \emph{slice at $p$} and $B_p$ a \emph{slice neighborhood} of $\Gamma\cdot p$.
The presheaf $U \mapsto \sG(\pi^{-1}(U))^\Gamma$ of $\Gamma$-equivariant sections in $\sG$ is actually a sheaf, which we will denote by $\pi_*(\sG)^\Gamma$; cf.~\cite[p.~198]{GrothendieckTohoku}.

The subset $\Twr$ is closed and $\Gamma$-invariant in $\Tw$, so that its image $\pi(\Twr)$ is closed in $\Tw/\Gamma$ and realizes the topological quotient $\Twr/\Gamma$, with quotient map $\rest{\pi}{\Twr}\from \Twr \to\pi(\Twr)$.
Similar to the above, we will look at the sheaf $(\rest{\pi}{\Twr})_*(\rest{\sG}{\Twr})^\Gamma$.
This is relevant, since the section $\sigma_\R$ we want to extend is a section of this sheaf of invariants. \newline

\vspace{-0.2cm}
\noindent
\emph{Main claim:} The following natural map is an isomorphism of sheaves of sets:
\begin{equation}
\label{eq:mainclaim}
\rest{\pi_*(\sG)^\Gamma}{\pi(\Twr)} \longrightarrow (\rest{\pi}{\Twr})_*(\rest{\sG}{\Twr})^\Gamma
\end{equation}

\vspace{0.2cm}

Since this can be checked at the level of stalks, let $[p] = \pi(p) \in \pi(\Twr)$ be any point, represented by $p \in \Twr$, and let $B_p$ be a slice neighborhood of $\Gamma \cdot p$ in $\Tw$; i.e., $B_p \cong \Gamma \x_{\Gamma_p} S_p$, with corresponding isomorphism $B_p / \Gamma \cong S_p/\Gamma_p$.
Open sets of the form $S_p / \Gamma_p$ a neighborhood basis of the point $[p] = \pi(p)$ in $\Tw/\Gamma$.
Then, the restriction map yields an equivalence of categories between the category of $\Gamma$-sheaves on $B_p$ and the category of $\Gamma_p$-sheaves on the slice $S_p$.
We are therefore reduced to showing the analog of \eqref{eq:mainclaim} for the quotient $S_p \to S_p/\Gamma_p$, together with the sheaf $\rest{\sF'}{S_p}$ and its restriction to $\Twr \cap S_p$.
This however follows from the simple observation that every $\Gamma_p$-equivariant section of $\rest{\sG}{\Twr \cap S_p}$ defined in a neighborhood of $p$ in $\Twr \cap S_p$ first can be extended uniquely to a section of $\sG$ defined in a small neighborhood of $p$ in $S_p$, which can then be restricted to a $\Gamma_p$-invariant neighborhood of $p$ and made equivariant by averaging over the finite group $\Gamma_p$.
This concludes the proof of the main claim. \newline

Finally, since the action of $\Gamma$ on $\Tw$ is properly discontinuous, the quotient $\Tw / \Gamma$ is Hausdorff, and locally embeddedable as an analytic subset of some open subset of $\C^N$, for some $N$, see~e.g.~\cite{CartanDiscreteQuotients}; in particular, $\Tw / \Gamma$ is locally compact and also second countable (see \cite[Proposition~3.56]{LeeTopological}). Consequently, the quotient space is paracompact, e.g.~by~\cite[Theorem~4.77]{LeeTopological}.

With these preparations at hand, we may apply \cite[Proposition~II.2.5.1(iii)]{KS94} to the sheaf $\pi_*(\sG)^\Gamma$ and $\sigma_{\R}$, seen via the isomorphism \eqref{eq:mainclaim} as an element of $\pi_*(\sG)^\Gamma(\pi(\Twr))$.
We obtain an open neighborhood $V$ of the closed subset $\pi(\Twr)$ in $\Tw/\Gamma$, together with a section $\sigma \in \pi_*(\sG)^\Gamma(V)$ uniquely extending $\sigma_\R$, which we may finally interpret as a $\Gamma$-invariant section of $\sG$ defined over $U \defeq \pi^{-1} (V)$, which is a $\Gamma$-stable open neighborhood of $\Twr$ in $\Tw$. 
\end{proof}

While the following is clear from the previous proof, we note it explicitly for later usage.
\begin{rema}[$\Gamma$- and $\tau$-stable neighborhood basis of $\Twr$]
\label{neighborhoods}
Since $\Gamma$ acts via homeo\-morphisms on $\Tw$, the quotient map $\pi\from\Tw\to\Tw/\Gamma$ is open.
Hence, there is a basis of $\Gamma$-stable open neighborhoods $U$ of $\Twr$ in $\Tw$.
Note that the action of $\Gamma$ commutes with $\tau$, hence by replacing each $U$ with $U\cap\tau(U)$ we can further achieve $\tau(U)=U$ and $U^{\tau}=\Twr$.
\end{rema}

Applying \Cref{extsection_general} and \Cref{neighborhoods} to the $\Gamma$-invariant section $\sigma_\R\in\sF(\Twr)$ of \Cref{realfamily} yields the following result.

\begin{coro}\label{extsection}
There is a $\Gamma$- and $\tau$-stable connected open neighborhood $U$ of $\Twr$ in $\Tw$ over which the section $\sigma_\R$ obtained in \Cref{realfamily} can be extended (uniquely) to a $\Gamma$-invariant section $\sigma \in \sF(U)$.
\end{coro}

\begin{rema}[Invariant sections induce equivariant families]
\label{equivariant}
Let $S\subset\cC$ be a $\gamma$-stable complex subspace for some $\gamma\in\Gamma$, and $\sigma\in\sF'(S)$ a $\gamma$-invariant section with induced marked family $\sigma^*\cN\eqdef(X\stackrel{\pi}{\to}S,\mu)$.
Then, the automorphism $\phi\defeq\mu\inv\circ\gamma^*\circ\mu$ of $\Rs2{\pi}{\csh\Z X}$ induces for each point $p\in V$ an effective Hodge isometry $\Hsh2{X_{\gamma(p)}}{\Z}\to\Hsh2{X_p}{\Z}$.
Hence, there must be a unique automorphism $\tilde{\gamma}\in\Aut_{\O}(X)$ such that $(\tilde{\gamma},\gamma)$ is a marked automorphism of $\sigma^*\cN$ fulfilling $\tilde{\gamma}^*=\phi$ (cf.~\cite[Theorem~1']{BR75}).
In particular, if $S$ is $\Gamma$-stable and $\sigma$ is $\Gamma$-invariant, then $\sigma^*\cN$ becomes a $\Gamma$-equivariant marked family via the induced action on $X$.
\end{rema}

\begin{defi}[Universal small deformation of K3 twistor spaces]
\label{defT}
Using \Cref{realfamily} together with \Cref{extsection} and \Cref{equivariant}, we obtain a $\Gamma$-invariant section $\sigma\in\sF(U)$ over a $\Gamma$- and $\tau$-stable open neighborhood $U$ of $\Twr$ in $\Tw$ such that the induced marked embedded second level family $\cT_U\defeq(\sigma^*\cN\to\rest{\cC}U\to U)$ of K3 surfaces is $\Gamma$-equivariant. We call $\cT_U$ the \emph{universal small deformation of K3 twistor spaces}.
\end{defi}

By functoriality, the period map and the second level period map of $\cT_U$ are the projection $\rest{\cC}U\stackrel{\pr2}{\to}\Omega$ and the inclusion $U\subset C_1(\Omega)$, respectively. The name for $\cT_U$ introduced in \Cref{defT} stems from the following universal property. 

\begin{theo}[Universal property of the extended family]
\label{univfamily}
If $(\cF_B,\mu)$ is a deformation of a marked K3 twistor family such that $\Per(B)\subset U$ holds for the associated second level period map $\Per$, then $(\cF_B,\mu)\isom\Per^*\cT_U$ for the family $\cT_U$ introduced in \Cref{defT}. In particular, the functor of small deformations of marked K3 twistor families is represented by the family $\cT_U$.
\end{theo}

\begin{proof}
By assumption and \Cref{twistorisomorphism} below, $(\cF_B,\mu)$ and $\cT_U$ have at least one isomorphic fiber.
Hence, $(\cF_B,\mu)\isom\Per^*\cT_U$ by \Cref{global pullback}.
For a small deformation of a marked K3 twistor family we can apply this observation up to shrinking the base.
The uniqueness of the pullback map follows from the functoriality of the second level period map, see \Cref{functoriality2}, because the second level period map of $\cT_U$ is injective.
\end{proof}

\begin{lemm}
\label{twistorisomorphism}
Two marked twistor families of K3 surfaces with the same image of their period maps are isomorphic as marked families.
\end{lemm}
\begin{proof}
Let $X=((M,I),\mu)$ and $X'=((M',I'),\mu')$ be two marked K3 surfaces with K\"ahler classes $a,a'$ and $C_{X,\mu,a}=C_{X',\mu',a'}$.
Let $g$ be the hyperk\"ahler metric on $M$ with respect to complex structures $I,J,K$ satisfying the quaternionic relations.
Then $\mu'_{\C}(\Hdg20{X'})=\mu_{\C}(\Hdg20{M,t_1I+t_2J+t_3K})$ for some $t\in S^2$.
Hence there is a Hodge isometry $\phi\from\Hsh2{X'}{\Z}\to\Hsh2{M,t_1I+t_2J+t_3K}{\Z}$.
By changing sign and composing with a Picard-Lefschetz reflection if necessary, we can assume $\phi$ to be effective and compatible with the K\"ahler cones.
By the Global Torelli Theorem, $\phi$ is then induced by an isomorphism $X'\bij(M,t_1I+t_2J+t_3K)$.
The hyperk\"ahler metric in each K\"ahler class is unique, hence the twistor constructions for $X$ and $X'$ yield isomorphic families.
\end{proof}

\subsection{Consequences of the existence and further questions}\label{section_disucssion_univ_family}

As $\cT_s$ is a twistor family for every $s\in\Twr$, we obtain the following immediate consequence of \Cref{univfamily}.
\begin{coro}
Let $(\cF_B,\mu)$ be a marked embedded second level family of K3 surfaces with second level period map $\Per$.
If the fiber $\cF_0$ over a point $0\in B$ is a twistor family, then, up to shrinking $B$ around $0$, for every $b\in B$ with $[\Per(b)]\in\Twr$ the fiber $\cF_b$ is also a twistor family.\qed
\end{coro}

The existence of a $\Gamma$-equivariant family over $U$ has an interesting consequence for the $\Gamma$-action on the cycle space, cf.~\Cref{FixedEqualsPointwiseFixed}.

\begin{prop}[Lattice automorphisms fixing cycles in $U$]\label{prop:FixedCyclesFixedPoints}
Let $[C] \in U \subset C_1(\Omega)$.
If $\gamma \in \Gamma_{[C]}$, then $\gamma \in \bigcap_{p \in C}\Gamma_p$.
That is, if $\gamma$ fixes $C$ as a set, $\gamma (C)=C$, then $\gamma$ already fixes $C$ pointwise.
A fortiori, $\gamma$ in fact induces a holomorphic automorphism of each fiber of $\cT_{[C]}$. 
\end{prop}

\begin{proof}
As $\cT_U$ is $\Gamma$-equivariant and $\gamma(C)=C$, the fiber $\cT_{[C]}$ is a $\gamma$-equivariant marked family of K3 surfaces.
Via \Cref{equivariant}, we lift $\gamma$ to a marked automorphism $(\tilde{\gamma},\gamma)$ of $\cT_{[C]}$.
The period map of $\cT_{[C]}$ is the $\gamma$-equivariant embedding $\per\from \{[C]\}\x C\stackrel{\pr2}{\inj}C\subset\Omega$.
By functoriality of the period map we must have $\per=\per\circ\gamma$ and thus $\gamma$ must act trivially on $C$.
Thus, as every $p \in C$ is fixed by $\gamma$, the automorphism $\tilde{\gamma}$ must in fact induce an automorphism on each fiber of $\cT_{[C]}$.
\end{proof}

\begin{rema}[Stabilizers of cycles and diffeomorphisms]\label{singleDiffeomo}
\Cref{prop:FixedCyclesFixedPoints} can be interpreted as saying that every element of $\Gamma_{[C]}$ is geometrically induced;
this should be compared with the situation regarding stabilizers $\Gamma_p$ of points $p \in \Omega$ described in Section~\ref{weyl group}.
We expect that in fact $\gamma$ is induced by a single diffeomorphism of the real-analytic four-manifold $M$ underlying the K3 surface in any fiber of $\cT_{[C]}$. This expectation is based on the observation made in \Cref{FixedHKTwistorLines} below. Before stating the result, we recall (for example from~\cite{BanerjeeMoore}) that a \emph{hyperk\"ahler isometry} of a hyperk\"ahler manifold $(M, g)$ with compatible complex structures $I, J, K \in \End(TM)$ is an isometry $\psi$ of the Riemannian manifold $(M, g)$ that satisfies $\psi^*(L) = L$ (i.e., that is holomorphic) with respect to all complex structures $L$ in the $2$-sphere of complex structures on $M$ generated by $I,J, K$.

\end{rema}

\begin{prop}[Twistor cycles fixed by lattice automorphisms]\label{FixedHKTwistorLines}
Let $[C]=[C_{X,\mu,a}] \in \Twr$ be a K3 twistor cycle, for $X=(M,I)$ with hyperk\"ahler metric $g$ on $M$ in the K\"ahler class $a$.
If $\gamma \in \Gamma_{[C]}$, then $\gamma(p) = p$ for all $p \in C$.
That is, if $\gamma(C) = C$, then $\gamma$ fixes $C$ pointwise.
A fortiori, $\gamma$ induces an automorphism of the corresponding twistor family induced by a hyperk\"ahler isometry of $(M,g)$. 
\end{prop}
\begin{proof}
Let $(T\to\P^1,\mu)$ be the corresponding marked twistor family.
As its period map is an isomorphism $\per\from \P^1\bij C$, the automorphism $\rest{\gamma}C\in\Aut(C)$ has at least one fixed point, say $p_0 = \per(z_0)$ for $z_0 \in \P^1$.
Writing $p_0 = [v_0]$ for $v_0 \in \Lambda_{\C}\setminus\{0\}$, we get $C = \P(V)\cap \Omega$ for $V \defeq \C v_0\oplus\C a_0\oplus\C\overline{v_0}$, where $a_0$ is the image of the class of the K\"ahler form of $g$ on the K3 surface $T_{z_0}$ under the marking.
As $\gamma$ is real, it also fixes $\C\overline{v_0} \subset V$ and hence $\overline{p_0}\in C$.
Moreover, $\gamma$ induces a unitary automorphism on $V$ with respect to $\rest{\langle\cdot,\overline{\cdot}\rangle_{\C}}V$, thus $\gamma$ preserves also the orthogonal complement $(\C v_0\oplus\C\overline{v_0})^{\perp}=\C a_0$ in $V$.
Consequently, $\gamma(a_0) = \pm a_0$.
Since $\gamma \in \Orthp(\Lambda)$ preserves the orientation of $V$, we must have $\gamma (a_0) = a_0$.
In particular, $\gamma$ fixes $C$ pointwise.

Let $z \in \P^1$ be a point with corresponding K\"ahler class mapping to $a \in \Lambda_\R$ via the marking $\mu_\R$.
From the above we infer that $\gamma$ yields a Hodge isometry of $\Hsh2{T_z}{\Z}$ whose real extension fixes the K\"ahler class $\mu^{-1}_{\R}(a)$.
Hence, by the Global Torelli Theorem, there exists an automorphism $\varphi$ of $X_z$ such that $\gamma = \varphi^*$ on $\Hsh2{T_z}{\Z}$.
Since the $(1,1)$-form $\kappa_z$ on $T_z$ induced by the K\"ahler-Einstein metric $g$ on $T_z$ is uniquely determined by the class $\mu^{-1}_{\R}(a)$, we may up to rescaling $a$ assume that $\varphi^*(\kappa_z)=\kappa_z$ and $\varphi^*(g)=g$.
In other words, $\varphi$ is an isometry of the Riemannian manifold $(M,g)$ underlying $T_z$ and therefore induces a holomorphic automorphism of the twistor family.
Moreover, as we have seen above, $\varphi$ also satisfies $\varphi^*(L) = L$ for all complex structures $L = L_z$, $z \in \P^1$, so that it is indeed a hyperk\"ahler isometry. 
\end{proof}

Once we know that $\sigma_{\R}$ extends to some $\Gamma$-stable open neighborhood $U$ of $\Twr$ in $\Tw$, it is natural to ask whether there is an extension to the whole of $\Tw$. The following result gives a negative answer to this question. 

\begin{lemm}[Extensions of $\sigma_{\R}$]
\label{extensions}
Let $\sigma\in\sF(U)$ be a section on a connected open neighborhood $U$ of $\Twr$ in $C_1^+(\Omega)$ with $\rest{\sigma}{\Twr}=\sigma_{\R}$.
Then, $U\subsetneq\Tw$ is a proper subset.
A fortiori, for every point $p\in\Omega$ the set $U\cap \Tw_p\defeq\{[C]\in U\cap\Tw\mid p\in C\}$ must not connect any two different connected components of $\Tw_{\R,p}\defeq\{[C]\in\Twr\mid p\in C\}$.
\end{lemm}

We will discuss this result in a couple of remarks before proving it further down. To understand the last condition, we denote for a given point $p\in\Omega$ the sets of real or complex cycles containing $p$ by $\Tw_{\R,p}\subset C_1(\Omega)_{\R,p}$ and $\Tw_p\subset C_1(\Omega)_p$, respectively.

\begin{rema}\label{cycles_through_point}
Although $\Tw$ is a complexification of $\Twr$, the space $\Tw_p$ is not a complexification of $\Tw_{\R,p}$.
In fact, the condition $p\in C$ fixes one complex direction of $\psi\inv([C])\subset\Lambda_{\C}$, but if $[C]$ is real, it fixes the two real directions $E(p)\subset\psi\inv([C])\cap\Lambda_{\R}$. 
Hence, we get \[\dim_\C \Tw_p =\dim_{\C}\Gr2{\C^{21}} > \dim_{\R}\Gr1{\R^{20}} =\dim_\R \Tw_{\R,p}.\]
This can also be seen from the Lie-theoretic description in \cite[beginning of Section~14.5]{FHW}, keeping in mind the description of $C_1(\Omega)$ given in Section~\ref{M_D}.
\end{rema}

If $\Delta_p\neq\emptyset$, then $\Tw_{\R,p}\subset C_1(\Omega)_{\R,p}$ has complement of real codimension one and $\Tw_p\subset C_1(\Omega)_p$ has complement of complex codimension two. Hence, $\Tw_p$ is connected while $\Tw_{\R,p}$ is not; we describe the set of these connected components in the following remark.

\begin{rema}[Connected components of $\Tw_{\R,p}$]
\label{connected components}
We let $p\in\Omega$ be a point and write $\sigma_{\R}=(\pr1,([\kappa],\pr2))$ as in the proof of \Cref{realfamily}.
By construction the map $(\kappa,\pr2)$ induces isomorphisms $\cC_{\R}\bij (K\Omega)_{=1}$ and $\rest{\cC}{\Twr}\bij (K\Omega)^{\circ}_{=1}$.
These restrict over the subspaces $C_1(\Omega)_{\R,p}$ and $\Tw_{\R,p}$ to isomorphisms to the space of classes of norm $1$ in $V_p$ or $V_p^{\circ}$, respectively. In particular, $([\kappa],\pr1)$ induces a bijection between the sets of connected components $\pi_0(\Tw_{\R,p})$ and $\pi_0(V_p^{\circ})$ --- that is, the set of Weyl chambers over $p$ (cf. also \cite[Section~16.4B, p.~235]{FHW} and the proof of \Cref{twr} above).
\end{rema}

\begin{proof}[Proof of \Cref{extensions}]
Let $\pr1,\pr2$ denote the projections from $\cC^+$ to $C_1^+(\Omega)$ and $\Omega$.
First, recall by \Cref{transverse} that for $\delta\in\Delta$ the hyperplane $H_{\delta}=\P(\delta^{\perp})\cap\Omega$ intersects every complex cycle with $C\not\subset H_{\delta}$ transversely in exactly two points.
Consequently, for every $\delta\in\Delta$ the preimage $Z_{\delta}\defeq\pr2\inv(H_{\delta})\subset\cC^+$ is a hyperplane such that for all $[C]\in C_1^+(\Omega)$ the intersection $\pr1\inv([C])\cap Z_{\delta}$ consists either of two distinct points if $C\not\subset H_{\delta}$, or of the entire fiber $\pr1\inv([C])=\{[C]\}\x C$ of $\pr1$ otherwise.
Hence the restriction $\rest{Z_\delta}{\Tw}\to\Tw$ of $\pr1$ is a $2$-sheeted covering, which must be trivial since $\Tw$ is simply connected. Now, let $\delta\in\Delta$ and $(\cF,\mu)$ be a marked twistor family with period map the inclusion of a twistor cycle $C\subset\Omega$, $[C]\in\Twr$.
Then $\mu^{-1}(\delta)$ is a $(1,1)$-class exactly on the fibers over the two distinct intersection points $C\cap H_{\delta}$.
By the construction of the twistor space (see Section~\ref{twistor}), these two fibers must necessarily carry conjugate complex structures, i.e. $\mu^{-1}(\delta)$ must be effective on exactly one of them, anti-effective on the other. 

Now, let $\sigma\in\sF(U)$ be an extension of $\sigma_{\R}$ over a connected open neighborhood $U$ of $\Twr$ in $C_1^+(\Omega)$.
By \Cref{varying partition}, this section induces for every $[C]\in U$ and $p\in C$ a partition $P_{([C],p)}$ of $\Delta_p$.
Recall that the complement of the inclusion $\Tw\subset C_1^+(\Omega)$ has complex codimension $3$, hence $\Tw$ and also $U\cap\Tw$ are connected.
Therefore, by the continuity explained in \Cref{varying partition}, $\delta$ must be effective on every point $([C],p)$ of one sheet of the covering $\rest{Z_{\delta}}{U\cap\Tw}\to U\cap\Tw$ and anti-effective on the points of the other sheet. If there was a cycle $[C]\in U\setminus\Tw$, then $C\subset H_{\delta}$ for a $\delta\in\Delta$ and thus $\rest{Z_{\delta}}U$ would be connected, contradicting the continuity discussed in \Cref{varying partition}.
Finally, to show the last claim of \Cref{extensions}, let $p$ be a point $p\in\Omega$ for which $\Tw_{\R,p}$ has at least two different connected components $K_1,K_2$.
Then we have $\Delta_p\neq\emptyset$, otherwise $\Tw_{\R,p}=C_1(\Omega)_{\R,p}$ would be connected.
By \Cref{connected components} the $[\kappa(K_i)]$ must be different Weyl chambers over $p$ and thus differ on their signs of their intersection with some $\delta\in\Delta_p$.
If $U\cap\Tw_{p}$ connected the two components $K_i$, by the continuity of $\kappa$ over $U\cap\Tw_p$ and the intersection product, there would have to be a $[C]\in U\cap\Tw_p$ with $\langle\kappa([C],p),\delta\rangle=0$, contradicting $\kappa([C],p)\in V_p^{\circ}$.
\end{proof}

\section{Complex-Riemannian metrics associated with cycles}
\label{section metrics}

Finally, based on the results obtained above, we are able to give an answer to \Cref{geomques}, which asks for geometric properties shared by K3-surfaces living in a family over a small deformation of a twistor cycle.
While \Cref{conformalStructures} below gives a satisfying solution to the posed problem, the answer leads to a number of new questions and directions. We will discuss some of these in Section~\ref{conclusion} below.

\subsection{Deformations of the twistor construction}

Assume that we are given a twistor family $\cF=(T\stackrel{f}{\to}\P^1)$ as in Section~\ref{twistor}. Recall from \Cref{deformation space} that the fibration structure survives on every small deformation of the total space $T$. 

In this section, we will show that also the remaining holomorphic data \ref{twistor}\ref{twistor line} and \ref{twistor}\ref{omega} deform with $T$. We first consider relative symplectic forms in \Cref{sympl} and then move on to twistor lines in \Cref{deform twistor lines} below.

\begin{lemm}[Relative symplectic form]
\label{sympl}
Possibly up to restricting further to a $\Gamma$- and $\tau$-stable open neighborhood $U$ of $\Twr$ in $\Tw$, the universal small deformation of K3 twistor spaces $\cT_U$ carries a unique $\Gamma$-invariant relative symplectic form $\omega$ restricting for each $[C]\in \Twr$ to the canonical relative symplectic form\footnote{Cf.~\ref{twistor}\ref{omega}}
on the twistor family $\cT_{[C]}$.
\end{lemm}
\begin{proof}
We write $F\from\cX\to\cC$ for the projection $\sigma^*\cN\to\rest{\cC}U$. To construct a $\Gamma$-invariant relative symplectic form on $\cX$ boils down to finding a $\Gamma$-invariant section $\omega$ in $F_*(\Omega^2_{\cX/\cC}\otimes F^*\O_{\cC}(2))$, cf.~\Cref{rel sympl}, where $\Gamma$ acts equivariantly on $\Omega^2_{\cX/\cC}$ via pulling back differential forms.
Here, it suffices to construct a \emph{real-analytic} relative symplectic form $\omega$ over $\rest{\cC}{\Twr}$, since then \Cref{extsection_general} applies to show that up to shrinking $U$ around $\Twr$, $\omega$ extends uniquely to a $\Gamma$-invariant (holomorphic) section of 
$F_*(\Omega^2_{\cX/\cC}\otimes F^*\O_{\cC}(2))$ over $\pr1^{-1}(U)$; by \Cref{neighborhoods} we can then further achieve $U$ to be $\Gamma$- and $\tau$-stable.

For the construction of the section over $\Twr$, recall that $\rest{\cX}{\Twr}\to\rest{\cC}{\Twr}\to\Twr$ is the real-analytic second level family of \Cref{realfamily}, all of whose fibers are marked K3 twistor families.
The first component of the inverse map $(\kappa,\pr2)\from\rest{\cC}{\Twr}\to(K\Omega)^{\circ}_{=1}$ constructed in \Cref{isom} provides $\rest{\cX}{\Twr}\to\rest{\cC}{\Twr}$ with a real-analytic section $\kappa$ of K\"ahler classes of volume one.
By Yau's solution of the Calabi conjecture, for every $([C],p)\in\rest{\cC}{\Twr}$ the class corresponding to $\kappa([C],p)$ contains a unique K\"ahler-Einstein metric $g([C],p)$, which induces a canonical relative holomorphic-symplectic form $\omega_{[C]}$ over each twistor cycle $[C]\in\Twr$. In fact, the metrics $g{([C],p)}$ depend real-analytically\footnote{This seems to be well-known to experts; the key point is to apply a version of the implicit function theorem for real-analytic maps between Banach spaces, see \cite[Theorem~IV.3.12]{MR0467421}, after identifying a partial derivative of a relative Monge-Amp\`ere-type operator as an (invertible) Laplacian.} on $([C],p)\in\rest{\cC}{\Twr}$, hence the same holds for the $\omega_{[C]}$, since these can be written down in terms of the complex structures and the metric by the explicit formula \cite[Equation~(3.87)]{HKLR}.
As the section $\kappa$ is $\Gamma$-invariant and each of the $g([C],p)$ is uniquely determined in the class corresponding to $\kappa{([C],p)}$, the group $\Gamma$ must also act equivariantly on the family of metrics on $\rest{\cC}{\Twr}$.
Consequently, $\omega$ is also $\Gamma$-invariant.
\end{proof}

As a next step, we deform twistor lines (see \ref{twistor}\ref{twistor line}) within the total space of deformations of IHS manifolds. The following result was partly inspired by the work of LeBrun on deformations of complex-Riemannian metrics in \cite[Section~6]{Leb83}. 

\begin{prop}[Deformations of twistor lines for IHS manifolds]
\label{deform twistor lines}
Let $T$ be a universal deformation over $U$ of a marked twistor family $T_0$ of an $n$-dimensional IHS manifold.
Then, the following statements hold true.
\begin{enumerate}
\item\label{family of lines}
Smooth rational curves inside $T$ with normal bundle $\O_{\P^1}(1)^{\oplus n}\oplus\O_{\P^1}^{\oplus 3n_-}$ give rise to a common deformation $L\to S$ inside $T$ of all twistor lines in $T_0$, where $S\subset\dou T$ is a smooth $(5n-3)$-dimensional open complex subspace and $L\subset\dou(T)\x T$.
\item\label{relative family}
Up to shrinking $U$, the evaluation morphism $L\to T$ is surjective, and there is a flat holomorphic submersion $p \colon S\to U$ such that over each $t\in U$ the fiber $S_t\subset\dou T_t$ is the smooth $2n$-dimensional open complex space of smooth rational curves in $T_t$ with normal bundle $\O_{\P^1}(1)^{\oplus n}$ inside $T_t$.
\item\label{involution on family}
Up to shrinking $U$, the natural antiholomorphic involution on the twistor family induces a unique equivariant antiholomorphic involution $\tau$ of $L\to S\to U$.
\item\label{K3 case}
For the marked second level family $\cT_U$ of K3 surfaces introduced in \Cref{defT}, for \ref{relative family} and \ref{involution on family} to hold it suffices to shrink $U$ around $\Twr$; in particular, $S$ contains all the twistor lines of each $\cT_t$.
The involution $\tau$ induces the complex conjugation on $U\subset C_1(\Omega)$, and moreover the natural antiholomorphic involution on every twistor family $\cT_t$. 
\end{enumerate}
\end{prop}
The following diagram summarizes the situation:
\vspace{-0.2cm}

\begin{equation}
\label{diag3}
\begin{tikzcd}[row sep=normal, column sep=normal]
L\ar[bend left=20, two heads]{rrr}{\pr2}\dar{\pr1}&L_0\dar{\pr1}\lar[hook']\rar[two heads]{\pr2}&T_0\dar{f_0}\rar[hook]&T\dar{f}\ar[bend left=50]{dd}{\pi}\\
S\ar[bend right=35]{rrrd}[near start]{p}&S_0\lar[hook']\drar{}&C_0\dar{}\rar[hook]&\rest{\cC(\pdom{\Lambda})}{U}\dar{\pr1}\\
&&\{[C_0]\}\rar[hook]&U
\end{tikzcd}
\end{equation}

The individual parts of \Cref{deform twistor lines} will be proven separately after the terminology valid for all parts of the proof has been set up first. 

\subsubsection*{Setup for the proof of \Cref{deform twistor lines}.}
Let $X=(M,I)$ be an $n$-dimensional IHS manifold with IHS lattice $\Lambda$. Its twistor family $\cF_0=(T_0\stackrel{f_0}{\to}C_0)$ discussed in Section~\ref{twistor} is a quadratic family of IHS manifolds \cite[Example~3.11]{BKS}. As $C_0\cong\P^1$ is simply connected, there is a $\Lambda$-marking of $\cF_0$, and by \Cref{twistor cycle decomposition} for every such marking the period map $\per\from C_0\inj\pdom{\Lambda}$ embeds $C_0$ as a real cycle.
Thus, we may assume that $C_0\subset\pdom{\Lambda}$ is a real cycle and $\per$ is the inclusion.
By \Cref{unobs}, we may additionally assume that the given universal deformation of $(\cF_0,\mu_0)$ is of the form \[(\cF_U,\mu)=(T\stackrel{f}{\to}\rest{\cC(\pdom{\Lambda})}U\stackrel{\pr1}{\to}U,\mu)\] over an open neighborhood $U\subset C_1(\pdom{\Lambda})$ of $[C_0]$.

\subsubsection*{Proof of \ref{family of lines}.}
We consider the set $S\subset\dou(T)$ of smooth rational curves $l\subset T$ with normal bundle $\O_l(1)^{\oplus n}\oplus\O_l^{\oplus 3n_-}$ in $T$, i.e., $\phi^*\Nb{l}{T_0}\cong\O_{\P^1}(1)^{\oplus n}\oplus\O_{\P^1}^{\oplus 3n_-}\eqdef\sN$ for a biholomorphism $\phi\from\P^1\to l$.
This is in fact an open condition on $[l]\in\dou(T)$. Indeed, recall that $\P^1$ is rigid and that by \cite[Theorem~3.5.3.2 and Section~3.8.1.2]{VB10} for the holomorphic vector bundle $\sN$ on $\P^1$ to be stable under small deformations it suffices to notice that
\[\hdim1{\P^1}{\End(\sN)}\defeq\hdim1{\P^1}{\sN\otimes\sN^{\vee}}=\hdim0{\P^1}{\O_{\P^1}(-2)\otimes\sN\otimes\sN^{\vee}} = 0.\]

Next, we check that $S$ indeed contains all twistor lines of $T_0$.
For this, we denote for each point $m\in M$ the corresponding twistor line by $l_m\defeq\{m\}\x\P^1\subset T_0$, which has normal bundle $\Nb{l_m}{T_0} \cong \O_{l_m}(1)^{\oplus n}$ in $T_0$.
In order to compare the normal bundles of $l_m$ inside $T_0$ and $T$, we consider the exact sequence \begin{equation}\label{eq:normalsequence}
0\to\Nb{l_m}{T_0}\to\Nb{l_m}{T}\to\rest{\Nb{T_0}{T}}{l_m}\to0.
\end{equation}
Note that we have isomorphisms \[\rest{\Nb{T_0}{T}}{l_m}\cong\rest{\O_{T_0}^{\oplus3n_-}}{l_m}\cong\O_{l_m}^{\oplus3n_-}.\] Therefore, as $\hdim1{l_m}{(\O_{l_m}^{\oplus3n_-})^{\vee}\otimes \O_{l_m}(1)^{\oplus n}} = 0$, the exact sequence \eqref{eq:normalsequence} splits; i.e., $\Nb{l_m}{T}\cong\O_{l_m}(1)^{\oplus n}\oplus\O_{l_m}^{\oplus3n_-}$, and thus $[l_m]\in S$, as claimed.

With a view towards the remaining claims, as we have $\hdim0{l}{\Nb{l}{T}}=5n-3$ and $\hdim1{l}{\Nb{l}{T}}=0$ for all $[l]\in S$, it follows from \cite[Theorem~1]{Kod62} that $S$ is indeed smooth of dimension $5n-3$.
The restriction of the universal family (see Section~\ref{douady section}) to the open subset $S$ can then be considered as a common deformation $g\from L\to S$ of all twistor lines $l_m$ inside $T$, where $L\subset\dou T\x T$. This concludes the proof of \ref{family of lines}.

\subsubsection*{Proof of \ref{relative family}.}
Due to the structure of their normal bundle, the twistor lines $l_m$ are \emph{free rational curves} in $T$, i.e., their deformations in $T$ cover neighborhoods of each $l_m$ in $T$, see for example~\cite[Corollary~II.3.5.4]{KollarRationalCurves}.
Hence, as the twistor lines $l_m$ cover $T_0$, the subspace $\pr2(L)\subset T$ must contain an open neighborhood of $T_0$ in $T$.
As $\pi$ is proper, by shrinking $U$ around $[C_0]$ we may assume that $\pr2(L)=T$ and that $U$ is Stein.
Then, $\pi$ must map each compact complex subspace of $T$ to a point of $U$ and thus the canonical embedding $\dou(T/U)\inj\dou(T)$ (cf.~Section~\ref{douady section}) is a biholomorphism. In particular, $\pi$ maps each $l\subset T$ with $[l]\in S$ to a point of $U$.
That is, $l$ is contained in one of the fibers of $\pi$ and hence of the projection $L\subset\dou T\x T\stackrel{\pr2}{\to}T\stackrel{\pi}{\to}U$.
Then, by the universal property of the relative Douady space $\dou(T/U)$ (see again Section~\ref{douady section}), there is a unique holomorphic map $p\from S\to U$ such $p\circ \pr1=\pi\circ\pr2\from L\to U$.

Now, for $t\in U$ let $l$ be a smooth rational curve contained in the fiber $T_t$ of $\pi$.
Using the exact sequence $0\to\Nb{l}{T_t}\to\Nb{l}{T}\to\rest{\Nb{T_t}{T}}{l}\to0$, cf.~\eqref{eq:normalsequence}, we see that the conditions $\Nb{l}{T_t}\cong\O_{l}(1)^{\oplus n}$ and $\Nb{l}{T}\cong\O_{l}(1)^{\oplus n}\oplus\O_{l}^{\oplus3n_-}$ on $l$ are equivalent.
In particular, the fiber $S_t$ of $p$ is via the canonical closed embedding $\dou(T_t)\to\dou(T/U)$ isomorphic to the space of smooth rational curves $l\subset T_t$ with normal bundle $\O_{l}(1)^{\oplus n}$ in $T_t$.
Note that the vector bundle $\O_{\P^1}(1)^{\oplus n}$ on $\P^1$ is stable under small deformations due to
\[\hdim1{\P^1}{\End(\O_{\P^1}(1)^{\oplus n})}=\hdim0{\P^1}{\O_{\P^1}(-2)^{\oplus n^2}}=0.\]
Therefore, having normal bundle $\O_{l}(1)^{\oplus n}$ in $T_t$ is also an open condition on $[l]\in\dou T_t$.
It follows, again with Kodaira's result, that $S_t$ is a $2n$-dimensional smooth open complex subspace of $\dou(T_t)$.
Consequently, the maps $p\from S\to U$ and $\pr2\from L\to T$ are flat submersions, see \cite[Chapter~II, Theorem~1.14]{SCV7}.

\subsubsection*{Proof of \ref{involution on family}.}
The proof of this item has similarities with the argument used to establish \Cref{douinvolution}. We consider the antiholomorphic involutions $\tau$ on $\pdom{\Lambda}$ and $C_1(\pdom{\Lambda})$ induced by the natural complex conjugation on $\Lambda_{\C}$ (cf.~Sections~\ref{subsect:cyclespaces} and \ref{realcycles}).
These induce an antiholomorphic involution $\tau$ on $\cC(\pdom{\Lambda})\subset C_1(\pdom{\Lambda})\x\pdom{\Lambda}$ via the actions on the two factors.
By shrinking $U$ to $U\cap\tau(U)$, we may assume that $U$ is stable under $\tau$.
Then $U^{\tau}=U\cap C_1(\pdom{\Lambda})_{\R}\subset\Twr$.

We define $\overline{T}$ and $\rest{\overline{\cC(\pdom{\Lambda})}}U$ as $T$ and $\rest{\cC(\pdom{\Lambda})}U$ equipped with the respective conjugate complex structures.
Then $\overline{T}\stackrel{f}{\to}\rest{\overline{\cC(\pdom{\Lambda})}}U\stackrel{\tau\circ\pr1}{\lto}U$ defines a second level family $\overline{\cF}_U$ of IHS manifolds. If we equip this second level family with the existing marking $\mu$, then the marked second level family $(\overline{\cF}_U,\mu)$ has period maps $\tau\circ\pr2\from\rest{\overline{\cC(\pdom{\Lambda})}}U\to\pdom{\Lambda}$ and the inclusion $U\inj C_1(\pdom{\Lambda})$.

By the defining properties of twistor families (cf.~Section~\ref{twistor}), there is an antiholomorphic involution $\tau_0$ of $\cF_0=(T_0\to C_0)$ over the antipodal map on $C_0\cong\P^1$.
Via $\tau_0$, we can consider $(\overline{\cF}_U,\mu)$ as a deformation of $(\cF_0,\mu_0)$.
Since $(\cF_U,\mu)$ is a universal deformation, up to shrinking $U$, this induces a universal morphism $(\overline{\cF}_U,\mu)\to(\cF_U,\mu)$.
By functoriality of the period maps, this must induce the identity on $U$ and $\tau$ on $\rest{\overline{\cC(\pdom{\Lambda})}}U$.
Moreover, $(\overline{T}\stackrel{f}{\to}\rest{\overline{\cC(\pdom{\Lambda})}}U)\to(T\stackrel{f}{\to}\rest{\cC(\pdom{\Lambda})}U)$ restricts to $\tau_0$ on the central fiber.
In other words, we have constructed an antiholomorphic involution of $T\stackrel{f}{\to}\rest{\cC(\pdom{\Lambda})}U$ over $\tau$ (which we also name $\tau$) that extends $\tau_0$ on the central fiber.

By \Cref{douinvolution}, this $\tau$ induces an antiholomorphic involution $\tau$ of the Douady space $\dou(T)$.
It preserves the subspace $S\subset\dou(T)$ and its fibration structure over $U$ because the involution $\tau$ on $T$ preserves the defining properties of the lines $[l]\in S$.
The uniqueness of $\tau$ follows from the uniqueness of the universal morphism $(\overline{\cF}_U,\mu)\to(\cF_U,\mu)$.

\subsubsection*{Proof of \ref{K3 case}.}
Let $\cT_U=(T\to\rest{\cC}U\to U)$ be the marked second level family of K3 surfaces defined in \Cref{defT} over an open neighborhood $U$ of $\Twr$ in $\Tw$.
By construction, $\cT_{t_0}$ is a marked K3 twistor family for every given $t_0\in\Twr$.

In an open neighborhood $U_{t_0}$ of $t_0$, we can proceed as in parts \ref{relative family} and \ref{involution on family} above to construct $L\to S\to U_{t_0}$ with equivariant antiholomorphic involution $\tau$.
Recalling the argument used to prove part~\ref{relative family}, the families $L_t\to S_t$ for $t\in U_{t_0}$ were constructed using universal properties of relative Douady spaces and hence do not depend on the $t_0\in\Twr$ we started with. A similar comment can be made regarding the construction carried out in the proof of part \ref{involution on family}. Therefore, the constructions of $L\to S\to U_{t_0}$ and their antiholomorphic involutions for varying $t_0$ glue together over an open neighborhood $U$ of $\Twr$ in $\Tw$.
For the construction of the involution $\tau$ we can shrink $U$ further around $\Twr$ to make it stable under the existing antiholomorphic involution on $C_1(\Omega)$.

As we are in the K3 surface case, using \Cref{univfamily} instead of just the universal property of universal deformation spaces, the universal morphism $(\overline{\cF}_U,\mu)\to(\cF_U,\mu)$ constructed in the proof of \ref{involution on family} above exists globally without having to shrink $U$ any further.
Moreover, by \Cref{unique isomorphism} this morphism is also independent of the chosen twistor family $\cT_{t_0}$ we started with in part~\ref{involution on family}.
As $\tau$ induces the natural complex conjugation on $U\subset C_1(\Omega)$, it fixes $U\cap\Twr$ and thus induces an antiholomorphic involution on every twistor family $\cT_t$ for $t\in\Twr$, which on $\cT_{t_0}$ coincides with the natural antiholomorphic involution.
As the construction of $\tau$ does not depend on the twistor family $\cT_{t_0}$ we started with, the involution $\tau$ must induce the natural antiholomorphic involution on all of them.

\medskip 
This concludes the proof of \Cref{deform twistor lines}.
\qed

\smallskip

\begin{rema}[Deforming the set of twistor lines over $U$]\label{relative_cycle_space_is_trivial}
We think about the holomorphic submersion $p \colon S\to U$ as ``deforming'' the manifold $S_0$ of twistor lines of $T_0$ over the basis $U$. However, recall from \cite[Corollary~4.6]{Cam91} that $S_0$ is never compact for any IHS manifold. Hence, $p$ is not proper and thus not a deformation in our terminology, cf.~Section~\ref{deformation sec}.
On the other hand, it follows from a result of Andreotti--Vesentini that $p$ is essentially locally trivial in the following precise sense: Mapping each point $m\in M$ to the corresponding twistor line $[l_m] \defeq [\{m\}\x\P^1 ] \in S_0$ defines a totally real embedding $\imath \colon M \hookrightarrow S_0$ of $M$ into the complex manifold $S_0$.
Therefore, there exists a neighborhood basis of Stein open subsets $ W_\epsilon$ of $\imath(M)$ inside $S_0$. It then follows from \cite{AV62} that for any relatively compact $W_0 \subset W_\epsilon$ neighborhood of $\imath(M)$ in $W_\epsilon$ there exists an open subset $W \subset S$ with $W \cap S_0 = W_0$ and such that $\rest{p}{W}\colon W \to p(W) \subset U$ is biholomorphic over $p(W)$ to the product $p(W) \x W_0$ of the open subset $p(W)$ with $W_0$. 
\end{rema}

\subsection{The question of Fels--Huckleberry--Wolf and Looijenga}\label{Answer_to_question}

Using the precise information gathered in \Cref{deform twistor lines} above, we are finally in the position to attack \Cref{geomques}, at least for cycles contained in the neighborhood $U$ of $\Twr$ constructed in part \ref{K3 case} of \Cref{deform twistor lines}. 

\begin{theo}[Complex twistor cycles and Ricci-flat ASD holomorphic conformal structures] 
\label{conformalStructures}
Let $\cT_U=(T\to\rest{\cC}U\to U)$ be the marked second level family of K3 surfaces introduced in \Cref{defT}, which is defined over an open subspace $U\subset C_1(\Omega)$, and let $M$ be the real $4$-manifold underlying every K3 surface. Then, up to shrinking $U\subset C_1(\Omega)$ around the space of twistor cycles, there is a canonical $\Gamma$-equivariant family of anti-selfdual holomorphic conformal structures on the submersion $p\colon S \to U$ constructed in \Cref{deform twistor lines} such that for every $t = [C] \in U$ the induced conformal equivalence class of metrics on the fiber $S_t$ contains a complex-hyperk\"ahler representative that induces exactly those complex structures on $M$ that correspond to the points $p \in C \subset \Omega$ via the Global Torelli Theorem. Over every twistor cycle this complex-hyperk\"ahler metric is the complexification of the K\"ahler-Einstein metric coming from Yau's solution of the Calabi conjecture.
\end{theo}
The claim of \Cref{conformalStructures} concerning the complex structures represented by the points in $C\subset \Omega$ is to be understood in the following way: By \Cref{relative_cycle_space_is_trivial}, every fiber $S_t$ of $p$ is a complexification of $M$ and  admits a totally real embedding of $M$. Now, as a complex-hyperk\"ahler manifold $S_t$ comes equipped with a family of integrable distributions $Q_\zeta \subset \sT_{S_t}$ parametrized by $\zeta \in \mathbb{P}^1$, see for example \cite[Def.~4.1--4.2]{MR4477203}. Restricting these distributions to $T^\mathbb{C}M$ yields integrable almost complex structures on $M$. The claim is that these complex structures on $M$ are exactly the ones corresponding to the points in $C$ via the Global Torelli Theorem (or in fact just the Weak Torelli Theorem as formulated in \cite[Chapter VIII, Corollary 11.2]{BHPV}).

\begin{proof}[Proof of \Cref{conformalStructures}]
Let $\cT_U$ be the marked second level family of K3 surfaces defined in \Cref{defT} with relative symplectic form $\omega$ on $T$. We shrink $U$ around $\Twr$, such that over $U$ the family $L\to S\stackrel{p}{\to} U$ (of parameter spaces) of deformed twistor lines with equivariant antiholomorphic involution $\tau$ obtained as part of \Cref{deform twistor lines} exists.

Note that $U$ is a complexification of $U^{\tau}=\Twr$, and that $S$ is a complexification of $S^{\tau}$. As $\tau$ is equivariant, $S^{\tau}=(\rest{S}{\Twr})^{\tau}$ lies over $\Twr$ with fibers $(S^{\tau})_{t_0}=(S_{t_0})^{\tau}$ for all ${t_0}\in\Twr$. We now apply the reconstruction theorem \cite[Theorem~3.3]{HKLR}, recalled also in \Cref{InversePenrose} above: By the construction in \Cref{deform twistor lines}\ref{relative family}, for every ${t_0}\in\Twr$ the fiber $(S_{t_0})^{\tau}$ is the space of real twistor lines of the twistor family $\cT_{t_0}$ and thus $(S_{t_0})^{\tau}$ is real-analytically isomorphic to the real-analytic four-manifold $M$ underlying the twistor construction of $\cT_{t_0}$.
In other words, fixing a diffeomorphism $T_{t_0}\isom M\x\P^1$, the evaluation map $L\to T$ maps the fibers of $\rest{(L\to S)}{S_{t_0}^{\tau}}$ bijectively to the lines $\{m\}\x\P^1\subset M\x\P^1$.
Hence $\rest{L}{S^{\tau}}\to\rest{T}{\Twr}$ is a real-analytic isomorphism and, due to the real-analytic Ehresmann Theorem, the restriction $S^{\tau}\to\Twr$ is a real-analytic $M$-fiber bundle. As a consequence of this,\footnote{Compare also with \Cref{relative_cycle_space_is_trivial}.} for every ${t_0}\in\Twr$ we can find an open neighborhood $U_{t_0}$ of $t_0$ and a complexification $M^{\C}$ of $M$, such that $\rest{S}{U_{t_0}}$ contains a $\tau$-stable open neighborhood of $\rest{S^{\tau}}U$ that is over $U_{t_0}$ biholomorphic to the trivial bundle $M^{\C}\x U_{t_0}\to U_{t_0}$. 

Next, we claim that the family $\cC^+(\Omega) \to C_1^+(\Omega)$ is a holomorphically trivial $\P^1$-bundle on $C_1^+(\Omega)$.
Indeed, notice that a priori, say by the Grauert-Fischer Theorem, $\cC(\Omega)\to C_1(\Omega)$ is a holomorphic fiber bundle with fiber $\P^1$ and structure group $\PSL(2, \C)$, the group of holomorphic automorphisms of $\P^1$.
In particular, the structure group is a connected complex Lie group. As $\G^+$ is contractible, the fiber bundle $\cC^+(\Omega)$ is topologically trivial, e.g.~by \cite[Proposition~7.5]{Wed16} or \cite[Corollary~11.6]{Ste51}.
So, based on our observation regarding the structure group and on the fact that $\G^+$ is Stein, the Oka principle may be applied to conclude that $\cC^+(\Omega)\to C_1^+(\Omega)$ is even holomorphically trivial \cite[Satz~6]{Grau58}.

Hence, we may trivialize the holomorphic $\mathbb{P}^1$-bundle $\pr1\colon \cC\defeq\cC^+(\Omega) \to C_1^+(\Omega)$; this leads to a projection $q\colon \cC \to \mathbb{P}^1$. In particular, we obtain an induced isomorphism $\O_{\cC}(1) \cong q^*\O_{\mathbb{P}^1}(1)$. This in turn induces a trivialization of 
\begin{equation}\label{pushforward_trivialisation}
(\pr1)_*(\O_{\cC}(1)) \cong \Hsh0{\mathbb{P}^1}{\O_{\mathbb{P}^1}(1)} \otimes \O_{C_1^+(\Omega)} \cong \mathbb{C}^2 \otimes \O_{C_1^+(\Omega)}.
\end{equation}
We fix these isomorphisms, restrict them to $U \subset C_1^+(\Omega)$, and endow $\mathbb{C}^2$ with the standard complex-symplectic structure. 

We will now carry out the reverse Penrose construction relative over $U$. Looking at Diagram~\eqref{diag3} and following the fiberwise arguments in \cite[Section~3]{MR3210602}, cf.~also \cite[Section~2.9]{MR1306960}, we see that the relative tangent bundle $\sT_{S/U}$ can be $\Gamma$-equivariantly decomposed as a tensor product
\begin{align*}\sT_{S/U} &\cong \left[(\pr1)_* ((\pr2)^*(\sT_{T/\cC} \otimes f^*\O_{\cC}(-1)) )\right] \otimes \left[(\pr1)_* (f^*\O_{\cC}(1)) \right]\\ &\eqdef \sS \otimes \tilde{\sS}.\end{align*}
As a consequence, the tensor product $\mathrm{Alt}^2 \sS \otimes \mathrm{Alt}^2 \tilde{\sS}$ embeds $\Gamma$-equivariantly as a sub line bundle into $\mathrm{Sym}^2(\sT_{S/U})$, defining a $\Gamma$-equivariant family of holomorphic conformal structures on the submersion $p\colon S \to U$, cf.~\cite[Chapter~I, \S3.3]{MR1632008}. 

In our situation, the relative symplectic form $\omega \in \Hsh0{\cC}{f_*\Omega^2_{T/\cC}\otimes f^*\O_{\cC}(2)}$ canonically endows $\sS$ with a $\Gamma$-invariant $p$-fiberwise holomorphic-symplectic structure, while the second factor $\tilde{\sS}$ inherits a $p$-fiberwise holomorphic-symplectic (but not necessarily $\Gamma$-equivariant) structure from the fixed decomposition \eqref{pushforward_trivialisation} and the standard symplectic form on $\mathbb{C}^2$. Combining the two, we obtain a $p$-fiberwise non-degenerate symmetric holomorphic $2$-tensor on $\sT_{S/U}$, i.e., a (not necessarily $\Gamma$-equivariant) family of holomorphic Riemannian metrics on the submersion $p\colon S \to U$ inside the given $\Gamma$-equivariant conformal classes. As fiberwise we just executed the Non-linear Graviton Construction, \Cref{InversePenrose} implies that the family of holomorphic Riemannian metrics just exhibited defines a complex-hyperk\"ahler metric on each fiber that for all $t\in\Twr$ coincides with the complexification of a real K\"ahler-Einstein metric. In particular, all conformal structures defined by $\mathrm{Alt}^2 \sS \otimes \mathrm{Alt}^2 \tilde{\sS}$ are anti-selfdual (which is a property of the conformal class of metrics) and admit a Ricci-flat representative.

Regarding the last claim, we note that after possibly shrinking $S_t$ around $M$, the topological quotient of $S_t \times \mathbb{P}^1$ (or, in invariant formulation, the projectivization of the trivial vector bundle $\mathbb{P}(\tilde{\sS}|_{S_t})$) by the twistor foliation defined for example on \cite[p.~14]{MR4477203} is $M \times \mathbb{P}^1$. From data given to us by the complex-hyperk\"ahler metric, we obtain all the structures on $M \times \mathbb{P}^1$ that make it into a complex twistor space, which by the correspondence between complex-hyperk\"ahler manifolds and complex twistor spaces, see e.g.~\cite[Theorem~10.5.5]{Dun10}, is isomorphic to the complex twistor space we started with; the claim follows. The proof of \Cref{conformalStructures} is therefore complete. 
\end{proof}

\begin{rema}[The complex-hyperk\"ahler metric is almost canonical]
 As the proof of \Cref{conformalStructures}  shows, the complex-hyperk\"ahler metric in the conformal class emerging from the construction is canonical up to scaling by a single non-zero complex number.
\end{rema}

\begin{coro}[Answer to \Cref{geomques}]\label{cor:commonHKmetric}
 Let $[C] \in U$ correspond to a small deformation of a twistor cycle. Then, the fibers of the marked embedded family $\cT_{[C]}=(T_{[C]}\to C)$ of K3 surfaces share a common complex-hyperk\"ahler metric.  
\end{coro}
\begin{proof} Let $C$ be a small deformation of the twistor cycle $C_0\subset \Omega$. I.e., setting $t= [C]$ and $t_0 = {[C_0]}$, in the notation of the second paragraph of the proof of \Cref{conformalStructures} above,  assume $t \in U_{t_0}$. As explained in \emph{loc.~cit.} (and also in \Cref{relative_cycle_space_is_trivial}) we know that the fiber $S_{t}$ of $p$ over $t \in U$ is a complexification of $M$; more precisely, composing the natural totally real embedding of $M$ into $S_{t_0}$ with the inverse of the local trivialization of $S$ near $t_0$, we obtain a totally real embedding of $M$ into $S_{t}$. Restricting the complex-hyperk\"ahler metric on $S_{t}$ obtained in \Cref{conformalStructures} to $M$ yields the desired complex-hyperk\"ahler metric on $M$. 
\end{proof}
As by \Cref{univfamily} any small deformation of any twistor family occurs as a fiber of $\cT_U$, \Cref{cor:commonHKmetric} applies to any such family and thus provided an answer to \Cref{geomques}.


\section{Open questions and future directions}
\label{conclusion}

While \Cref{conformalStructures} and \Cref{cor:commonHKmetric} answer \Cref{geomques} from the introduction, a lot of interesting problems remain, some of which we would like to discuss here. 

\subsubsection*{Natural open subsets of the cycle space}
First, it is curious that the universal small deformation $\cT_U$ exists on a small neighborhood $U$ of $\Twr$ but cannot be extended to the whole of $\Tw$ by \Cref{extensions}.
It is therefore a natural question to determine at least one natural maximal open subset to which this family extends.

In a different but related direction, while we restricted the discussion in this paper to the open neighborhood $C_1^+(\Omega) \supset \Tw$ of $\Twr$ inside $C_1(\Omega)$, once it is possible to move away further from the real locus $C_1(\Omega)_{\R} \supset \Twr$ into the cycle space, it will also become interesting to investigate the function-theoretic properties of the inclusion of the Hermitian-symmetric space $C_1^+(\Omega)$ into the connected component of the cycle space containing the real cycles in more detail. This is especially interesting in the light of the possibility of using the Penrose transform to pass between the period domain and (subsets of) its cycle space, cf.~\cite[Chapter 14]{FHW}.

\subsubsection*{Moduli of (complex-)hyperk\"ahler metrics}
The quotient $\Twr/\Gamma$ can be identified with a moduli space of K\"ahler-Einstein metrics on the four-manifold $M$ underlying all K3 surfaces, see for example \cite[Section~12.113]{Besse} or \cite{Kob90}. Since the action of $\Gamma$ on the cycle space, and hence in particular on the neighborhood $U$ of $\Twr$, is proper, one might ask whether a similar moduli-theoretic interpretation exists for the complex space $U/\Gamma$. For this, one has to investigate whether the determinant of the bundle $(\pr1)_*(\O_{\cC}(1)) $ of \eqref{pushforward_trivialisation} can be $\Gamma$-equivariantly trivialized over $U$ (this would yield a canonical metric on the submersion $p$). Moreover, it is necessary to establish positive results towards the expectation expressed in \Cref{singleDiffeomo} in order to find out which equivalence relation is induced by the $\Gamma$-action.

\subsubsection*{Higher dimensions}
Of course it would be interesting to see in which form and to which extend the results obtained in Section~\ref{Answer_to_question} generalize from K3 surfaces to higher-dimensional IHS manifolds.

\subsubsection*{Degenerate twistor families and singular IHS-varieties}

Points in $C_1(\Omega)_{\R}\setminus \Twr$ are characterized by the phenomenon that $\kappa$ degenerates. 
Besse asked in \cite[Section~12.K]{Besse} what happens geometrically under these degenerations of real twistor cycles. The question has a known answer involving families of orbifold K\"ahler-Einstein metrics on K3 surfaces with ADE singularities due to Kobayashi and Todorov \cite{MR902574}, see also \cite{Kob90}. As a consequence, these ``holes'' can be filled in with \emph{generalized K3-surfaces} obtained by blowing down $(-2)$-curves that intersect the corresponding $\kappa$ trivially. The problem on the level of markings discussed in Section~\ref{section_disucssion_univ_family} above appears in a different guise in the definition of marking for a holomorphic family of such generalized K3-surfaces, see \cite[p.~278ff]{Kob90}. This indicates that considering families also involving such generalized K3-surfaces with the appropriate markings, the family over $\Twr $ can indeed by extended to $C_1(\Omega)_{\R}$. On the other hand, the connection to the Burns-Rapoport space
would at least partially be lost, so that it is interesting to see whether the sheaf-theoretic approach taken by us can be modified to prove such a result. Once this is in place, it seems feasible to ``fill in the holes'' in a moduli space of hyperk\"ahler metrics; i.e., the coarse moduli space of (orbifold) K\"ahler-Einstein metrics should have a modular complexification parametrizing both smooth and orbifold complex-hyperk\"ahler metrics. 

\subsubsection*{Complexified K\"ahler classes}
Focusing on the K3 case, we have analyzed the space of cycles $C_1(\Omega)_p$ through a given point $p$ in the K3 period domain in \Cref{cycles_through_point}. While real cycles through $p$ come from K\"ahler classes on the K3 surface $X$ corresponding to $p$, some (but certainly not all) complex cycles through $p$ arise from complex perturbations of K\"ahler classes, i.e., classes of the form $\kappa + i \beta$ for $\kappa$ a K\"ahler class on $X$ and $\beta \in \Hdg11X$. In the literature, special representatives of such cohomology classes have been discussed for some time, see the discussion in \cite[Section 1]{MR4767634}. It is a natural question how these compare with the complex-hyperk\"ahler metrics found by us above. 

\vspace{0.2cm}

\bibliographystyle{alpha}
\bibliography{cycles}

\vfill
\end{document}